\documentclass[12pt,reqno,a4paper]{amsart}
\usepackage[margin=1in]{geometry}
\usepackage{amsmath,amssymb,amsthm,mathtools}
\usepackage{enumitem}
\usepackage{xcolor}
 \usepackage[colorlinks=true, linkcolor=blue, citecolor=blue]{hyperref} 
\usepackage{orcidlink}
\usepackage{etoolbox}
\usepackage{todonotes}

\newcommand{\orcid}[1]{\,{\large\orcidlink{#1}}}

\newcommand{\AffiliationsBlock}{%
  \vspace{0.8em}
  \begingroup
  \centering
  \footnotesize\itshape
  $^{a}$ Department of Mathematics, Indian Institute of Technology Madras, \\
  Chennai-600036, Tamil Nadu, India.\par
  \vspace{0.4em}
  $^{b}$ Theoretical Statistics and Mathematics Division, Indian Statistical Institute, \\
  Kolkata-700108, West Bengal, India.\par
  \endgroup
  \vspace{-1.3em}
}
\makeatletter
\patchcmd{\@maketitle}{
  \ifx\@empty\@dedicatory
  \else
}{
  \AffiliationsBlock
  \vspace{1.5em}
  \ifx\@empty\@dedicatory
  \else
}{}{\message{Patch failed!}}
\makeatother

\newcommand{\Z}{\mathbb{Z}}
\newcommand{\N}{\mathbb{N}}
\newcommand{\Q}{\mathbb{Q}}
\newcommand{\C}{\mathbb{C}}
\newcommand{\F}{\mathbb{F}}
\newcommand{\Pe}{\mathcal{P}}
\newcommand{\K}{\mathcal{K}}
\newcommand{\W}{\mathcal{W}}
\newcommand{\m}{\mathfrak{m}}
\newcommand{\mn}{\mathfrak{n}}
\newcommand{\clo}[1]{\overline{\{#1\}}}
\newcommand{\pow}[1]{[\![#1]\!]}
\DeclareMathOperator{\Max}{Max}
\DeclareMathOperator{\Fin}{Fin}
\DeclareMathOperator{\supp}{supp}
\DeclareMathOperator{\jac}{jac}
\DeclareMathOperator{\Cl}{Cl}

\makeatletter
\newcommand*{\namedlabel}[2]{\begingroup\phantomsection\def\@currentlabel{#2}\label{#1}\endgroup}
\makeatother

\theoremstyle{plain}
\newtheorem{theorem}{Theorem}[section]
\newtheorem{proposition}[theorem]{Proposition}
\newtheorem{lemma}[theorem]{Lemma}
\newtheorem{corollary}[theorem]{Corollary}
\newtheorem{fact}[theorem]{Fact}

\newtheorem{problem}[theorem]{Problem}
\newtheorem*{macproblem}{Problem}

\theoremstyle{definition}

\newtheorem{example}[theorem]{Example}
\newtheorem{remark}[theorem]{Remark}
\newtheorem*{acknow}{\textup{Acknowledgments}}
\begin{document}
\title[Prime Density and Homeomorphisms between Mac\'ias Spaces]{Prime Density, Kolmogorov Quotients\\ and Homeomorphisms between\\ Mac\'ias Spaces of Integral Domains}
\author[S. Mandal]{{Souvik Mandal$^{\,a,*}$}\orcid{0009-0004-0429-3063}}
\author[A. Sarkar]{{Ankur Sarkar$^{\,b}$}\orcid{0009-0002-3320-2568}}
\thanks{\hspace{-1.1em}* \textit{Corresponding author.}}
\thanks{\makebox[8.5em][l]{\textit{E-mail addresses:}}%
  \parbox[t]{\dimexpr\textwidth-10em\relax}{\texttt{ma22d014@smail.iitm.ac.in}, \texttt{ssouvik.xyz@gmail.com} (S.~Mandal);\\
  \texttt{ankurimsc@gmail.com} (A.~Sarkar).}\spacefactor3000 }
\begin{abstract}
For an integral domain $R$ that is not a field, the Mac\'ias space $M(R)$ is the topological space whose underlying set $R^0=R\setminus\{0\}$ is endowed with the topology generated by the basis $\{\sigma_k^0:k\in R^0\}$, where $\sigma_k^0=\{x\in R^0:\langle x\rangle+\langle k\rangle=R\}$. We characterize the openness of the unit group in $M(R)$ via semiprimitivity and, for Dedekind domains and unique factorization domains (UFDs), relate semiprimitivity and the infinitude of prime elements to the density of irreducible and prime elements. We identify the Kolmogorov quotient of $M(R)$ with a space of sets of maximal ideals and reduce the homeomorphism problem to these quotients and the cardinalities of the fibers of the quotient maps; in particular, local domains of equal cardinality have homeomorphic Mac\'ias spaces. For Dedekind domains, we solve this problem in terms of the cardinalities of their unit groups and relations among the ideal classes of their maximal ideals, thereby answering the homeomorphism problem of Mac\'ias (2024) for countable semiprimitive ones. When the class groups are torsion, as they are for principal ideal domains, the cardinalities of the unit groups and of the sets of maximal ideals suffice. For UFDs, we characterize homeomorphic Mac\'ias spaces via comaximality with prime elements and, under an additional hypothesis, obtain a solution analogous to that for Dedekind domains. We solve the problem for several classes of polynomial rings over discrete valuation rings, UFDs and fields. Finally, we characterize additive homeomorphisms between Mac\'ias spaces of certain domains of characteristic zero via ring isomorphisms, partially answering a problem of Mac\'ias and Ortiz (2025).
\end{abstract}
\maketitle
\vspace{-0.5em}
\begin{center}
\begin{minipage}{0.845\textwidth}
    \footnotesize
    \begin{list}{}{
        \leftmargin=5.5em
        \labelwidth=5.5em
        \labelsep=0pt \parsep=0pt \topsep=0pt \itemsep=0pt
    }
      \item[\textit{Keywords:}\hfill] {\raggedright Mac\'ias topology, Kolmogorov quotient, Homeomorphism problem, Semiprimitivity, Infinitude of primes, Dedekind domain, Ideal class group, Unique factorization domain. \par}
    \end{list}

    \vspace{4pt}

    \begin{list}{}{
        \leftmargin=18.5em
        \labelwidth=18.5em
        \labelsep=0pt \parsep=0pt \topsep=0pt \itemsep=0pt
    }
                \item[2020\hspace{1mm}\textit{Mathematics Subject Classification:}\hfill]
{\raggedright Primary 54A05;\\ Secondary 11A41, 13A05, 13A15, 13B10, 13C20, 13F05, 13F10, 13F15, 13F20, 13G05, 54B15, 54D10.\par}
    \end{list}
\end{minipage}
\end{center}
\vspace{1em}
\section{Introduction}\label{sec:intro}
The intersection of topology and number theory has a rich history; one of its best-known results is Furstenberg's 1955 topological proof of the infinitude of primes \cite{Furstenberg}. This proof uses the topology on $\Z$ generated by the arithmetic progressions, which is metrizable and in which every arithmetic progression is both open and closed. In contrast, Brown \cite{Brown} introduced, and Golomb \cite{Golomb} studied, a connected Hausdorff topology on $\N$ generated by the progressions $\{a+bn:n\geq 0\}$ with coprime $a,b\in\N$. This topology is coarser than the one induced on $\N$ by Furstenberg's topology, and Golomb~\cite[Theorem~6]{Golomb} observed that Dirichlet's theorem on primes in arithmetic progressions~\cite[Theorem~2.1]{diri} is equivalent to the density of the set of primes in it. Golomb topologies on arbitrary integral domains were introduced by Knopfmacher and Porubsk\'y \cite{KnopfmacherPorubsky} and studied further in \cite{Clark,CLLP}. Several rigidity results are known for Golomb spaces; for instance, the only self-homeomorphisms of the Golomb space of $\Z$ are the identity and multiplication by $-1$ \cite[Theorem~7.7]{Spirito}, while the Golomb space of $\N$ admits no self-homeomorphism other than the identity \cite[Theorem~1]{BST}.

Mac\'ias \cite{MaciasIntegers} provided another topological proof of the infinitude of primes by using a topology on $\N$ generated by the sets of natural numbers coprime to a fixed natural number, which is known as the \emph{Mac\'ias topology} and is coarser than the Golomb topology \cite[Remark~1]{MaciasIntegers}. In \cite{MaciasDomains}, he extended the notion of the Mac\'ias topology to integral domains that are not fields, explored various properties of the resulting Mac\'ias spaces, and obtained a topological proof of the infinitude of prime elements in a principal ideal domain (PID) whose unit group is finite. Mac\'ias and Ortiz \cite{MaciasOrtiz} studied closures in Mac\'ias spaces and proved topologically that a domain has infinitely many maximal ideals if its unit group is not open in its Mac\'ias space, and infinitely many pairwise non-associated irreducible elements if, in addition, every non-zero non-unit has an irreducible divisor \cite[Theorems~4.9 and~4.13]{MaciasOrtiz}. Self-homeomorphisms of the Mac\'ias space of $\N$ were studied in \cite{MaciasSelf}.

Throughout, all rings are commutative with identity. We denote $\#X$ for the cardinality of a set $X$ and write $\aleph_0=\#\N$. For a prime number $p$, we write $\F_p=\Z/p\Z$ for the field with $p$ elements. For a ring $T$, we write $X^0=X\setminus\{0\}$ for a subset $X$ of $T$, $\langle a\rangle=aT$ for the principal ideal generated by $a\in T$, $U(T)$ for the group of units of $T$, $\Max T$ for the set of maximal ideals of $T$ and $J(T)=\bigcap\limits_{\m\in\Max T}\m$ for the Jacobson radical of $T$. The ring $T$ is said to be \emph{semiprimitive} if $J(T)=0$. Two elements $a,b\in T$ are \emph{associates}, written $a\sim b$, if $a=ub$ for some $u\in U(T)$. Unless stated otherwise, $R$ and $S$ denote integral domains that are not fields.

For $k\in R^0$, define
\[
\sigma_k^0:=\{x\in R^0:\langle x\rangle+\langle k\rangle=R\}.
\]
The topology on $R^0$ generated by the basis $\beta_{R^0}=\{\sigma_k^0:k\in R^0\}$ is called the \emph{Mac\'ias topology} on $R^0$ \cite{MaciasDomains}. The \emph{Mac\'ias space} of $R$, denoted by $M(R)$, is the set $R^0$ endowed with the Mac\'ias topology. Motivated by the homeomorphism problem for Golomb spaces \cite[Section~3]{CLLP}, Mac\'ias posed the following problem.

\begin{macproblem}[{\cite[Problem~4.1]{MaciasDomains}}]
Let $R$ and $S$ be countably infinite semiprimitive integral domains. Decide whether $M(R)$ and $M(S)$ are homeomorphic.
\end{macproblem}
In this paper, we first determine when the unit group $U(R)$ is open in $M(R)$ and relate this to the density and the infinitude of prime elements (see Theorems~\ref{thm:B},~\ref{thm:primes} and Corollary~\ref{cor:pid-four}). Since $M(R)$ is not a Kolmogorov ($T_0$) space (Remark~\ref{rem:T0}), we then describe its Kolmogorov quotient in terms of the maximal ideals of $R$ and reduce the homeomorphism problem for Mac\'ias spaces to these quotients and the cardinalities of the classes of topologically indistinguishable elements (Theorem~\ref{thm:A}); using this reduction, we solve the homeomorphism problem for the Mac\'ias spaces of several classes of domains (Theorems~\ref{thm:C}, \ref{thm:ufd-general}, \ref{thm:ufd-iff} and~\ref{cor:polynomial} and Corollary~\ref{cor:pid-main}). Finally, we characterize additive homeomorphisms between Mac\'ias spaces of a certain class of integral domains of characteristic~$0$ (Theorem~\ref{thm:additive}). This answers, for domains in this class, a problem raised by Mac\'ias and Ortiz \cite[Problem~5.6]{MaciasOrtiz}.

A ring $T$ is said to be \emph{Noetherian} if every ideal of $T$ is finitely generated, and an integral domain $R$ is said to be \emph{integrally closed} if every element of the field of fractions of $R$ that is a root of a monic polynomial with coefficients in $R$ lies in $R$. A \emph{Dedekind domain} is a Noetherian, integrally closed domain in which every non-zero prime ideal is maximal. Every PID is a Dedekind domain \cite[Section~16.3, Proposition~14(1)]{DummitFoote} and an integral domain is both a Dedekind domain and a unique factorization domain (UFD) if and only if it is a PID \cite[Section~16.3, Corollary~20]{DummitFoote}.

The following theorem, proved in Section~\ref{sec:density}, characterizes the semiprimitivity of $R$ in terms of the topology of $M(R)$ for various domains $R$.
\begin{theorem}\label{thm:B}
Let $R$ be an integral domain that is not a field.
\begin{enumerate}[label=\textup{(\roman*)}]
\item $U(R)$ is not open in $M(R)$ if and only if $J(R)=0$.
\item If $R$ is a Dedekind domain, then the following conditions are equivalent:
\begin{enumerate}[label=\textup{(\alph*)}]
\item $U(R)$ is not open in $M(R)$;
\item the set of irreducible elements of $R$ is dense in $M(R)$;
\item $R$ has infinitely many maximal ideals;
\item $J(R)=0$.
\end{enumerate}
\item If $R$ is a UFD, then the following conditions are equivalent:
\begin{enumerate}[label=\textup{(\alph*)}]
\item $U(R)$ is not open in $M(R)$;
\item the set of prime elements of $R$ is dense in $M(R)$;
\item $J(R)=0$.
\end{enumerate}
\end{enumerate}
\end{theorem}

Theorem~\ref{thm:B}(i) recovers the results of \cite[Section~3]{MaciasDomains} and \cite[Theorem~4.4]{MaciasOrtiz}, which show that $U(R)$ is not open in $M(R)$ for PIDs with finite unit group and for domains with $\#U(R)<\#R$, respectively; indeed, $J(R)=0$ whenever $\#U(R)<\#R$ \cite[Proposition~2.2 and Corollary~4.3]{Clark}. Theorem~\ref{thm:B}(i) removes such cardinality restrictions; in particular, $U(R)$ is not open in $M(R)$ for every semiprimitive domain $R$, including domains with $\#U(R)=\#R$ such as $\Z[\sqrt2]$ and $\Q[x]$. The equivalence of (c) and (d) in Theorem~\ref{thm:B}(ii) is due to Clark \cite[Theorem~4.8]{Clark}. None of the hypotheses of Theorem~\ref{thm:B} can be omitted; see Remarks~\ref{rem:jacobson} and~\ref{rem:antimatter}.

We say that $R$ satisfies \textbf{\textit{condition~(H)}}\namedlabel{cond:H}{\textup{(H)}} if $R/\langle p\rangle$ is semiprimitive for every prime element $p$ of $R$; equivalently, every non-zero principal prime ideal of $R$ is an intersection of maximal ideals.

Condition~\ref{cond:H} is motivated by a question of Mac\'ias and Ortiz \cite[Problem~3.18]{MaciasOrtiz}, which, reformulated in terms of $M(R)$ (see Remark~\ref{rem:MO}(b)), asks in which domains $R$ other than PIDs the closure of an irreducible element $p$ is the set $\langle p\rangle\setminus\{0\}$. By Corollary~\ref{cor:closureprime}, this holds precisely when $R/\langle p\rangle$ is semiprimitive. Since the irreducible elements of a UFD are its prime elements, a UFD satisfies \ref{cond:H} if and only if the closure of each of its irreducible elements is the set of its non-zero multiples; equivalently, the closure of each $x\in R^0$ is the set of elements of $R^0$ divisible by every prime divisor of $x$ (Propositions~\ref{prop:closure} and~\ref{prop:J-ufd}).

Condition~\ref{cond:H} holds for every PID, since then $\langle p\rangle$ is a maximal ideal for every prime element $p$. It also holds for every UFD that is a Hilbert ring, that is, a ring in which every prime ideal is an intersection of maximal ideals (hence the letter H). By \cite[Theorem~4.19]{Eisenbud}, where Hilbert rings are called Jacobson rings, every ring that is finitely generated as an algebra over a field or over $\Z$ is a Hilbert ring; thus condition~\ref{cond:H} holds, for instance, for the polynomial rings $F[x_1,\dots,x_n]$ over a field $F$ and for $\Z[x_1,\dots,x_n]$. On the other hand, condition~\ref{cond:H} fails for the local UFD $\C\pow{x,y}$ \cite[Chapter~1, Exercise~5]{AtiyahMacdonald}, \cite[Theorem~20.3]{Matsumura}, since $\C\pow{x,y}/\langle x\rangle\cong\C\pow{y}$ is local but not a field. 

Theorem~\ref{thm:primes} below, also proved in Section~\ref{sec:density}, characterizes the infinitude of prime elements in terms of the topology of $M(R)$ for various domains $R$.
\begin{theorem}\label{thm:primes}
Let $R$ be an integral domain that is not a field.
\begin{enumerate}[label=\textup{(\roman*)}]
\item If $R$ is a Dedekind domain, then the set of prime elements of $R$ is dense in $M(R)$ if and only if $R$ has infinitely many pairwise non-associated prime elements.
\item If $R$ is a UFD satisfying condition~\ref{cond:H}, then the following conditions are equivalent:
\begin{enumerate}[label=\textup{(\alph*)}]
\item $U(R)$ is not open in $M(R)$;
\item the set of prime elements of $R$ is dense in $M(R)$;
\item $R$ has infinitely many pairwise non-associated prime elements;
\item $J(R)=0$.
\end{enumerate}
\end{enumerate}
\end{theorem}

\begin{corollary}\label{cor:pid-four}
Let $R$ be a PID that is not a field. Then the following conditions are equivalent:
\begin{enumerate}[label=\textup{(\alph*)}]
\item $U(R)$ is not open in $M(R)$;
\item the set of prime elements of $R$ is dense in $M(R)$;
\item $R$ has infinitely many pairwise non-associated prime elements;
\item $J(R)=0$.
\end{enumerate}
\end{corollary}

Note that by Theorem~\ref{thm:B}(iii), conditions (a), (b) and (d) of Theorem~\ref{thm:primes}(ii) are equivalent for every UFD, so condition~\ref{cond:H} is needed only for (c).

In Theorem~\ref{thm:primes}, neither the hypothesis that $R$ is a Dedekind domain in (i) nor condition~\ref{cond:H} in (ii) can be omitted, and the two equivalent conditions in (i) imply $J(R)=0$, but not conversely; see Remark~\ref{rem:ufd-sharp}.

We now turn our attention to the homeomorphism problem for Mac\'ias spaces. For $x\in R^0$, we define the \emph{maximal support} of $x$ to be the set $V(x):=\{\m\in\Max R:x\in\m\}$; when several rings are involved, we write $V_R(x)$. We denote by $\K(R)$ the set of maximal supports of the elements of $R^0$, that is,
\[
\K(R)=\{V(x):x\in R^0\},
\]
and for $k\in R^0$ and $E\in\K(R)$, we put
\[
\W_k=\{F\in\K(R):F\cap V(k)=\emptyset\}\qquad\text{and}\qquad R^0_E=\{x\in R^0:V(x)=E\}.
\]
The collection $\{\W_k:k\in R^0\}$ is a basis for a topology on $\K(R)$ (see Lemma~\ref{lem:K}), and we endow $\K(R)$ with this topology. The Mac\'ias topology is the coarsest topology on $R^0$ for which the map $V_R\colon R^0\to2^{\Max R}$, $x\mapsto V_R(x)$, is continuous, where the power set $2^{\Max R}$ of $\Max R$ carries the topology generated by the sets $\{F\subseteq\Max R:F\cap V_R(k)=\emptyset\}$ with $k\in R^0$ (see Remark~\ref{rem:kolmogorov}). The term `maximal support' is motivated by the identity $V(x)=\operatorname{Supp}(R/\langle x\rangle)\cap\Max R$, where $\operatorname{Supp}$ denotes the support of a module. For Dedekind domains, the sets $V(x)$, $\K(R)$ and $R^0_E$ were already considered by Spirito in his study of the Golomb topology \cite{Spirito}, where $\K(R)$ and $R^0_E$ are denoted by $\boldsymbol{\mathcal V}(R)$ and $G_\Delta(R)$, respectively.

Two points of a topological space are said to be \emph{topologically indistinguishable} if they have the same open neighborhoods, and the quotient space obtained by identifying such points is called the \emph{Kolmogorov quotient} of the space. Lemma~\ref{lem:indist} characterizes topological indistinguishability in $M(R)$ in terms of maximal supports and identifies the Kolmogorov quotient of $M(R)$ with $\K(R)$ via the map $V_R\colon M(R)\to\K(R)$; in particular, the collection $\{R^0_E:E\in\K(R)\}$ partitions $R^0$ into the classes of topologically indistinguishable elements of $M(R)$. The proof of Lemma~\ref{lem:indist} relies on the description of closures in $M(R)$ due to Mac\'ias and Ortiz \cite[Proposition~3.8]{MaciasOrtiz}, which we express in terms of maximal supports in Proposition~\ref{prop:closure}. For $R=\Z[i]$, the description of the Kolmogorov quotient of $M(R)$ given by Corollary~\ref{cor:ufd-kq}(b) was obtained earlier in \cite{Mandal}.

The following theorem, proved in Section~\ref{sec:kolmogorov}, characterizes the existence of a homeomorphism between $M(R)$ and $M(S)$ in terms of the spaces $\K(R)$ and $\K(S)$ and the cardinalities of the sets $R^0_E$ and $S^0_F$.
\begin{theorem}\label{thm:A}
Let $R$ and $S$ be integral domains that are not fields. Then $M(R)$ and $M(S)$ are homeomorphic if and only if there is a homeomorphism $\Phi\colon\K(R)\to\K(S)$ with $\#R^0_E=\#S^0_{\Phi(E)}$ for every $E\in\K(R)$.
\end{theorem}

A ring is said to be \emph{semilocal} if it has only finitely many maximal ideals. For semilocal domains, Theorem~\ref{thm:A} yields the following classification.
\begin{corollary}\label{cor:semilocal}
Let $R$ and $S$ be semilocal domains that are not fields. Then $M(R)$ and $M(S)$ are homeomorphic if and only if $\#R=\#S$ and $\#{\Max R}=\#{\Max S}$.
\end{corollary}

For local domains, that is, semilocal domains with exactly one maximal ideal, the condition $\#{\Max R}=\#{\Max S}$ in Corollary~\ref{cor:semilocal} holds automatically. Hence two local domains that are not fields have homeomorphic Mac\'ias spaces if and only if they have the same cardinality.

The \emph{Krull dimension} of a ring $T$ is the supremum of the lengths $r$ of chains $\mathfrak p_0\subsetneq\mathfrak p_1\subsetneq\dots\subsetneq\mathfrak p_r$ of prime ideals of $T$. Corollary~\ref{cor:semilocal} shows that the Krull dimension is not invariant under homeomorphisms of Mac\'ias spaces; for instance, the local UFDs $\C\pow{t}$ and $\C\pow{x,y}$ have equal cardinality, and hence homeomorphic Mac\'ias spaces, although their Krull dimensions are $1$ and $2$, respectively.

For countable domains, Theorem~\ref{thm:A} takes the following form.

\begin{corollary}\label{cor:countableK}
Let $R$ and $S$ be countable integral domains that are not fields. Then $M(R)$ and $M(S)$ are homeomorphic if and only if $\#U(R)=\#U(S)$ and the spaces $\K(R)$ and $\K(S)$ are homeomorphic.
\end{corollary}

However, for arbitrary integral domains $R$ and $S$, and even for countable ones, it remains open whether the existence of a homeomorphism between $\K(R)$ and $\K(S)$ can be reduced further, since in general no characterization of such homeomorphisms is known. Thus the homeomorphism problem for Mac\'ias spaces of arbitrary countable semiprimitive domains remains open.

The next theorem, also proved in Section~\ref{sec:kolmogorov}, describes the homeomorphisms between Mac\'ias spaces.

\begin{theorem}\label{cor:homeo}
Let $R$ and $S$ be integral domains that are not fields. A bijection $h\colon R^0\to S^0$ is a homeomorphism from $M(R)$ onto $M(S)$ if and only if there is a homeomorphism $\Phi\colon\K(R)\to\K(S)$ with $V_S(h(x))=\Phi(V_R(x))$ for every $x\in R^0$.
\end{theorem}

Let $R$ be a Dedekind domain with field of fractions $K$. A \emph{fractional ideal} of $R$ is an $R$-submodule $M$ of $K$ such that $xM\subseteq R$ for some $x\in R^0$. The non-zero fractional ideals of $R$ form an abelian group under multiplication \cite[p.~98]{AtiyahMacdonald}, and the \emph{ideal class group} $\Cl(R)$ of $R$ is the quotient of this group by its subgroup $\{uR:u\in K^0\}$ of non-zero principal fractional ideals. We write $\Cl(R)$ additively and denote by $[I]\in\Cl(R)$ the class of a non-zero ideal $I$ of $R$; the cardinality of $\Cl(R)$ is called the \emph{class number} of $R$. The group $\Cl(R)$ is trivial if and only if $R$ is a PID; indeed, if $M$ is a non-zero fractional ideal of $R$ and $x\in R^0$ satisfies $xM\subseteq R$, then $M$ is principal precisely when the ideal $xM$ of $R$ is principal. We call a finite set $E\subseteq\Max R$ \emph{balanced} if $\sum\limits_{\m\in E}n_\m[\m]=0$ in $\Cl(R)$ for some positive integers $n_\m$; the empty set is balanced, the sum being empty. By Lemma~\ref{lem:dedekind}(a), the balanced subsets of $\Max R$ are exactly the members of $\K(R)$.

The homeomorphism problem for the Mac\'ias spaces of Dedekind domains is solved by the following theorem, proved in Section~\ref{sec:dedekind}.
\begin{theorem}\label{thm:C}
Let $R$ and $S$ be Dedekind domains that are not fields. Then $M(R)$ and $M(S)$ are homeomorphic if and only if $\#U(R)=\#U(S)$ and there exists a bijection $\varphi\colon\Max R\to\Max S$ that maps the balanced subsets of $\Max R$ onto those of $\Max S$.
\end{theorem}

If $\Cl(R)$ and $\Cl(S)$ are torsion groups, then every finite subset of $\Max R$ and of $\Max S$ is balanced, as one may take each $n_\m$ to be the order of $[\m]$. Hence Theorem~\ref{thm:C} gives the following corollary.
\begin{corollary}\label{cor:torsion}
Let $R$ and $S$ be Dedekind domains with torsion class groups, neither of which is a field. Then $M(R)$ and $M(S)$ are homeomorphic if and only if $\#U(R)=\#U(S)$ and $\#{\Max R}=\#{\Max S}$.
\end{corollary}

The torsion hypothesis cannot be omitted in Corollary~\ref{cor:torsion} (Remark~\ref{rem:dedekind-sharp}(c)), nor can the Dedekind hypothesis; in fact, even for UFDs $R$, the cardinalities $\#U(R)$ and $\#{\Max R}$ do not determine $M(R)$ up to homeomorphism (Remark~\ref{rem:dedekind-sharp}(a)).

For a UFD $R$, we write $\Pe_R$ for the set of associate classes of prime elements of $R$, and $[p]\in\Pe_R$ for the class of a prime element $p$. For a UFD $S$ satisfying condition~\ref{cond:H}, the homeomorphism type of $M(S)$ determines whether $S$ is a PID (Proposition~\ref{prop:pid-char}); together with Corollary~\ref{cor:torsion}, this yields the following.
\begin{corollary}\label{cor:pid-ufd}
Let $R$ be a Dedekind domain with torsion class group and $S$ a UFD satisfying condition~\ref{cond:H}, neither of which is a field. Then $M(R)$ and $M(S)$ are homeomorphic if and only if $S$ is a PID, $\#U(R)=\#U(S)$ and $\#{\Max R}=\#\Pe_S$.
\end{corollary}

If $R$ is a PID, the last condition reduces to $\#\Pe_R=\#\Pe_S$, since
$\#{\Max R}=\#\Pe_R$. Condition~\ref{cond:H} is essential in
Corollary~\ref{cor:pid-ufd}, as shown by Remark~\ref{rem:J-ufd-sharp}.
For PIDs, Corollary~\ref{cor:pid-ufd} thus provides the following
classification.
\begin{corollary}\label{cor:pid-main}
Let $R$ and $S$ be PIDs that are not fields. Then $M(R)$ and $M(S)$ are homeomorphic if and only if $\#U(R)=\#U(S)$ and $\#\Pe_R=\#\Pe_S$.
\end{corollary}
Corollary~\ref{cor:torsion} also yields the following.

\begin{corollary}\label{cor:countable}
Let $R$ and $S$ be countable semiprimitive Dedekind domains with torsion class groups, neither of which is a field. Then $M(R)$ and $M(S)$ are homeomorphic if and only if $\#U(R)=\#U(S)$.
\end{corollary}

Theorem~\ref{thm:C} answers the homeomorphism problem of Mac\'ias for countable semiprimitive Dedekind domains. For those with torsion class group, Corollary~\ref{cor:countable} shows that the answer depends only on the cardinalities of the unit groups; in particular, neither the class number nor the characteristic is invariant under homeomorphisms of Mac\'ias spaces (Example~\ref{ex:examples}). 

The next theorem, also proved in Section~\ref{sec:dedekind}, describes the homeomorphisms between the Mac\'ias spaces of Dedekind domains in terms of maximal ideals.

\begin{theorem}\label{cor:dedekind-homeo}
Let $R$ and $S$ be Dedekind domains that are not fields. A bijection $h\colon R^0\to S^0$ is a homeomorphism from $M(R)$ onto $M(S)$ if and only if there is a bijection $\varphi\colon\Max R\to\Max S$ with $V_S(h(x))=\varphi(V_R(x))$ for every $x\in R^0$.
\end{theorem} 

In contrast with the rigidity of Golomb spaces recalled above, Theorem~\ref{cor:dedekind-homeo} shows that $M(\Z)$ has uncountably many self-homeomorphisms (see Remark~\ref{rem:golomb}).

We now turn our attention to UFDs. For a UFD $R$ and $x\in R^0$, we denote by $\mathcal{C}_R(x)\subseteq\Pe_R$ the set of associate classes of prime elements of $R$ that are comaximal with $x$, and we write $\mathfrak{C}(R)=\{\mathcal{C}_R(x):x\in R^0\}$. In Lemma~\ref{lem:types}, we show that the map $V(x)\mapsto\mathcal{C}_R(x)$ is a well-defined bijection from $\K(R)$ onto $\mathfrak{C}(R)$. For $k\in R^0$ and $C\in\mathfrak{C}(R)$, we define
\[
\mathcal{O}_R(k):=\{\mathcal{C}_R(x):x\in R^0,\ \langle x\rangle+\langle k\rangle=R\}\qquad\text{and}\qquad\mu_R(C):=\#\{x\in R^0:\mathcal{C}_R(x)=C\}.
\]

The following theorem, proved in Section~\ref{sec:beyond}, characterizes the pairs of UFDs with homeomorphic Mac\'ias spaces in terms of the sets $\mathcal{O}_R(k)$ and the cardinalities $\mu_R(C)$.
\begin{theorem}\label{thm:ufd-general}
Let $R$ and $S$ be UFDs that are not fields. Then $M(R)$ and $M(S)$ are homeomorphic if and only if there is a bijection $\Phi\colon\mathfrak{C}(R)\to\mathfrak{C}(S)$ with the following properties:
\begin{enumerate}[label=\textup{(\roman*)}]
\item for every $k\in R^0$, the image $\Phi(\mathcal{O}_R(k))$ is a union of sets $\mathcal{O}_S(l)$ with $l\in S^0$, and for every $l\in S^0$, the preimage $\Phi^{-1}(\mathcal{O}_S(l))$ is a union of sets $\mathcal{O}_R(k)$ with $k\in R^0$;
\item $\mu_S(\Phi(C))=\mu_R(C)$ for every $C\in\mathfrak{C}(R)$.
\end{enumerate}
\end{theorem}

For UFDs $R$ and $S$, we say that a bijection $\varphi\colon\Pe_R\to\Pe_S$ \emph{preserves comaximality} if, for all prime elements $p,q$ of $R$ and $p',q'$ of $S$ satisfying $[p']=\varphi([p])$ and $[q']=\varphi([q])$, one has $\langle p\rangle+\langle q\rangle=R$ if and only if $\langle p'\rangle+\langle q'\rangle=S$.
\begin{corollary}\label{cor:comax}
Let $R$ and $S$ be UFDs that are not fields. If $\#U(R)=\#U(S)$ and there is a bijection $\Pe_R\to\Pe_S$ that preserves comaximality, then $M(R)$ and $M(S)$ are homeomorphic.
\end{corollary}
The converse of Corollary~\ref{cor:comax} does not hold in general (Remark~\ref{rem:J-ufd-sharp}).

For a UFD $R$, we call a subset $Y\subseteq\Pe_R$ \emph{comaximally open} if each finite subset of $Y$ is contained in $\mathcal{C}_R(k)\subseteq Y$ for some $k\in R^0$. The term `comaximally open' is motivated by the fact that, for a UFD $R$ satisfying condition~\ref{cond:H}, a subset $Y\subseteq\Pe_R$ is comaximally open if and only if the set $\{x\in R^0:[p]\in Y\text{ for every prime factor }p\text{ of }x\}$ is open in $M(R)$ (see Lemma~\ref{lem:fin-homeo}(a)). If $R$ is a PID, then $\mathcal{C}_R(k)$ consists of the classes of prime elements not dividing $k$, so that the comaximally open subsets of $\Pe_R$ are precisely the cofinite ones, that is, those with finite complement in $\Pe_R$. For UFDs $R$ and $S$, we say that a bijection $\varphi\colon\Pe_R\to\Pe_S$ \emph{preserves comaximally open sets} if, for every $Y\subseteq\Pe_R$, the set $Y$ is comaximally open precisely when $\varphi(Y)$ is. Every bijection $\Pe_R\to\Pe_S$ that preserves comaximality also preserves comaximally open sets (Section~\ref{sec:beyond}), but whether the existence of a bijection $\Pe_R\to\Pe_S$ that preserves comaximally open sets implies the existence of one that preserves comaximality remains open (see Problem~\ref{prob:ufd}).

\begin{theorem}\label{thm:ufd-iff}
Let $R$ and $S$ be UFDs that are not fields and satisfy condition~\ref{cond:H}. Then $M(R)$ and $M(S)$ are homeomorphic if and only if $\#U(R)=\#U(S)$ and there is a bijection $\Pe_R\to\Pe_S$ that preserves comaximally open sets.
\end{theorem}

Condition~\ref{cond:H} cannot be omitted in Theorem~\ref{thm:ufd-iff} (see Remark~\ref{rem:J-ufd-sharp}). We also describe the Mac\'ias space of a UFD in terms of the comaximality relation among its prime elements (Proposition~\ref{prop:ufd-comax}).

A ring is said to be a \emph{discrete valuation ring} (DVR) if it is a local PID that is not a field; by the \emph{residue field} of a DVR $A$, we mean the field $A/\m$, where $\Max A=\{\m\}$. The following theorem, proved in Sections~\ref{sec:polynomial} and~\ref{sec:proof-b}, settles the homeomorphism problem for the Mac\'ias spaces of large classes of polynomial rings.
\begin{theorem}\label{cor:polynomial}
Let $A$ and $A'$ be DVRs with residue fields $\kappa$ and $\kappa'$, let $B$ and $B'$ be at most countable UFDs, and let $k$ and $k'$ be at most countable fields.
\begin{enumerate}[label=\textup{(\alph*)}]
\item $M(A[x])$ is not homeomorphic to $M(S)$ for any UFD $S$ satisfying condition~\ref{cond:H}. 
\item $M(A[x])$ and $M(A'[x])$ are homeomorphic if and only if $\#A=\#A'$ and $\max\{\#\kappa,\aleph_0\}=\max\{\#\kappa',\aleph_0\}$.
\item $M(B[x_1,x_2,\dots])$ and $M(B'[x_1,x_2,\dots])$ are homeomorphic if and only if $\#U(B)=\#U(B')$.
\item $M(k[x_1,\dots,x_n])$ and $M(k'[x_1,\dots,x_m])$ are homeomorphic if and only if $\#k=\#k'$ and either $n=m=1$ or $n,m\geq2$.
\end{enumerate}
\end{theorem}

By Lemma~\ref{lem:H-several}, Theorem~\ref{cor:polynomial}(a) yields $M(A[x])\not\cong M(A'[x_1,\dots,x_n])$ for every $n\geq2$. However, the classification of the spaces $M(A[x_1,\dots,x_n])$, where $A$ is a DVR and $n\geq2$, remains open. By Theorem~\ref{cor:polynomial}(b), the spaces $M(A[x])$ with $A$ a countable DVR are all homeomorphic, since then $\#A=\aleph_0$ and $\#\kappa\leq\aleph_0$.

Theorem~\ref{cor:polynomial}(c) does not extend to polynomial rings in one variable; for instance, it follows from Example~\ref{ex:examples} and Remark~\ref{rem:dedekind-sharp}(a) that $M(\F_3[x])\not\cong M(\Z[x])$, although the rings $\F_3$ and $\Z$ have two units each. For arbitrary UFDs $B$ and $B'$, the conditions $\#U(B)=\#U(B')$ and $\max\{\#B,\aleph_0\}=\max\{\#B',\aleph_0\}$ are necessary for $M(B[x_1,x_2,\dots])$ and $M(B'[x_1,x_2,\dots])$ to be homeomorphic, by Corollary~\ref{cor:units}, as $U(B[x_1,x_2,\dots])=U(B)$ (see the proof of Lemma~\ref{lem:bijection}) and $\#B[x_1,x_2,\dots]=\max\{\#B,\aleph_0\}$; whether they are also sufficient remains open.

For $n=1$ or $m=1$, Theorem~\ref{cor:polynomial}(d) holds for arbitrary fields $k$ and $k'$ (see Remark~\ref{rem:poly-arbitrary}). More precisely, $M(k[x])\cong M(k'[x])$ if and only if $\#k=\#k'$, whereas $M(k[x])\not\cong M(k'[x_1,\dots,x_m])$ for every $m\geq2$. For $n,m\geq2$, however, it remains open whether Theorem~\ref{cor:polynomial}(d) extends to uncountable fields.

It is natural to ask whether isomorphic integral domains have homeomorphic Mac\'ias spaces, and to what extent the converse holds. The answer to the first question is affirmative, since every ring isomorphism $R\to S$ yields a homeomorphism $M(R)\to M(S)$ \cite[Theorem~5.2]{MaciasOrtiz}. The converse, however, does not hold in general, since non-isomorphic integral domains may have homeomorphic Mac\'ias spaces (Example~\ref{ex:examples}); moreover, even a self-homeomorphism of $M(R)$ need not be induced by a ring automorphism of $R$ \cite[Example~5.4]{MaciasOrtiz}. 

In \cite[Theorem~2.1]{MaciasDomains}, Mac\'ias also introduced the space $\widetilde{M}(R)$ obtained by endowing $R$ with the topology generated by the basis $\{\sigma_k : k \in R\}$, where $\sigma_k := \{x \in R : \langle x \rangle + \langle k \rangle = R\}$. Thus $\sigma_k^0=\sigma_k\setminus\{0\}$ for every $k\in R^0$. Mac\'ias and Ortiz formulated the second question in terms of the spaces $\widetilde{M}(R)$ as follows.
\begin{macproblem}[{\cite[Problem~5.6]{MaciasOrtiz}}]
When does a homeomorphism $\widetilde{M}(R)\to\widetilde{M}(S)$ determine a ring isomorphism $R\to S$?
\end{macproblem}
By Proposition~\ref{prop:tilde}(c), $\widetilde{M}(R)\cong\widetilde{M}(S)$ if and only if $M(R)\cong M(S)$; hence the classification of the spaces $\widetilde{M}(R)$ up to homeomorphism reduces to that of the spaces $M(R)$. Moreover, by Proposition~\ref{prop:tilde}(a), if $R$ and $S$ are semiprimitive, then the homeomorphisms $\widetilde{M}(R)\to\widetilde{M}(S)$ are precisely the extensions by $0\mapsto0$ of the homeomorphisms $M(R)\to M(S)$; in this case, the problem of Mac\'ias and Ortiz can therefore be stated entirely in terms of $M(R)$ and $M(S)$. We address it under the following arithmetic condition.

We say that $R$ satisfies \textbf{\textit{condition~(P)}}\namedlabel{cond:P}{\textup{(P)}} if the maximal ideals $\m$ of $R$ for which $R/\m$ is a \emph{prime field} (that is, a field without proper subfields) have trivial intersection; equivalently, every $x\in R^0$ lies outside some such maximal ideal.

Condition~\ref{cond:P} allows us to recover the multiplication of $R$ from its addition and the topology of $M(R)$ (see Section~\ref{sec:additive}). If $h\colon R\to S$ is an additive bijection that restricts to a homeomorphism $M(R)\to M(S)$, then $h(1)$ is a unit of $S$ and $h$ maps maximal ideals onto maximal ideals (Corollary~\ref{cor:units} and Lemma~\ref{lem:additive-max}). Thus, for every maximal ideal $\m$ of $R$, the map $h(1)^{-1}h$ induces an additive bijection $R/\m\to S/h(\m)$ that preserves the multiplicative identity. Since a prime field is generated as a field by its multiplicative identity, this bijection is a ring isomorphism whenever $R/\m$ is a prime field (Lemma~\ref{lem:prime-field}). Condition~\ref{cond:P} holds precisely when every element of $R$ is determined by its images in the residue fields $R/\m$ that are prime fields; it therefore forces $h(1)^{-1}h$ to be multiplicative (see the proof of Theorem~\ref{thm:additive}).

If $R$ satisfies condition~\ref{cond:P}, then $R$ is semiprimitive, since $J(R)$ is contained in every maximal ideal of $R$. Moreover, $R$ must have characteristic~$0$. Indeed, if $R$ had characteristic $p>0$, then every residue field $R/\m$ that is a prime field would be isomorphic to $\F_p$. Hence, by Fermat's little theorem, the element $x^p-x$ would lie in all such maximal ideals for every $x\in R$. Since these ideals have trivial intersection by condition~\ref{cond:P}, it would follow that $x^p=x$ for every $x\in R$ and hence, $R$ being a domain, that $x^{p-1}=1$ for every $x\in R^0$. Every non-zero element of $R$ would thus be a unit, contrary to the assumption that $R$ is not a field. On the other hand, condition~\ref{cond:P} is strictly stronger than being semiprimitive of characteristic~$0$. If $R$ contains a field $K$ that is not a prime field, then every residue field of $R$ contains a copy of $K$ and hence a proper subfield, so that no residue field of $R$ is a prime field and condition~\ref{cond:P} fails; this applies, for instance, to the semiprimitive domain $\C[x]$.

Condition~\ref{cond:P} holds for $\Z[1/n]$ for every $n\in\N$, in particular for $\Z$, as well as for $\Z[i]$ and the polynomial rings $\Z[x_1,\dots,x_d]$ and $\Q[x_1,\dots,x_d]$ for every $d\geq1$. To see this, it suffices to find, for each of these domains $R$ and every $f\in R^0$, a ring homomorphism from $R$ onto a prime field that does not map $f$ to $0$, since its kernel is then a maximal ideal that does not contain $f$ and whose residue field is a prime field. For $R=\Z[1/n]$, write $f=a/n^e$ with $a\in\Z^0$ and $e\geq0$, and let $p$ be a prime number not dividing $an$. Since $n$ is invertible modulo $p$, the canonical map $\Z\to\F_p$ extends to $\Z[1/n]$ and maps $f$ to $(a+p\Z)(n+p\Z)^{-e}\neq0$. For $R=\Z[i]$, Dirichlet's theorem \cite[Theorem~2.1]{diri} yields a prime number $p\equiv1\pmod 4$ not dividing the positive integer $f\bar{f}$, and Euler's criterion \cite[Theorem~83]{HardyWright}, applied to $-1$, shows that there is an integer $a$ with $a^2\equiv-1\pmod p$; thus $i\mapsto a+p\Z$ defines a ring homomorphism from $\Z[i]$ onto $\F_p$, which maps $f\bar{f}$ to $f\bar{f}+p\Z\neq0$ and hence does not map $f$ to $0$. For $R=\Q[x_1,\dots,x_d]$ or $R=\Z[x_1,\dots,x_d]$, induction on $d$ yields $c\in\Z^d$ with $f(c)\neq0$, since a non-zero polynomial in one variable has only finitely many roots. Evaluation at $c$ is then a ring homomorphism from $\Q[x_1,\dots,x_d]$ onto $\Q$, and evaluation at $c$ followed by reduction modulo a prime number $p$ not dividing $f(c)$ is a ring homomorphism from $\Z[x_1,\dots,x_d]$ onto $\F_p$; neither of them maps $f$ to $0$.

For domains satisfying condition~\ref{cond:P}, the problem of Mac\'ias and Ortiz is answered by the following theorem, proved in Section~\ref{sec:additive}.
\begin{theorem}\label{thm:additive}
Let $R$ and $S$ be integral domains that are not fields, and suppose that $R$ satisfies condition~\ref{cond:P}. Then a bijection $h\colon R\to S$ is additive and restricts to a homeomorphism from $M(R)$ onto $M(S)$ if and only if there are a ring isomorphism $\rho\colon R\to S$ and a unit $u$ of $S$ such that $h(x)=u\,\rho(x)$ for every $x\in R$.
\end{theorem}

Theorem~\ref{thm:additive} yields the following characterization of the integral domains isomorphic to a given integral domain satisfying condition~\ref{cond:P}.
\begin{corollary}\label{cor:additive-iso}
Let $R$ and $S$ be integral domains that are not fields, and suppose that $R$ satisfies condition~\ref{cond:P}. Then $R$ and $S$ are isomorphic if and only if there exists an additive bijection $R\to S$ that restricts to a homeomorphism from $M(R)$ onto $M(S)$.
\end{corollary}

Additivity cannot be dropped from Theorem~\ref{thm:additive} and Corollary~\ref{cor:additive-iso}, since the homeomorphism type of a Mac\'ias space retains little of the ring structure; for instance, $\Z$ satisfies condition~\ref{cond:P} and $M(\Z)\cong M(\Z[\sqrt{-5}])$ by Example~\ref{ex:examples}, although $\Z\not\cong\Z[\sqrt{-5}]$. Moreover, additivity cannot be replaced by multiplicativity in either result, and condition~\ref{cond:P} cannot be omitted in Theorem~\ref{thm:additive}, even in characteristic~$0$ (see Remark~\ref{rem:additive-sharp}).

\subsection*{Organization of the paper}
Theorems~\ref{thm:B} and~\ref{thm:primes} as well as Corollary~\ref{cor:pid-four} are proved in Section~\ref{sec:density}. Section~\ref{sec:kolmogorov} contains the proofs of Theorems~\ref{thm:A} and~\ref{cor:homeo} and of Corollaries~\ref{cor:semilocal} and \ref{cor:countableK}, and Section~\ref{sec:dedekind} those of Theorems~\ref{thm:C} and~\ref{cor:dedekind-homeo} and of Corollaries~\ref{cor:torsion} and~\ref{cor:countable}. In Section~\ref{sec:beyond} we prove Theorems~\ref{thm:ufd-general} and~\ref{thm:ufd-iff} as well as Corollaries~\ref{cor:pid-ufd}, \ref{cor:pid-main} and~\ref{cor:comax}, and we obtain further results on the Mac\'ias spaces of UFDs, including a characterization of the PIDs among the UFDs satisfying condition~\ref{cond:H} (Proposition~\ref{prop:pid-char}). Parts~(a), (c) and~(d) of Theorem~\ref{cor:polynomial} are proved in Section~\ref{sec:polynomial}, where we also show that polynomial rings in at least two variables over a DVR satisfy condition~\ref{cond:H} (Lemma~\ref{lem:H-several}) and may therefore serve as the ring~$S$ in part~(a). Part~(b) is proved in Section~\ref{sec:proof-b}. Theorem~\ref{thm:additive} and Corollary~\ref{cor:additive-iso} are proved in Section~\ref{sec:additive}, where we also compare the spaces $\widetilde{M}(R)$ and $M(R)$ (Proposition~\ref{prop:tilde}). Finally, in Section~\ref{sec:conclusion} we pose an open problem concerning the converse of Corollary~\ref{cor:comax} for UFDs satisfying condition~\ref{cond:H}.
\section{Semiprimitivity, Infinitude and Density of prime elements}\label{sec:density}
In this section we prove Theorems~\ref{thm:B} and~\ref{thm:primes} and Corollary~\ref{cor:pid-four}, using the classical characterization \cite[Proposition~1.9]{AtiyahMacdonald}
\begin{equation}\label{eq:jacobson}
k\in J(R)\iff 1+kr\in U(R)\ \text{ for every } r\in R
\end{equation}
of the Jacobson radical. We first determine the basic open sets of $M(R)$ that coincide with $U(R)$; this yields Theorem~\ref{thm:B}(i) and Proposition~\ref{prop:density}.

\begin{lemma}\label{lem:sigmaJ}
Let $k\in R^0$. Then $\sigma_k^0=U(R)$ if and only if $k\in J(R)$.
\end{lemma}
\begin{proof}
Suppose that $k\in J(R)$, and let $x\in\sigma_k^0$. Since $k$ lies in every maximal ideal of $R$ and $\langle x\rangle+\langle k\rangle=R$, no maximal ideal of $R$ contains $x$; hence $x$ is a unit by \cite[Corollary~1.4]{AtiyahMacdonald}. Thus $\sigma_k^0\subseteq U(R)$, and equality holds because every unit is comaximal with $k$. 

Conversely, suppose that $\sigma_k^0=U(R)$. Then $k$ is not a unit, for otherwise $\sigma_k^0=R^0$, whereas $R^0\neq U(R)$ as $R$ is not a field. Let $r\in R$. Since $1=(1+kr)-kr$, the element $1+kr$ is comaximal with $k$, and $1+kr\neq0$ because $k$ is not a unit; hence $1+kr\in\sigma_k^0=U(R)$. As this holds for every $r\in R$, it follows from \eqref{eq:jacobson} that $k\in J(R)$.
\end{proof}
\begin{proof}[\normalfont\bfseries Proof of Theorem~\ref{thm:B}(i)]
Suppose that $U(R)$ is open in $M(R)$. Since $1\in U(R)$ and $\beta_{R^0}$ is a basis of $M(R)$, there is $k\in R^0$ with $1\in\sigma_k^0\subseteq U(R)$. As every unit lies in $\sigma_k^0$, it follows that $\sigma_k^0=U(R)$; hence $k\in J(R)$ by Lemma~\ref{lem:sigmaJ}, and so $J(R)\neq0$. Conversely, if $J(R)\neq0$, then $\sigma_k^0=U(R)$ for every non-zero $k\in J(R)$ by Lemma~\ref{lem:sigmaJ}, so that $U(R)$ is open in $M(R)$.
\end{proof}
\begin{remark}\label{rem:jacobson}
Neither hypothesis on $R$ can be omitted in Theorem~\ref{thm:B}(i). Since the space $M(R)$ is defined only for integral domains that are not fields, let $M'(R)$ denote, for an arbitrary ring $R$, the topological space whose underlying set $R^0$ is endowed with the topology generated by the subbasis $\beta_{R^0}=\{\sigma_k^0:k\in R^0\}$; thus $M'(R)=M(R)$ whenever $R$ is an integral domain that is not a field. If $R$ is a field, then $U(R)=R^0$ is open in $M'(R)$, whereas $J(R)=0$. For $R=\Z\times\Z$, which is not an integral domain, the sets $\sigma^0_{(1,0)}=\Z\times\{\pm1\}$ and $\sigma^0_{(0,1)}=\{\pm1\}\times\Z$ are open in $M'(R)$. Therefore, their intersection $U(R)$ is also open, while $J(R)=0$.
\end{remark}
\begin{proposition}\label{prop:density}
Let $T\subseteq R^0\setminus U(R)$ be such that every element of $R^0\setminus U(R)$ has a divisor in $T$. Then $T$ is dense in $M(R)$ if and only if $R$ is semiprimitive.
\end{proposition}
\begin{proof}
Since $\beta_{R^0}$ is a basis of $M(R)$, the set $T$ is dense if and only if it meets $\sigma_k^0$ for every $k\in R^0$. Each set $\sigma_k^0$ contains $U(R)$ and is closed under taking divisors, because $\langle s\rangle\subseteq\langle t\rangle$ whenever $t$ divides $s$. By the hypothesis on $T$, it follows that $\sigma_k^0$ meets $T$ if and only if $\sigma_k^0\neq U(R)$, that is, by Lemma~\ref{lem:sigmaJ}, if and only if $k\notin J(R)$. Hence $T$ is dense if and only if $J(R)=0$.
\end{proof}
\begin{remark}\label{rem:MO-dense}
Proposition~\ref{prop:density} applied to $T=R^0\setminus U(R)$, together with Theorem~\ref{thm:B}(i), implies that the set $R^0\setminus U(R)$ is dense in $M(R)$ whenever $U(R)$ is not open in $M(R)$. The analogous result for $\widetilde{M}(R)$ is \cite[Theorem~4.3]{MaciasOrtiz}.
\end{remark}
\begin{proof}[\normalfont\bfseries Proof of Theorem~\ref{thm:B}(ii) and~(iii)]
(ii) Let $R$ be a Dedekind domain. Conditions {(a)} and {(d)} are equivalent by {(i)}. Since $R$ is Noetherian, every non-zero non-unit of $R$ is a product of irreducible elements \cite[Section~2.1]{Clark}; hence {(b)} and {(d)} are equivalent by Proposition~\ref{prop:density}, applied to the set of irreducible elements of $R$. Finally, {(c)} and {(d)} are equivalent by \cite[Theorem~4.8]{Clark}.

(iii) Let $R$ be a UFD. Then {(a)} and {(c)} are equivalent by {(i)}. Since every non-zero non-unit of $R$ has a prime divisor, the equivalence of {(b)} and {(c)} follows from Proposition~\ref{prop:density}, applied to the set of prime elements of $R$.
\end{proof}
\begin{remark}\label{rem:antimatter}
The hypotheses on $R$ in (ii) and (iii) of Theorem~\ref{thm:B} cannot be omitted, as the following examples show.

(a) The ring $\overline{\Z}$ of all algebraic integers is a semiprimitive domain that is not a field and has no irreducible elements \cite[Example~4.7]{Clark}. Hence it has no prime elements. Thus conditions (b) and (d) of Theorem~\ref{thm:B}(ii), as well as conditions (b) and (c) of Theorem~\ref{thm:B}(iii), are not equivalent for $\overline{\Z}$.

(b) The power series ring $\Z\pow{t}$ is a Noetherian UFD with infinitely many maximal ideals, whereas $t\in J(\Z\pow{t})$ \cite[Example~4.12]{Clark}. Thus, for this ring, conditions (c) and (d) of Theorem~\ref{thm:B}(ii) are not equivalent, and condition (c) of Theorem~\ref{thm:B}(ii) is not equivalent to conditions (a)--(c) of Theorem~\ref{thm:B}(iii); hence no analogue of condition (c) of (ii) can be added to (iii).
\end{remark}
\begin{proof}[\normalfont\bfseries Proof of Theorem~\ref{thm:primes}]
We first show that, without further hypotheses on $R$, the set of prime elements of $R$ is not dense in $M(R)$ if $R$ has only finitely many associate classes of prime elements. Let $p_1,\dots,p_r$ represent these classes, and put $k=p_1\cdots p_r$, with $k=1$ if $r=0$. Every prime element $p$ satisfies $p\sim p_i$ for some $i$, so $k\in\langle p\rangle$. Hence $\langle p\rangle+\langle k\rangle=\langle p\rangle\neq R$. Thus the non-empty open set $\sigma_k^0$ contains no prime element.

{(i)} Let $R$ be a Dedekind domain. By the observation above, it suffices to show that the set of prime elements of $R$ is dense in $M(R)$ if $R$ has infinitely many pairwise non-associated prime elements. Let $k\in R^0$. By \cite[Corollary~9.4]{AtiyahMacdonald}, the ideal $\langle k\rangle$ is either $R$ or a product $\mathfrak{p}_1\cdots\mathfrak{p}_s$ of non-zero prime ideals, each of which is maximal by the definition of a Dedekind domain. As a prime ideal containing a product of ideals contains one of them, every prime ideal containing $k$ contains, and hence equals, some $\mathfrak{p}_i$. On the other hand, the ring $R$ has, by assumption, infinitely many pairwise non-associated prime elements, and these generate distinct prime ideals; hence there is a prime element $p$ with $k\notin\langle p\rangle$. As the non-zero prime ideal $\langle p\rangle$ is maximal, we obtain $\langle p\rangle+\langle k\rangle=R$, that is, $p\in\sigma_k^0$. Thus every basic open set of $M(R)$ contains a prime element.

{(ii)} Conditions (a), (b) and (d) are equivalent by Theorem~\ref{thm:B}(iii), and (b) implies (c) by the observation above. Conversely, assume (c). By condition~\ref{cond:H}, every ideal $\langle q\rangle$ generated by a prime element $q$ is an intersection of maximal ideals, so that $J(R)\subseteq\langle q\rangle$. 
Hence every element of \(J(R)\) is divisible by infinitely many pairwise non-associated prime elements. Since \(R\) is a UFD, this is possible only for the zero element.
Thus $J(R)=0$.
\end{proof}
\begin{remark}\label{rem:ufd-sharp}
(a) The hypothesis that $R$ is a Dedekind domain in Theorem~\ref{thm:primes}(i) and condition~\ref{cond:H} in Theorem~\ref{thm:primes}(ii) cannot be omitted. Indeed, the UFD $\Z\pow{t}$ of Remark~\ref{rem:antimatter}(b) has infinitely many pairwise non-associated prime elements, namely the prime numbers \cite[Example~4.12]{Clark}, whereas $J(\Z\pow{t})\neq0$. Therefore the set of its prime elements is not dense in $M(\Z\pow{t})$ by Theorem~\ref{thm:B}(iii).

\noindent(b) For a Dedekind domain $R$, the equivalent conditions of Theorem~\ref{thm:primes}(i) imply that $R$ has infinitely many maximal ideals, and hence that $J(R)=0$ by Theorem~\ref{thm:B}(ii). The converse fails, since the Dedekind domain $\C[x,y]/\langle y^2-x^3-x\rangle$ has infinitely many maximal ideals but no prime elements \cite[proof of Theorem~5.5]{Clark}.
\end{remark}
\begin{proof}[\normalfont\bfseries Proof of Corollary~\ref{cor:pid-four}]
A PID that is not a field is a UFD satisfying condition~\ref{cond:H}, as noted in Section~\ref{sec:intro}. Hence the assertion follows from Theorem~\ref{thm:primes}(ii).
\end{proof}

\section{Homeomorphisms between Mac\'ias spaces of integral domains}\label{sec:kolmogorov}

In this section we prove Theorems~\ref{thm:A} and~\ref{cor:homeo} as well as Corollaries~\ref{cor:semilocal}--\ref{cor:countableK}. We begin with three elementary properties of maximal supports.

\begin{lemma}\label{lem:V}
Let $x,y,k\in R^0$.
\begin{enumerate}[label=\textup{(\alph*)}]
\item $V(x)=\emptyset$ if and only if $x\in U(R)$.
\item $V(xy)=V(x)\cup V(y)$; in particular, $V(x)\subseteq V(y)$ whenever $x$ divides $y$.
\item $x\in\sigma_k^0$ if and only if $V(x)\cap V(k)=\emptyset$.
\end{enumerate}
\end{lemma}
\begin{proof}
Parts (a) and (c) follow from \cite[Corollary~1.4]{AtiyahMacdonald}, applied to $\langle x\rangle$ and to $\langle x\rangle+\langle k\rangle$, respectively, and (b) holds because maximal ideals are prime.
\end{proof}
For an ideal $I$ of $R$, the \emph{Jacobson radical} $\jac(I)$ of $I$ is the intersection of the maximal ideals of $R$ containing $I$, with the convention $\jac(R)=R$. By \cite[Proposition~1.1]{AtiyahMacdonald}, we have $\jac(I)/I=J(R/I)$ for $I\neq R$.

Following Mac\'ias \cite{MaciasDomains} and Mac\'ias and Ortiz \cite[Proposition~3.8]{MaciasOrtiz}, the next proposition describes the closure of a point of $M(R)$ in terms of maximal supports.
\begin{proposition}\label{prop:closure}
Let $x\in R^0$. Then
\[
\clo{x}=\{y\in R^0:V(x)\subseteq V(y)\}=\jac(\langle x\rangle)^0.
\]
In particular, $\clo{x}=R^0$ if and only if $x\in U(R)$.
\end{proposition}

\begin{proof}
The second equality holds because $\jac(\langle x\rangle)$ is the intersection of the members of $V(x)$. If $x\in U(R)$, then $x$ lies in every set $\sigma_k^0$ and $V(x)=\emptyset$, so that all three sets equal $R^0$. Otherwise, $\clo{x}=\jac(\langle x\rangle)^0$ by \cite[Proposition~3.8]{MaciasOrtiz}, which, together with the second equality, yields the first.

Finally, by the first equality, $\clo{x}=R^0$ if and only if $V(x)\subseteq V(1)=\emptyset$. Now Lemma~\ref{lem:V}(a) implies that $x\in U(R)$.
\end{proof}
\begin{corollary}\label{cor:closureprime}
Let $x\in R^0\setminus U(R)$. Then $\clo{x}=\langle x\rangle^0$ if and only if $R/\langle x\rangle$ is semiprimitive.
\end{corollary}
\begin{proof}
By Proposition~\ref{prop:closure}, we have $\clo{x}=\jac(\langle x\rangle)^0$, so $\clo{x}=\langle x\rangle^0$ if and only if $\jac(\langle x\rangle)/\langle x\rangle=0$. Since $x$ is not a unit, the quotient $\jac(\langle x\rangle)/\langle x\rangle$ equals $J(R/\langle x\rangle)$, and the assertion follows.
\end{proof}
\begin{remark}\label{rem:MO}
(a) For a PID $R$, Mac\'ias \cite[Theorem~3.1]{MaciasDomains} showed that the closure of a prime element $p$ in $\widetilde{M}(R)$ is $\langle p\rangle$. Mac\'ias and Ortiz \cite[Corollary~3.11]{MaciasOrtiz} later asserted that the same holds for every prime element of an arbitrary integral domain. This assertion is false in general, and so is the analogous statement for $M(R)$. For instance, $x$ is a prime element of the local ring $R=\C\pow{x,y}$. Since two elements of $R$ generate $R$ if and only if at least one of them is a unit, the only open subsets of $\widetilde{M}(R)$ are $\emptyset$, $U(R)$ and $R$. Hence the closure of $x$ in $\widetilde{M}(R)$ is $R\setminus U(R)=\langle x,y\rangle$, and its closure in $M(R)$ is $\langle x,y\rangle^0$ by Proposition~\ref{prop:closure}.

(b) Corollary~\ref{cor:closureprime} answers \cite[Problem~3.18]{MaciasOrtiz}, which asks in which integral domains other than PIDs the closure of an irreducible element in $\widetilde{M}(R)$ is the ideal it generates. Indeed, let $p\in R^0\setminus U(R)$. The basic open sets of $\widetilde{M}(R)$ are $\sigma_1=R$, $\sigma_0=U(R)$ and the sets $\sigma_k=\sigma_k^0$ with $k\in R^0\setminus U(R)$. As the closure $\clo{p}$ of $p$ in $M(R)$ contains no unit (Proposition~\ref{prop:closure}), it follows that the closure of $p$ in $\widetilde{M}(R)$ is $\{0\}\cup\clo{p}$, which, by Corollary~\ref{cor:closureprime}, equals $\langle p\rangle$ if and only if $R/\langle p\rangle$ is semiprimitive. In particular, for a UFD $R$, the closure in $\widetilde{M}(R)$ of every irreducible element $p$ is $\langle p\rangle$ if and only if $R$ satisfies condition~\ref{cond:H}.
\end{remark}
\begin{corollary}\label{cor:units}
Let $h\colon M(R)\to M(S)$ be a homeomorphism. Then $h(U(R))=U(S)$ and consequently, $\#U(R)=\#U(S)$.
\end{corollary}
\begin{proof}
Since $h$ is a homeomorphism, $h(\clo{x})=\clo{h(x)}$ for every $x\in R^0$. Hence $h(U(R))=U(S)$ by the last assertion of Proposition~\ref{prop:closure}.
\end{proof}
The analogue of the preceding corollary for $\widetilde{M}(R)$ and $\widetilde{M}(S)$ is \cite[Theorem~5.5]{MaciasOrtiz}.
\begin{remark}\label{rem:T0}
By Proposition~\ref{prop:closure} and Lemma~\ref{lem:V}(b), $\clo{x}=\clo{x^2}$ for every $x\in R^0$, whereas $x\neq x^2$ if $x$ is not a unit. Hence $M(R)$ is not a $T_0$ space, extending the corresponding observation for PIDs in \cite{MaciasDomains} (see also \cite[Theorem~2.1(13)]{MaciasOrtiz}).
\end{remark}

We next establish three properties of the space $\K(R)$ introduced in Section~\ref{sec:intro}. They are used in the proofs of Lemma~\ref{lem:indist} and Corollary~\ref{cor:countableK}, and in Sections~\ref{sec:dedekind} and~\ref{sec:beyond}.

\begin{lemma}\label{lem:K}
\begin{enumerate}[label=\textup{(\alph*)}]
\item The collection $\{\W_k:k\in R^0\}$ is a basis for a topology on $\K(R)$. For every $k\in R^0$, $V^{-1}(\W_k)=\sigma_k^0$ and $V(\sigma_k^0)=\W_k$.
\item For all $E,F\in\K(R)$, one has $E\in\clo{F}$ if and only if $F\subseteq E$.
\item $R^0_\emptyset=U(R)$, and $R^0_E$ is infinite for every $E\in\K(R)\setminus\{\emptyset\}$.
\end{enumerate}
\end{lemma}

\begin{proof}
(a) By Lemma~\ref{lem:V}(b), $\W_k\cap\W_l=\W_{kl}$ for all $k,l\in R^0$, and $\W_1=\K(R)$. Therefore $\{\W_k:k\in R^0\}$ is a basis for a topology on $\K(R)$. By Lemma~\ref{lem:V}(c), $V^{-1}(\W_k)=\sigma_k^0$. Hence $V(\sigma_k^0)=\W_k$, as $V$ maps $R^0$ onto $\K(R)$.

(b) Write $E=V(x)$ and $F=V(y)$ with $x,y\in R^0$. By (a), the sets $\W_k$ form a basis of $\K(R)$, and each set $\W_k$ contains $E$ or $F$ if and only if $\sigma_k^0$ contains $x$ or $y$, respectively. Hence $E\in\clo{F}$ if and only if every set $\sigma_k^0$ containing $x$ contains $y$, that is, $x\in\clo{y}$; by Proposition~\ref{prop:closure}, this holds if and only if $F\subseteq E$.

(c) The first assertion follows from Lemma~\ref{lem:V}(a). If $E=V(x)\neq\emptyset$, then $x$ is not a unit, so the powers $x^n$, $n\geq1$, are pairwise distinct, and they lie in $R^0_E$ by Lemma~\ref{lem:V}(b).
\end{proof}
\begin{remark}\label{rem:kolmogorov}
For $k\in R$, let $D(k)=\{\m\in\Max R:k\notin\m\}$. By Lemma~\ref{lem:V}(c), $\sigma_k^0=V^{-1}(2^{D(k)})$ for every $k\in R^0$. Endow $2^{\Max R}$ with the topology generated by $\{2^{D(k)}:k\in R^0\}$. Thus the Mac\'ias topology is the coarsest topology on $R^0$ for which the map $V\colon R^0\to2^{\Max R}$ is continuous. Moreover, $\K(R)$ is a subspace of $2^{\Max R}$, as $\W_k=\K(R)\cap2^{D(k)}$. If $J(R)=0$ and $2^{\Max R}$ is instead endowed with the topology generated by $\{2^{D(k)}:k\in R\}$, then $V\colon M(R)\to2^{\Max R}$ is not continuous and $\K(R)$ is not a subspace of $2^{\Max R}$, since $2^{D(0)}=\{\emptyset\}$ and $V^{-1}(\{\emptyset\})=U(R)$ is not open in $M(R)$, by Theorem~\ref{thm:B}(i).
\end{remark}

Let $X$ be a topological space. For $x,y\in X$, we write $x\equiv y$ if $x$ and $y$ are topologically indistinguishable or, equivalently, if $\clo{x}=\clo{y}$. A quotient map $q\colon X\to K$ is called a \emph{Kolmogorov map} if, for all $x,y\in X$, the equality $q(x)=q(y)$ holds if and only if $x\equiv y$. We write $\operatorname{KQ}(X)$ for the Kolmogorov quotient $X/{\equiv}$ of $X$. The canonical quotient map $q_X\colon X\to\operatorname{KQ}(X)$ is a Kolmogorov map by the definition of $X/{\equiv}$. By \cite[Corollary~22.3]{Munkres}, every Kolmogorov map $q\colon X\to K$ induces a homeomorphism $\operatorname{KQ}(X)\to K$, given by $q_X(x)\mapsto q(x)$.

We now record some well-known properties of Kolmogorov maps needed in the proofs of Theorems~\ref{thm:A} and~\ref{cor:homeo} and Corollary~\ref{cor:ufd-kq}; these properties follow from elementary results on quotient maps~\cite[Section~22]{Munkres}.
\begin{fact}\label{fact:kolmogorov}
Let $q\colon X\to K$ and $q'\colon X'\to K'$ be Kolmogorov maps.
\begin{enumerate}[label=\textup{(\alph*)}]
\item Every homeomorphism $h\colon X\to X'$ induces a unique homeomorphism $\Phi\colon K\to K'$ with $q'\circ h=\Phi\circ q$, and $h$ maps $q^{-1}(t)$ bijectively onto $(q')^{-1}(\Phi(t))$ for every $t\in K$.
\item If $\Phi\colon K\to K'$ is a homeomorphism and $h\colon X\to X'$ is a bijection with $q'\circ h=\Phi\circ q$, then $h$ is a homeomorphism.
\end{enumerate}
\end{fact}

   Fact~\ref{fact:kolmogorov} yields the following homeomorphism criterion.
\begin{lemma}\label{property:kolmogorov}
Let $q\colon X\to K$ and $q'\colon X'\to K'$ be Kolmogorov maps. Then $X$ and $X'$ are homeomorphic if and only if there is a homeomorphism $\Phi\colon K\to K'$ with $\#q^{-1}(t)=\#(q')^{-1}(\Phi(t))$ for every $t\in K$.
\end{lemma}
\begin{proof}
If $X$ and $X'$ are homeomorphic, then such a $\Phi$ exists by Fact~\ref{fact:kolmogorov}(a). Conversely, let $\Phi$ be such a homeomorphism, and choose a bijection $h_t\colon q^{-1}(t)\to (q')^{-1}(\Phi(t))$ for each $t\in K$. As $t$ ranges over $K$, the fibers $q^{-1}(t)$ partition $X$ and, since $\Phi$ is bijective, the fibers $(q')^{-1}(\Phi(t))$ partition $X'$. Hence the maps $h_t$ together define a bijection $h\colon X\to X'$ with $q'\circ h=\Phi\circ q$, which is a homeomorphism by Fact~\ref{fact:kolmogorov}(b).
\end{proof}
The next lemma identifies the Kolmogorov quotient of $M(R)$ with the space $\K(R)$.
\begin{lemma}\label{lem:indist}
The map $V\colon M(R)\to\K(R)$ is a Kolmogorov map. In particular, two elements $x,y\in R^0$ are topologically indistinguishable in $M(R)$ if and only if $V(x)=V(y)$, and $V$ induces a homeomorphism $\operatorname{KQ}(M(R))\to\K(R)$.
\end{lemma}
\begin{proof}
It follows from Lemma~\ref{lem:K}(a) that a subset $U\subseteq M(R)$ is open if and only if $U=V^{-1}(W)$ for some open subset $W\subseteq\K(R)$. Consequently, the surjective map $V$ is a quotient map, and for $x,y\in R^0$ one has $x\equiv y$ in $M(R)$ if and only if $V(x)\equiv V(y)$ in $\K(R)$, which is equivalent to $V(x)=V(y)$, since $\K(R)$ is a $T_0$ space by Lemma~\ref{lem:K}(b). Thus $V$ is a Kolmogorov map.
\end{proof}
\begin{proof}[\normalfont\bfseries Proof of Theorem~\ref{thm:A}]
By Lemma~\ref{lem:indist}, the maps $V_R\colon M(R)\to\K(R)$ and $V_S\colon M(S)\to\K(S)$ are Kolmogorov maps, whose fibers over $E\in\K(R)$ and $F\in\K(S)$ are $R^0_E$ and $S^0_F$, respectively. Now the assertion follows from Lemma~\ref{property:kolmogorov}.
\end{proof}
The following lemma, together with Theorem~\ref{thm:A}, yields Corollary~\ref{cor:semilocal}.
\begin{lemma}\label{lem:semilocal}
Let $R$ be a semilocal domain that is not a field.
\begin{enumerate}[label=\textup{(\alph*)}]
\item $\K(R)=2^{\Max R}$, and the collection $\{2^Y:Y\subseteq\Max R\}$ forms a basis for the topology of $\K(R)$.
\item $\#R^0_E=\#R$ for every $E\in\K(R)$.
\end{enumerate}
\end{lemma}
\begin{proof}
The ideal $J(R)$ contains the product of the maximal ideals of $R$, which is non-zero, as $R$ is a domain that is not a field.

(a) Let $E\subseteq\Max R$. If $E=\Max R$, then $E=V(x)$ for every $x\in J(R)^0$. Otherwise, the Chinese remainder theorem \cite[Proposition~1.10(ii)]{AtiyahMacdonald} yields an element $x\in R^0$ that lies in every member of $E$ but in no member of $\Max R\setminus E$, so that $V(x)=E$. Thus $\K(R)=2^{\Max R}$, and hence the basic open sets $\W_k=2^{\Max R\setminus V(k)}$ with $k\in R^0$ are precisely the sets $2^Y$ with $Y\subseteq\Max R$.

(b) Let $E\in\K(R)$ and $x\in R^0_E$. Since the ideal $J(R)$ is contained in every maximal ideal of $R$, we have $V(x+y)=V(x)=E$ for every $y\in J(R)$ with $x+y\neq0$; that is, the set $R^0_E$ contains $(x+J(R))^0$. For $z\in J(R)^0$, the map $r\mapsto zr$ is an injection from $R$ into $J(R)$, so that $\#J(R)=\#R$ by \cite[Theorem~3.2]{Jech}. As $R$, not being a field, is infinite, the set $(x+J(R))^0$ thus has cardinality $\#R$, and hence so does $R^0_E$, again by \cite[Theorem~3.2]{Jech}.
\end{proof}
\begin{proof}[\normalfont\bfseries Proof of Corollary~\ref{cor:semilocal}]
If $\#R=\#S$ and $\#{\Max R}=\#{\Max S}$, then, by Lemma~\ref{lem:semilocal}(a), the assignment $E\mapsto\varphi(E)$ defines a homeomorphism $\Phi\colon\K(R)\to\K(S)$ for every bijection $\varphi\colon\Max R\to\Max S$, as $\Phi$ maps the basic open set $2^Y$ of $\K(R)$ onto the basic open set $2^{\varphi(Y)}$ of $\K(S)$ for every $Y\subseteq\Max R$. Consequently $M(R)\cong M(S)$ by Lemma~\ref{lem:semilocal}(b) and Theorem~\ref{thm:A}. Conversely, if $M(R)\cong M(S)$, then Theorem~\ref{thm:A} yields a homeomorphism $\Phi\colon\K(R)\to\K(S)$ with $\#R^0_E=\#S^0_{\Phi(E)}$ for every $E\in\K(R)$. By Lemma~\ref{lem:semilocal}(a), the finite sets $2^{\Max R}$ and $2^{\Max S}$ then have equal cardinality, so that $\#{\Max R}=\#{\Max S}$, and Lemma~\ref{lem:semilocal}(b) yields $\#R=\#S$.
\end{proof}
\begin{proof}[\normalfont\bfseries Proof of Corollary~\ref{cor:countableK}]
Every homeomorphism $\Phi\colon\K(R)\to\K(S)$ satisfies $\Phi(\emptyset)=\emptyset$, since, by Lemma~\ref{lem:K}(b), the empty set is the unique element of $\K(R)$ whose closure is $\K(R)$, and likewise for $\K(S)$. Therefore such a $\Phi$ satisfies $\#R^0_E=\#S^0_{\Phi(E)}$ for every $E\in\K(R)$ if and only if $\#U(R)=\#U(S)$, and the assertion follows from Theorem~\ref{thm:A}.
\end{proof}
\begin{proof}[\normalfont\bfseries Proof of Theorem~\ref{cor:homeo}]
By Lemma~\ref{lem:indist}, $V_R$ and $V_S$ are Kolmogorov maps. The assertion now follows from Fact~\ref{fact:kolmogorov}(a) and~(b).
\end{proof}
\section{Homeomorphisms between Mac\'ias spaces of Dedekind domains}\label{sec:dedekind}
In this section we prove Theorems~\ref{thm:C} and~\ref{cor:dedekind-homeo} as well as Corollaries~\ref{cor:torsion} and~\ref{cor:countable}. Let $R$ be a Dedekind domain that is not a field, so that the non-zero prime ideals of $R$ are precisely its maximal ideals. It follows from \cite[Corollary~9.4]{AtiyahMacdonald} that every non-zero ideal $I$ of $R$ factors uniquely as a product of maximal ideals; thus $I=\prod\limits_{\m\in\Max R}\m^{v_\m(I)}$, where the exponents $v_\m(I)\geq0$ are determined by $I$ and almost all zero. Since a prime ideal containing a product of ideals contains one of them, the maximal ideals containing $I$ are precisely those $\m$ with $v_\m(I)>0$; in particular, the set $V(x)$ is finite for every $x\in R^0$. Moreover, the ideal $I$ is principal if and only if its class $[I]=\sum_{\m}v_\m(I)[\m]$ vanishes in $\Cl(R)$.

To describe $\K(R)$ when $\Cl(R)$ is a torsion group (Lemma~\ref{lem:dedekind}(c)), we use the following topology on finite sets. For a set $X$, let $\Fin(X)$ denote the collection of finite subsets of $X$. For each $F\in\Fin(X)$, set $O_F=\{A\in\Fin(X):A\cap F=\emptyset\}$. Since $O_F\cap O_G=O_{F\cup G}$ for all $F,G\in\Fin(X)$ and $O_\emptyset=\Fin(X)$, the collection $\{O_F:F\in\Fin(X)\}$ is a basis for a topology $\lambda_X$ on $\Fin(X)$.\\
The \emph{lower topology} of a poset $(P,\leq)$ is the topology on $P$ generated by $\{P\setminus{\uparrow}a:a\in P\}$, where ${\uparrow}a=\{b\in P:a\leq b\}$. As $O_F=\bigcap\limits_{a\in F}O_{\{a\}}$ and $\Fin(X)\setminus{\uparrow}A=\bigcup\limits_{a\in A}O_{\{a\}}$ for all $A,F\in\Fin(X)$, the topology $\lambda_X$ is the lower topology of the poset $(\Fin(X),\subseteq)$.

 In view of Theorem~\ref{thm:A}, the next lemma describes the space $\K(R)$ and the cardinalities of the sets $R^0_E$.
\begin{lemma}\label{lem:dedekind}
Let $R$ be a Dedekind domain that is not a field.
\begin{enumerate}[label=\textup{(\alph*)}]
\item $\K(R)$ is the set of balanced subsets of $\Max R$.
\item $\#R^0_E=\#U(R)\cdot\aleph_0$ for every $E\in\K(R)\setminus\{\emptyset\}$.
\item If $\Cl(R)$ is a torsion group, then $\K(R)$ and $(\Fin(\Max R),\lambda_{\Max R})$ coincide as topological spaces.
\end{enumerate}
\end{lemma}

\begin{proof}
(a) As noted above, $V(x)$ is finite for every $x\in R^0$; hence it suffices to consider a finite set $E\subseteq\Max R$. For $x\in R^0$, one has $V(x)=E$ if and only if $\langle x\rangle=\prod\limits_{\m\in E}\m^{n_\m}$ for some integers $n_\m\geq1$. Hence $E\in\K(R)$ if and only if one of these products is principal, that is, if and only if $E$ is balanced.

(b) The map $y\mapsto\langle y\rangle$ sends $R^0_E$ into the countable set $\bigl\{\prod\limits_{\m\in E}\m^{n_\m}:\text{all }n_\m\geq1\bigr\}$, and the fibers of this map are $yU(R)$ with $y\in R^0_E$, each of cardinality $\#U(R)$. Thus $R^0_E$ is the disjoint union of at most countably many of these fibers, and it is non-empty since $E\in\K(R)$; hence $\#U(R)\leq\#R^0_E\leq\#U(R)\cdot\aleph_0$. Moreover, $R^0_E$ is infinite by Lemma~\ref{lem:K}(c), so that $\aleph_0\leq\#R^0_E$; together with $\#U(R)\leq\#R^0_E$, we obtain $\max\{\#U(R),\aleph_0\}\leq\#R^0_E$. If $U(R)$ is finite, then $\#U(R)\cdot\aleph_0=\aleph_0$; otherwise, $\#U(R)\cdot\aleph_0=\#U(R)$ by \cite[(3.14)]{Jech}. Thus $\#U(R)\cdot\aleph_0=\max\{\#U(R),\aleph_0\}$ in either case, and the Cantor-Bernstein theorem \cite[Theorem~3.2]{Jech} yields $\#R^0_E=\#U(R)\cdot\aleph_0$.

(c) If $\Cl(R)$ is a torsion group, then every finite set $E\subseteq\Max R$ is balanced, as one may take $n_\m$ to be the order of $[\m]$; hence, by (a), $\K(R)=\Fin(\Max R)$ as sets. Moreover, $\W_k=O_{V(k)}$ for every $k\in R^0$, and $V(k)$ runs through $\Fin(\Max R)$ as $k$ runs through $R^0$; thus the basis $\{\W_k:k\in R^0\}$ of $\K(R)$ (Lemma~\ref{lem:K}(a)) is the defining basis $\{O_F:F\in\Fin(\Max R)\}$ of $\lambda_{\Max R}$, so that the two topologies coincide.
\end{proof}

Note that Lemma~\ref{lem:dedekind}(b) allows us to remove the countability hypothesis from Corollary~\ref{cor:countableK} for Dedekind domains.

\begin{proposition}\label{prop:dedekind}
Let $R$ and $S$ be Dedekind domains that are not fields. Then $M(R)$ and $M(S)$ are homeomorphic if and only if $\#U(R)=\#U(S)$ and the spaces $\K(R)$ and $\K(S)$ are homeomorphic.
\end{proposition}

\begin{proof}
This follows from an argument analogous to the one in the proof of Corollary~\ref{cor:countableK}, using Lemma~\ref{lem:dedekind}(b) in place of the countability of $R$ and $S$.
\end{proof}

A subset $D$ of a poset $(P,\leq)$ is said to be \emph{downward closed} if $a\in D$ whenever $a\leq b$ for some $b\in D$, and \emph{directed} if $D$ is non-empty and any two elements of $D$ have an upper bound in $D$. Following \cite[Definition~5.1.45]{GoubaultLarrecq}, an \emph{ideal} of $P$ is a downward closed directed subset of $P$; it is \emph{proper} if it is different from $P$. The following lemma determines the maximal proper ideals of $(\Fin(X),\subseteq)$ for a set $X$.

\begin{lemma}\label{lem:max-ideals}
Let $X$ be a set, and, for $a\in X$, let $\mathcal{J}_a=\{E\in\Fin(X):a\notin E\}$. Then every proper ideal of the poset $(\Fin(X),\subseteq)$ is contained in $\mathcal{J}_a$ for some $a\in X$, and the assignment $a\mapsto\mathcal{J}_a$ is a bijection from $X$ onto the set of maximal proper ideals of $(\Fin(X),\subseteq)$.
\end{lemma}

\begin{proof}
For $a\in X$, the set $\mathcal{J}_a$ is downward closed, contains $\emptyset$ and the upper bound $E\cup F$ of any $E,F\in\mathcal{J}_a$; hence it is an ideal of $(\Fin(X),\subseteq)$, which is proper since $\{a\}\notin\mathcal{J}_a$. Moreover, for distinct $a,b\in X$, we have $\mathcal{J}_a\not\subseteq\mathcal{J}_b$, since $\{b\}\in\mathcal{J}_a\setminus\mathcal{J}_b$. Therefore, the assignment $a\mapsto\mathcal{J}_a$ is injective. 

Let $\mathcal{I}$ be an ideal of $(\Fin(X),\subseteq)$, and suppose that $\mathcal{I}$ is not contained in $\mathcal{J}_a$ for any $a\in X$, that is, every $a\in X$ belongs to some member of $\mathcal{I}$. Given $F\in\Fin(X)$, choose for each $a\in F$ a member of $\mathcal{I}$ containing $a$. Since $\mathcal{I}$ is directed, the finite set of these members has an upper bound $G\in\mathcal{I}$ \cite[p.~59]{GoubaultLarrecq}; then $F\subseteq G$, and hence $F\in\mathcal{I}$ because $\mathcal{I}$ is downward closed. Thus $\mathcal{I}=\Fin(X)$. Consequently, every proper ideal of $(\Fin(X),\subseteq)$ is contained in $\mathcal{J}_a$ for some $a\in X$.

Let $\mathcal{I}$ be a proper ideal, and choose $b\in X$ with $\mathcal{I}\subseteq\mathcal{J}_b$. If $\mathcal{I}$ is a maximal proper ideal, then $\mathcal{I}=\mathcal{J}_b$. If $\mathcal{I}$ contains $\mathcal{J}_a$ for some $a\in X$, then $\mathcal{J}_a\subseteq\mathcal{J}_b$, which forces $a=b$, as $\mathcal{J}_a\not\subseteq\mathcal{J}_b$ whenever $a\neq b$; hence $\mathcal{I}=\mathcal{J}_a$. Thus each $\mathcal{J}_a$ is a maximal proper ideal, and the injective map $a\mapsto\mathcal{J}_a$ is a bijection from $X$ onto the set of maximal proper ideals of $(\Fin(X),\subseteq)$.
\end{proof}
The following corollary determines the maximal proper ideals of $(\K(R),\subseteq)$.

\begin{corollary}\label{cor:max-ideals-K}
Let $R$ be an integral domain that is not a field, and suppose that $V(x)$ is finite for every $x\in R^0$. For $\m\in\Max R$, let $\mathcal J_{\m}
=\{E\in\Fin(\Max R):\m\notin E\}$. Then the map $\m\mapsto\mathcal{J}_\m\cap\K(R)$ is a bijection from $\Max R$ onto the set of maximal proper ideals of the poset $(\K(R),\subseteq)$.
\end{corollary}
\begin{proof}
By hypothesis, $\K(R)\subseteq\Fin(\Max R)$. For an ideal $\mathcal{I}$ of $(\K(R),\subseteq)$, let ${\downarrow}\mathcal{I}=\{E\in\Fin(\Max R):E\subseteq F\text{ for some }F\in\mathcal{I}\}$. This is an ideal of $(\Fin(\Max R),\subseteq)$ with ${\downarrow}\mathcal{I}\cap\K(R)=\mathcal{I}$; in particular, ${\downarrow}\mathcal{I}$ is proper if $\mathcal{I}$ is. For $\m\in\Max R$, the set $\mathcal{J}_\m\cap\K(R)$ is a proper ideal of $(\K(R),\subseteq)$, as it is downward closed, contains $\emptyset=V(1)$ and, by Lemma~\ref{lem:V}(b), the union of any two of its members, but does not contain $V(x)$ for $x\in\m^0$, where $\m\neq0$ as $R$ is not a field. Moreover, if $E\in\mathcal{J}_\m$ and $x_\mn\in\mn\setminus\m$ for each $\mn\in E$, then $y=\prod\limits_{\mn\in E}x_\mn$ lies in $R^0$, and, by Lemma~\ref{lem:V}(b), $V(y)=\bigcup\limits_{\mn\in E}V(x_\mn)$ contains $E$ but not $\m$, that is $E\in{\downarrow}(\mathcal{J}_\m\cap\K(R))$. Hence $\mathcal{J}_\m\subseteq{\downarrow}(\mathcal{J}_\m\cap\K(R))$, and the reverse inclusion holds as $\mathcal{J}_\m$ is downward closed. Thus ${\downarrow}(\mathcal{J}_\m\cap\K(R))=\mathcal{J}_\m$. 

Now let $\mathcal{I}$ be a proper ideal of $(\K(R),\subseteq)$. By Lemma~\ref{lem:max-ideals}, ${\downarrow}\mathcal{I}\subseteq\mathcal{J}_\m$ for some $\m\in\Max R$, that is, $\mathcal{I}\subseteq\mathcal{J}_\m\cap\K(R)$. If, moreover, $\mathcal{I}$ contains $\mathcal{J}_{\m'}\cap\K(R)$ for some $\m'\in\Max R$, then ${\downarrow}\mathcal{I}$ contains $\mathcal{J}_{\m'}$, so that ${\downarrow}\mathcal{I}=\mathcal{J}_{\m'}$ by Lemma~\ref{lem:max-ideals}, and hence $\mathcal{I}=\mathcal{J}_{\m'}\cap\K(R)$. Therefore the maximal proper ideals of $(\K(R),\subseteq)$ are precisely the ideals $\mathcal{J}_\m\cap\K(R)$ with $\m\in\Max R$; these are pairwise distinct, as ${\downarrow}(\mathcal{J}_\m\cap\K(R))=\mathcal{J}_\m$ and the sets $\mathcal{J}_\m$ are pairwise distinct by Lemma~\ref{lem:max-ideals}.
\end{proof}

We now apply Corollary~\ref{cor:max-ideals-K} to show that, for Dedekind domains $R$ and $S$, every homeomorphism $\K(R)\to\K(S)$ is induced by a bijection $\Max R\to\Max S$.

\begin{lemma}\label{lem:dedekind-K}
Let $R$ and $S$ be Dedekind domains that are not fields. A map $\Phi\colon\K(R)\to\K(S)$ is a homeomorphism if and only if there is a bijection $\varphi\colon\Max R\to\Max S$ such that $\{\varphi(E):E\in\K(R)\}=\K(S)$ and $\Phi(E)=\varphi(E)$ for every $E\in\K(R)$.
\end{lemma}
\begin{proof}
Let $\varphi$ be a bijection as in the statement. Then $\Phi$ is bijective. Note that if $k\in R^0$ and $l\in S^0$ satisfy $V_S(l)=\varphi(V_R(k))$, then $\Phi(\W_k)=\W_l$. Since $\{\varphi(V_R(k)):k\in R^0\}=\{\varphi(E):E\in\K(R)\}=\K(S)$, for every $k\in R^0$ there exists $l\in S^0$ such that $V_S(l)=\varphi(V_R(k))$, and vice versa. Hence $\Phi$ maps the basis $\{\W_k:k\in R^0\}$ of $\K(R)$ onto the basis $\{\W_l:l\in S^0\}$ of $\K(S)$, and is therefore a homeomorphism.

Conversely, let $\Phi$ be a homeomorphism. Since $\Phi$ preserves closures, Lemma~\ref{lem:K}(b)  implies that $\Phi$ is an isomorphism from the poset $(\K(R),\subseteq)$ onto $(\K(S),\subseteq)$. In particular, $\Phi$ maps maximal proper ideals onto maximal proper ideals. As $V_R(x)$ and $V_S(y)$ are finite for all $x\in R^0$ and $y\in S^0$ (see the beginning of this section), Corollary~\ref{cor:max-ideals-K} applies to both $R$ and $S$; hence there is a bijection $\varphi\colon\Max R\to\Max S$ such that $\Phi(\mathcal{J}_\m\cap\K(R))=\mathcal{J}_{\varphi(\m)}\cap\K(S)$ for every $\m\in\Max R$. Consequently, for $E\in\K(R)$ and $\m\in\Max R$, we have $\m\in E$ if and only if $\varphi(\m)\in\Phi(E)$. Since $\varphi$ is surjective, this yields $\Phi(E)=\varphi(E)$. Therefore $\{\varphi(E):E\in\K(R)\}=\Phi(\K(R))=\K(S)$.
\end{proof}

\begin{proof}[\normalfont\bfseries Proof of Theorem~\ref{thm:C}]
By Lemma~\ref{lem:dedekind}(a), a bijection $\varphi\colon\Max R\to\Max S$ satisfies $\{\varphi(E):E\in\K(R)\}=\K(S)$ if and only if it maps the balanced subsets of $\Max R$ onto those of $\Max S$. Hence, by Lemma~\ref{lem:dedekind-K}, the spaces $\K(R)$ and $\K(S)$ are homeomorphic if and only if such a bijection exists, and the assertion follows from Proposition~\ref{prop:dedekind}.
\end{proof}
\begin{proof}[\normalfont\bfseries Proof of Corollary~\ref{cor:torsion}]
Since $\Cl(R)$ and $\Cl(S)$ are torsion groups, every finite subset of $\Max R$ and of $\Max S$ is balanced, as observed in Section~\ref{sec:intro}. Hence every bijection $\Max R\to\Max S$ maps the balanced subsets of $\Max R$ onto those of $\Max S$. The assertion now follows from Theorem~\ref{thm:C}.
\end{proof}
\begin{proof}[\normalfont\bfseries Proof of Corollary~\ref{cor:countable}]
Since $R$ is countable and each $V(x)$ is finite, as observed at the beginning of this section, the set $\Max R=\bigcup\limits_{x\in R^0\setminus U(R)}V(x)$ is countable. Moreover, $J(R)=0$, so $\Max R$ is infinite by Theorem~\ref{thm:B}(ii). Thus $\#{\Max R}=\aleph_0$, and likewise $\#{\Max S}=\aleph_0$; hence the assertion follows from Corollary~\ref{cor:torsion}.
\end{proof}
\begin{proof}[\normalfont\bfseries Proof of Theorem~\ref{cor:dedekind-homeo}]
By Theorem~\ref{cor:homeo} and Lemma~\ref{lem:dedekind-K}, $h$ is a homeomorphism if and only if there is a bijection $\varphi\colon\Max R\to\Max S$ such that $V_S(h(x))=\varphi(V_R(x))$ for every $x\in R^0$ and $\{\varphi(E):E\in\K(R)\}=\K(S)$. If the first condition holds, then $\{\varphi(E):E\in\K(R)\}=\{\varphi(V_R(x)):x\in R^0\}=\{V_S(h(x)):x\in R^0\}=\K(S)$, where the last equality follows from the surjectivity of $h$. Thus the first condition implies the second, and the theorem follows.
\end{proof}
\begin{remark}\label{rem:dedekind-sharp}
(a) The Dedekind hypothesis cannot be omitted in Corollaries~\ref{cor:torsion} and~\ref{cor:countable}, nor can it be replaced by the hypothesis that $R$ and $S$ are UFDs. Indeed, $\Z$ and $\Z[x]$ are countable UFDs with two units each. It is well known that the maximal ideals of $\Z$ are the ideals $\langle p\rangle$, and those of $\Z[x]$ are the ideals $\langle p,f\rangle$, where $p$ is a prime number and $f\in\Z[x]$ is irreducible modulo~$p$. Hence $\#{\Max\Z}=\#{\Max\Z[x]}=\aleph_0$. Moreover, $\Z[x]$ is a Hilbert ring \cite[Theorem~4.19]{Eisenbud}, so it satisfies condition~\ref{cond:H}; it is also semiprimitive, since its zero ideal is prime and therefore an intersection of maximal ideals. Since $\Z[x]$ is not a PID, $M(\Z)\not\cong M(\Z[x])$ by Corollary~\ref{cor:pid-ufd}.

(b) For a Dedekind domain $R$ that is not a field, the group $\Cl(R)$ is generated by the classes of the maximal ideals, so it is a torsion group if and only if every finite subset of $\Max R$ is balanced, that is, by Lemma~\ref{lem:dedekind}(a), if and only if $\K(R)=\Fin(\Max R)$, as noted in \cite[Section~3]{Spirito}. Consequently, by Theorem~\ref{thm:C}, $M(R)\not\cong M(S)$ whenever $S$ is a Dedekind domain that is not a field and exactly one of $\Cl(R)$ and $\Cl(S)$ is a torsion group, in analogy with Spirito's result for Golomb spaces \cite[Theorem~3.4]{Spirito}.

(c) The torsion hypothesis cannot be omitted in Lemma~\ref{lem:dedekind}(c) and Corollary~\ref{cor:torsion}. Indeed, let $R=\C[x,y]/\langle y^2-x^3-x\rangle$ be the Dedekind domain of Remark~\ref{rem:ufd-sharp}(b), and let $S=\C[t]$. By Hilbert's Nullstellensatz, $\Max R=\{\langle\bar x-a,\bar y-b\rangle:a,b\in\C,\ b^2=a^3+a\}$ and $\Max S=\{\langle t-a\rangle:a\in\C\}$. Since every $a\in\C$ admits some $b\in\C$ with $b^2=a^3+a$, it follows that $\#{\Max R}=\#\C=\#{\Max S}$. Moreover, $U(S)=\C^\times$ and $\C^\times\subseteq U(R)\subseteq R$ with $\#R=\#\C$; therefore $\#U(R)=\#\C=\#U(S)$. However, $\Cl(S)$ is trivial, $S$ being a PID, whereas $\Cl(R)$ is not a torsion group, as it is isomorphic to the group of complex points of the elliptic curve $y^2=x^3+x$ \cite[Section~2.4]{ClarkElliptic}, which is uncountable but has a countable torsion subgroup \cite[Corollary~III.6.4(b)]{Silverman}. Hence, by~(b), $\K(R)\neq\Fin(\Max R)$ and $M(R)\not\cong M(S)$.
\end{remark}
\begin{example}\label{ex:examples}
The rings $\Z$, $\F_3[x]$ and $\Z[\sqrt{-5}]$ are countable Dedekind domains with two units each and class numbers $1$, $1$ and $2$, respectively. Indeed, the first two are PIDs, while $\Z[\sqrt{-5}]$ is the ring of integers of $\Q(\sqrt{-5})$, hence a Dedekind domain \cite[Theorem~9.5]{AtiyahMacdonald}, with unit group $\{\pm1\}$ and class number $2$ (see number field~2.0.20.1 in the $L$-functions and modular forms database). As these rings have only two units, they are semiprimitive \cite[Proposition~2.2 and Corollary~4.3]{Clark}. Therefore Corollary~\ref{cor:countable} yields $M(\Z)\cong M(\F_3[x])\cong M(\Z[\sqrt{-5}])$. Thus neither the class number nor the characteristic is invariant under homeomorphisms of Mac\'ias spaces.
\end{example}
\begin{remark}\label{rem:golomb}
The Golomb space of $\Z$ is not homeomorphic to that of any other Dedekind domain whose field of fractions is algebraic over $\Q$ \cite[Theorem~7.8]{Spirito}, whereas $M(\Z)\cong M(\Z[\sqrt{-5}])$ by Example~\ref{ex:examples}. Likewise, the Golomb space of $\Z$ admits only two self-homeomorphisms \cite[Theorem~7.7]{Spirito}, whereas, for each prime $p$, every permutation of the set $\Z^0_{\{p\Z\}}={\{x\in\Z^0:V(x)=\{p\Z\}\}}={\{\pm p^n:n\geq1\}}$ extends by the identity to a self-homeomorphism of $M(\Z)$ by Theorem~\ref{cor:dedekind-homeo} (for $p=2$, the self-homeomorphism interchanging $2$ and $4$ appears in \cite[Example~5.4]{MaciasOrtiz}). Consequently, $M(\Z)$ has uncountably many self-homeomorphisms, since the countably infinite set $\Z^0_{\{p\Z\}}$ has at least $2^{\aleph_0}$ permutations. Similarly, the Mac\'ias space of $\N$ admits non-trivial self-homeomorphisms \cite{MaciasSelf}, unlike the Golomb space of $\N$ (Section~\ref{sec:intro}).
\end{remark}
\section{Homeomorphisms between Mac\'ias spaces of UFDs}\label{sec:beyond}
In this section we prove Theorems~\ref{thm:ufd-general} and~\ref{thm:ufd-iff} as well as Corollaries~\ref{cor:pid-ufd}, \ref{cor:pid-main} and~\ref{cor:comax}. For a UFD $R$ and $x\in R^0$, we define the \emph{prime support} of $x$ by $\supp(x)=\{[p]\in\Pe_R:p\mid x\}$. Thus $\supp(x)\in\Fin(\Pe_R)$, and $\supp(x)=\emptyset$ if and only if $x\in U(R)$. By Lemma~\ref{lem:V}(a) and~(b),
\begin{equation}\label{eq:supp}
V(x)=\bigcup_{[p]\in\supp(x)}V(p)\qquad\text{for every }x\in R^0.
\end{equation}
The following proposition describes the Mac\'ias space of a UFD in terms of the comaximality relation among its prime elements.

\begin{proposition}\label{prop:ufd-comax}
Let $R$ be a UFD that is not a field.
\begin{enumerate}[label=\textup{(\alph*)}]
\item For $x,k\in R^0$, one has $x\in\sigma_k^0$ if and only if $\langle p\rangle+\langle q\rangle=R$ for all $[p]\in\supp(x)$ and $[q]\in\supp(k)$.
\item $R$ is a PID if and only if $\sigma_k^0=\{x\in R^0:\supp(x)\cap\supp(k)=\emptyset\}$ for every $k\in R^0$ \textup{(}cf.\ \cite[Corollary~3.17]{MaciasOrtiz}\textup{)}.
\end{enumerate}
\end{proposition}

\begin{proof}
(a) It follows from Lemma~\ref{lem:V}(c) and \eqref{eq:supp} that $x\in\sigma_k^0$ if and only if $V(p)\cap V(q)=\emptyset$ for all $[p]\in\supp(x)$ and $[q]\in\supp(k)$. Again by Lemma~\ref{lem:V}(c), $V(p)\cap V(q)=\emptyset$ if and only if $\langle p\rangle+\langle q\rangle=R$.

(b) By (a), every $x\in\sigma_k^0$ satisfies $\supp(x)\cap\supp(k)=\emptyset$. If $R$ is a PID, then $\langle p\rangle$ is a maximal ideal for every prime element $p$. Hence non-associated prime elements are comaximal. Hence, by (a), $x\in\sigma_k^0$ whenever $\supp(x)\cap\supp(k)=\emptyset$. Conversely, suppose that $\sigma_k^0=\{x\in R^0:\supp(x)\cap\supp(k)=\emptyset\}$ for every $k\in R^0$, let $I$ be a non-zero ideal of $R$, and choose, by the well-ordering of $\N$, an element $a\in I^0$ with the least number of prime factors, counted with multiplicity. Let $b\in I^0$, and let $d$ be a greatest common divisor of $a$ and $b$. Then $a/d$ and $b/d$ have disjoint prime supports, so that $\langle a/d\rangle+\langle b/d\rangle=R$ by hypothesis, and hence $\langle a\rangle+\langle b\rangle=\langle d\rangle$. Consequently, $d\in I^0$. Since $d$ divides $a$ and, by the choice of $a$, $d$ has at least as many prime factors as $a$, counted with multiplicity, the element $a/d$ has no prime factor, that is, $a/d$ is a unit. Hence $b\in\langle d\rangle=\langle a\rangle$. Thus $I=\langle a\rangle$, and $R$ is a PID.
\end{proof}

\begin{lemma}\label{lem:types}
Let $R$ be a UFD that is not a field, and let $x,y\in R^0$. Then $V(x)\subseteq V(y)$ if and only if $\mathcal{C}_R(y)\subseteq\mathcal{C}_R(x)$. In particular, $x$ and $y$ are topologically indistinguishable in $M(R)$ if and only if $\mathcal{C}_R(x)=\mathcal{C}_R(y)$.
\end{lemma}

\begin{proof}
By Proposition~\ref{prop:closure}, $V(x)\subseteq V(y)$ if and only if $y\in\clo{x}$, that is, if and only if every member of $\beta_{R^0}$ containing $y$ contains $x$. By \eqref{eq:supp} and Lemma~\ref{lem:V}(c), $\sigma_k^0=\bigcap\limits_{[q]\in\supp(k)}\sigma_q^0$ for every $k\in R^0$, with the convention that an intersection indexed by the empty set is $R^0$. Hence $V(x)\subseteq V(y)$ if and only if $x\in\sigma_q^0$ for every prime element $q$ with $y\in\sigma_q^0$, that is, if and only if $\mathcal{C}_R(y)\subseteq\mathcal{C}_R(x)$. The second assertion follows from the first and Lemma~\ref{lem:indist}.
\end{proof}
\begin{proof}[\normalfont\bfseries Proof of Theorem~\ref{thm:ufd-general}]
By Lemma~\ref{lem:types}, the map $\theta_R\colon\K(R)\to\mathfrak{C}(R)$, $V(x)\mapsto\mathcal{C}_R(x)$, is a well-defined bijection. By Lemma~\ref{lem:K}(a), the collection $\{\W_k:k\in R^0\}$ is a basis of $\K(R)$, and $\theta_R(\W_k)=\theta_R(V(\sigma_k^0))=\mathcal{O}_R(k)$ for every $k\in R^0$. Thus the collection $\{\mathcal{O}_R(k):k\in R^0\}$ is a basis for a topology on $\mathfrak{C}(R)$ for which $\theta_R$ is a homeomorphism. Moreover, $R^0_E=\{x\in R^0:\mathcal{C}_R(x)=\theta_R(E)\}$, and hence $\#R^0_E=\mu_R(\theta_R(E))$, for every $E\in\K(R)$. The same holds for $S$.

Let $\Phi\colon\mathfrak{C}(R)\to\mathfrak{C}(S)$ be a bijection, and let $\Psi=\theta_S^{-1}\circ\Phi\circ\theta_R$. Condition~(i) states that $\Phi$ and $\Phi^{-1}$ map basic open sets to open sets, which means that $\Phi$, or equivalently $\Psi$, is a homeomorphism. Since $\theta_R$ is surjective and $\#S^0_{\Psi(E)}=\mu_S(\Phi(\theta_R(E)))$ for every $E\in\K(R)$, condition~(ii) holds if and only if $\#R^0_E=\#S^0_{\Psi(E)}$ for every $E\in\K(R)$. As every bijection $\K(R)\to\K(S)$ is of the form $\theta_S^{-1}\circ\Phi\circ\theta_R$ for some bijection $\Phi\colon\mathfrak{C}(R)\to\mathfrak{C}(S)$, the assertion follows from Theorem~\ref{thm:A}.
\end{proof}
\begin{proof}[\normalfont\bfseries Proof of Corollary~\ref{cor:comax}]
By hypothesis, there are a bijection $\varphi\colon\Pe_R\to\Pe_S$ that preserves comaximality and a bijection $f\colon U(R)\to U(S)$. Choose sets $\Pi_R$ and $\Pi_S$ of representatives of $\Pe_R$ and $\Pe_S$, and let $\psi\colon\Pi_R\to\Pi_S$ be the bijection with $[\psi(p)]=\varphi([p])$ for every $p\in\Pi_R$. Each $x\in R^0$ factors uniquely as $x=u\prod\limits_{p\in\Pi_R}p^{a_p}$ with $u\in U(R)$ and integers $a_p\geq0$, almost all zero, and we define $H\colon R^0\to S^0$ by
\[
H(x)=f(u)\prod_{p\in\Pi_R}\psi(p)^{a_p}.
\]
Then $H$ is a bijection with $\supp(H(x))=\varphi(\supp(x))$ for every $x\in R^0$. By Proposition~\ref{prop:ufd-comax}(a), for all $x,k\in R^0$, one has $[p]\in\mathcal{C}_R(x)$ if and only if $\langle p\rangle+\langle q\rangle=R$ for every $[q]\in\supp(x)$, and hence $x\in\sigma_k^0$ if and only if $\supp(x)\subseteq\mathcal{C}_R(k)$. Thus $\mathcal{O}_R(k)$ consists of the sets $\mathcal{C}_R(x)$ with $x\in R^0$ and $\supp(x)\subseteq\mathcal{C}_R(k)$. The same holds for $S$. Since $\varphi$ preserves comaximality, it follows that $\mathcal{C}_S(H(x))=\varphi(\mathcal{C}_R(x))$ for every $x\in R^0$. Consequently, as $H$ is bijective, it induces a bijection $\Phi\colon\mathfrak{C}(R)\to\mathfrak{C}(S)$ with $\Phi(\mathcal{C}_R(x))=\mathcal{C}_S(H(x))$ for every $x\in R^0$. Then $\Phi(\mathcal{O}_R(k))=\mathcal{O}_S(H(k))$ for every $k\in R^0$, and hence $\Phi^{-1}(\mathcal{O}_S(l))=\mathcal{O}_R(H^{-1}(l))$ for every $l\in S^0$. Moreover, for every $C\in\mathfrak{C}(R)$, the map $H$ restricts to a bijection from $\{x\in R^0:\mathcal{C}_R(x)=C\}$ onto $\{y\in S^0:\mathcal{C}_S(y)=\Phi(C)\}$, so that $\mu_S(\Phi(C))=\mu_R(C)$. Thus $\Phi$ satisfies conditions~(i) and~(ii) of Theorem~\ref{thm:ufd-general}, and $M(R)$ and $M(S)$ are therefore homeomorphic.
\end{proof}

The following proposition characterizes condition~\ref{cond:H} in terms of maximal and prime supports.
\begin{proposition}\label{prop:J-ufd}
Let $R$ be a UFD that is not a field. The following conditions are equivalent:
\begin{enumerate}[label=\textup{(\roman*)}]
\item $R$ satisfies condition~\ref{cond:H};
\item for all $x,y\in R^0$, one has $V(x)\subseteq V(y)$ if and only if $\supp(x)\subseteq\supp(y)$.
\end{enumerate}
\end{proposition}

\begin{proof}
By \eqref{eq:supp}, the inclusion $\supp(x)\subseteq\supp(y)$ implies $V(x)\subseteq V(y)$ for all $x,y\in R^0$. 

Assume (i) holds, and let $x,y\in R^0$ with $V(x)\subseteq V(y)$. For every $[p]\in\supp(x)$, Lemma~\ref{lem:V}(b) yields $V(p)\subseteq V(x)\subseteq V(y)$. Therefore by Proposition~\ref{prop:closure} and Corollary~\ref{cor:closureprime}, $y\in\clo{p}=\langle p\rangle^0$. Hence $\supp(x)\subseteq\supp(y)$. Thus (i) implies (ii). 

Conversely, assume (ii) holds, and let $p$ be a prime element of $R$. By Proposition~\ref{prop:closure} and (ii), we have $\clo{p}=\{y\in R^0:[p]\in\supp(y)\}=\langle p\rangle^0$. Hence $R/\langle p\rangle$ is semiprimitive by Corollary~\ref{cor:closureprime}. Thus (ii) implies (i).
\end{proof}

For a UFD $R$, the collection $\{\Fin(\mathcal{C}_R(k)):k\in R^0\}$ forms a basis for a topology $\tau_R$ on $\Fin(\Pe_R)$. Indeed, $\Fin(\mathcal{C}_R(1))=\Fin(\Pe_R)$, and, for all $k,l\in R^0$, Lemma~\ref{lem:V}(b) and~(c) yield $\mathcal{C}_R(kl)=\mathcal{C}_R(k)\cap\mathcal{C}_R(l)$, and hence $\Fin(\mathcal{C}_R(kl))=\Fin(\mathcal{C}_R(k))\cap\Fin(\mathcal{C}_R(l))$. If $R$ also satisfies condition~\ref{cond:H}, Proposition~\ref{prop:J-ufd} gives the following description of the point closures and the Kolmogorov quotient of $M(R)$.

\begin{corollary}\label{cor:ufd-kq}
Let $R$ be a UFD that is not a field and satisfies condition~\ref{cond:H}.
\begin{enumerate}[label=\textup{(\alph*)}]
\item For every $x\in R^0\setminus U(R)$, one has $\clo{x}=\bigcap\limits_{[p]\in\supp(x)}\langle p\rangle^0$.
\item The prime support map $\supp\colon M(R)\to(\Fin(\Pe_R),\tau_R)$, $x\mapsto\supp(x)$, is a Kolmogorov map. Consequently, the map $V(x)\mapsto\supp(x)$ is a well-defined homeomorphism from $\K(R)$ onto $(\Fin(\Pe_R),\tau_R)$.
\end{enumerate}
\end{corollary}
\begin{proof}
(a) By Propositions~\ref{prop:closure} and~\ref{prop:J-ufd}(ii), an element $y\in R^0$ lies in $\clo{x}$ if and only if $\supp(x)\subseteq\supp(y)$, that is, if and only if $y\in\langle p\rangle^0$ for every $[p]\in\supp(x)$.

(b) By Proposition~\ref{prop:ufd-comax}(a), $\supp^{-1}(\Fin(\mathcal{C}_R(k)))=\sigma_k^0$ for every $k\in R^0$, so that the open subsets of $M(R)$ are precisely the sets $\supp^{-1}(U)$ with $U\in\tau_R$. Since every finite subset of $\Pe_R$ is the prime support of a product of representatives of its members, $\supp$ is surjective, and therefore a quotient map. By Propositions~\ref{prop:closure} and~\ref{prop:J-ufd}(ii), two elements $x,y\in R^0$ are topologically indistinguishable if and only if $\supp(x)=\supp(y)$. Thus $\supp$ is a Kolmogorov map. Finally, since $V\colon M(R)\to\K(R)$ is also a Kolmogorov map by Lemma~\ref{lem:indist}, the last assertion follows from Fact~\ref{fact:kolmogorov}(a), applied to the identity map of $M(R)$.
\end{proof}
\begin{remark}\label{rem:tau-lambda}
Let $R$ be as in Corollary~\ref{cor:ufd-kq}. If $R$ is a PID, then $\langle p\rangle$ is a maximal ideal for every prime element $p$ of $R$, so $\mathcal{C}_R(k)=\Pe_R\setminus\supp(k)$ and $\Fin(\mathcal{C}_R(k))=O_{\supp(k)}$ for every $k\in R^0$. As $\supp(k)$ runs through $\Fin(\Pe_R)$, it follows that $\tau_R$ is the topology $\lambda_{\Pe_R}$ of Section~\ref{sec:dedekind}. If $R$ is not a PID, then $\tau_R$ is strictly finer than $\lambda_{\Pe_R}$. Indeed, $\lambda_{\Pe_R}\subseteq\tau_R$, since $\supp$ is a quotient map by Corollary~\ref{cor:ufd-kq}(b) and, by Corollary~\ref{cor:ufd-kq}(a), the set $\supp^{-1}(O_F)=R^0\setminus\bigcup\limits_{[p]\in F}\clo{p}$ is open in $M(R)$ for every $F\in\Fin(\Pe_R)$. On the other hand, by Proposition~\ref{prop:pid-char} below, there is a prime element $q$ of $R$ such that $\Pe_R\setminus\mathcal{C}_R(q)$ is infinite. Given $F\in\Fin(\Pe_R)$, choose $a\in\Pe_R\setminus(\mathcal{C}_R(q)\cup F)$; then $\{a\}\in O_F\setminus\Fin(\mathcal{C}_R(q))$. Thus no set $O_F$ is contained in $\Fin(\mathcal{C}_R(q))$, and as $\emptyset\in\Fin(\mathcal{C}_R(q))$, it follows that $\Fin(\mathcal{C}_R(q))\in\tau_R\setminus\lambda_{\Pe_R}$.
\end{remark}
The next proposition characterizes the PIDs among the UFDs satisfying condition~\ref{cond:H}, in terms of both comaximality and the topology of their Mac\'ias spaces.
\begin{proposition}\label{prop:pid-char}
Let $R$ be a UFD that is not a field and satisfies condition~\ref{cond:H}. The following conditions are equivalent:
\begin{enumerate}[label=\textup{(\roman*)}]
\item $R$ is a PID;
\item for every prime element $q$ of $R$, the set $\Pe_R\setminus\mathcal{C}_R(q)$ is finite;
\item the topology of $M(R)$ is generated by the collection $\{R^0\setminus\clo{y}:y\in R^0\}$.
\end{enumerate}
\end{proposition}
\begin{proof}
Assume (i) holds. Then by Proposition~\ref{prop:ufd-comax}(b) and Corollary~\ref{cor:ufd-kq}(a), \[\sigma_k^0=\bigcap\limits_{[p]\in\supp(k)}\bigl(R^0\setminus\clo{p}\bigr),~\text{for every}~ k\in R^0.\] Thus every member of the basis $\beta_{R^0}$ of $M(R)$ is a finite intersection of sets of the form $R^0\setminus\clo{y}$. Since these sets are open in $M(R)$, (iii) holds.

Next, assume (iii) holds, and let $q$ be a prime element of $R$. By (iii), the finite intersections of members of $\{R^0\setminus\clo{y}:y\in R^0\}$ form a basis of $M(R)$. As $\sigma_q^0$ is an open set of $M(R)$ containing $1$, there are $y_1,\dots,y_r\in R^0$ such that $W:=\bigcap\limits_{i=1}^{r}\bigl(R^0\setminus\clo{y_i}\bigr)$ satisfies $1\in W\subseteq\sigma_q^0$. Since $1\notin\clo{y_i}$ for every $i$, the elements $y_1,\dots,y_r$ are non-units by Proposition~\ref{prop:closure}. Let $p$ be a prime element with $[p]\in\Pe_R\setminus\mathcal{C}_R(q)$. Then $p$ lies outside $\sigma_q^0$, and hence outside $W$, that is, $p\in\clo{y_i}$ for some $i$; by Corollary~\ref{cor:ufd-kq}(a), this forces $\supp(y_i)=\{[p]\}$. Consequently, $\Pe_R\setminus\mathcal{C}_R(q)$ is contained in the finite set $\supp(y_1)\cup\dots\cup\supp(y_r)$, and thus (ii) holds.

Finally, suppose that (ii) holds. let $q$ be a prime element of $R$ and $\m$ be a maximal ideal containing $q$. By (ii), $\Pe_R\setminus\mathcal{C}_R(q)=\{[p_1],\dots,[p_n]\}$ for some prime elements $p_1,\dots,p_n$. As $\m$ is a prime ideal containing $q$, every non-zero element of $\m$ has a prime factor $p\in\m$, which satisfies $\langle p\rangle+\langle q\rangle\subseteq\m$ and therefore $p\sim p_i$ for some $i$. Hence $\m\subseteq\langle p_1\rangle\cup\dots\cup\langle p_n\rangle$, and by prime avoidance \cite[Proposition~1.11]{AtiyahMacdonald}, $\m\subseteq\langle p_i\rangle$ for some $i$. Since $q\in\m$ , it follows that $p_i\sim q$ and $\m=\langle q\rangle$. Thus $\langle q\rangle$ is a maximal ideal for every prime element $q$ of $R$. Consequently, any two non-associated prime elements of $R$ are comaximal, and $R$ is a PID by Proposition~\ref{prop:ufd-comax}(a) and~(b).
\end{proof}
\begin{remark}\label{rem:pid-char-H}
For a UFD $R$ that is not a field, conditions (i) and (ii) of Proposition~\ref{prop:pid-char} are equivalent and imply (iii) even if $R$ does not satisfy condition~\ref{cond:H}. Indeed, a PID satisfies condition~\ref{cond:H}, so that (i) implies (ii) and (iii) by Proposition~\ref{prop:pid-char}, and the above proof that (ii) implies (i) does not use condition~\ref{cond:H}. On the other hand, (iii) does not imply (i) in general; for instance, the local UFD $\C\pow{x,y}$ satisfies (iii) but is not a PID (see Remark~\ref{rem:J-ufd-sharp}).
\end{remark}
\begin{proof}[\normalfont\bfseries Proof of Corollary~\ref{cor:pid-ufd}]
As noted in Section~\ref{sec:intro}, a PID that is not a field is a Dedekind domain with trivial class group; its maximal ideals are precisely the principal ideals with prime generators, which are unique up to associates. Corollary~\ref{cor:torsion} therefore reduces the assertion to showing that $M(R)\cong M(S)$ only if $S$ is a PID. As $\Cl(R)$ is a torsion group, it follows from the definition of the class group that, for every $\m\in\Max R$, the ideal $\m^{e_\m}$ is principal for some integer $e_\m\geq1$, say generated by $y_\m$; then $V(y_\m)=\{\m\}$, so that $\clo{y_\m}=\m^0$ by Proposition~\ref{prop:closure}.
Lemma~\ref{lem:V}(c) now yields $\sigma_c^0=\bigcap\limits_{\m\in V(c)}\bigl(R^0\setminus\clo{y_\m}\bigr)$ for every $c\in R^0$, where the empty intersection is $R^0$. Since each $V(c)$ is finite, as observed at the beginning of Section~\ref{sec:dedekind}, the above identity shows that the topology of $M(R)$ is generated by the open sets $R^0\setminus\clo{y}$ with $y\in R^0$. Since this property is preserved under homeomorphisms, $S$ is a PID by Proposition~\ref{prop:pid-char} whenever $M(R)\cong M(S)$.
\end{proof}
\begin{proof}[\normalfont\bfseries Proof of Corollary~\ref{cor:pid-main}]
As observed in Section~\ref{sec:intro}, every PID that is not a field is a Dedekind domain with trivial class group and a UFD satisfying condition~\ref{cond:H}. The assertion therefore follows from Corollary~\ref{cor:pid-ufd}.
\end{proof}
The following lemma, analogous to Lemma~\ref{lem:dedekind-K}, describes the homeomorphisms between the spaces $(\Fin(\Pe_R),\tau_R)$ for UFDs $R$ that are not fields and satisfy condition~\ref{cond:H}.
\begin{lemma}\label{lem:fin-homeo}
Let $R$ and $S$ be UFDs that are not fields and satisfy condition~\ref{cond:H}.
\begin{enumerate}[label=\textup{(\alph*)}]
\item For $Y\subseteq\Pe_R$, $\Fin(Y)\in\tau_R$ if and only if $Y$ is comaximally open.
\item A map $\Psi\colon(\Fin(\Pe_R),\tau_R)\to(\Fin(\Pe_S),\tau_S)$ is a homeomorphism if and only if there is a bijection $\varphi\colon\Pe_R\to\Pe_S$ preserving comaximally open sets such that $\Psi(A)=\varphi(A)$ for every $A\in\Fin(\Pe_R)$.
\end{enumerate}
\end{lemma}
\begin{proof}
(a) As the sets $\Fin(\mathcal{C}_R(k))$ with $k\in R^0$ form a basis of $\tau_R$ (Proposition~\ref{prop:J-ufd}), we have $\Fin(Y)\in\tau_R$ if and only if each $A\in\Fin(Y)$ satisfies $A\in\Fin(\mathcal{C}_R(k))\subseteq\Fin(Y)$ for some $k\in R^0$. Since $A\in\Fin(\mathcal{C}_R(k))\subseteq\Fin(Y)$ is equivalent to $A\subseteq\mathcal{C}_R(k)\subseteq Y$, it follows that $\Fin(Y)\in\tau_R$ if and only if $Y$ is comaximally open.

(b) Let $\varphi\colon\Pe_R\to\Pe_S$ be a bijection, preserving comaximally open sets such that $\Psi(A)=\varphi(A)$ for every $A\in\Fin(\Pe_R)$. Then $\Psi$ is a bijection with inverse $B\mapsto\varphi^{-1}(B)$. For every $k\in R^0$, the set $\varphi(\mathcal{C}_R(k))$ is comaximally open, as $\mathcal{C}_R(k)$ is; hence $\Psi(\Fin(\mathcal{C}_R(k)))=\Fin(\varphi(\mathcal{C}_R(k)))\in\tau_S$ by~(a). Thus $\Psi$ is an open map and, analogously, so is $\Psi^{-1}$. Therefore $\Psi$ is a homeomorphism.

Conversely, let $\Psi$ be a homeomorphism. By Proposition~\ref{prop:J-ufd}(ii), the homeomorphism $V(x)\mapsto\supp(x)$ of Corollary~\ref{cor:ufd-kq}(b) is also an isomorphism from the poset $(\K(R),\subseteq)$ onto $(\Fin(\Pe_R),\subseteq)$. Therefore, by Lemma~\ref{lem:K}(b), $\clo{A}=\{B\in\Fin(\Pe_R):A\subseteq B\}$ for every $A\in\Fin(\Pe_R)$, and similarly for $S$. Being a homeomorphism, $\Psi$ preserves closures, and is therefore an isomorphism from the poset $(\Fin(\Pe_R),\subseteq)$ onto $(\Fin(\Pe_S),\subseteq)$; in particular, it maps the singletons in $\Fin(\Pe_R)$ bijectively onto those in $\Fin(\Pe_S)$, as these are precisely the minimal non-empty elements of the two posets. Hence there is a bijection $\varphi\colon\Pe_R\to\Pe_S$ with $\Psi(\{a\})=\{\varphi(a)\}$ for every $a\in\Pe_R$, and $\Psi(A)=\varphi(A)$ for every $A\in\Fin(\Pe_R)$, as the singletons contained in $\Psi(A)$ are precisely the images under $\Psi$ of those contained in $A$. Finally, $\Psi(\Fin(Y))=\Fin(\varphi(Y))$ for every $Y\subseteq\Pe_R$; since $\Psi$ is a homeomorphism, $\varphi$ preserves comaximally open sets by~(a).
\end{proof}

\begin{proof}[\normalfont\bfseries Proof of Theorem~\ref{thm:ufd-iff}]
Following the proof of Theorem~\ref{thm:ufd-general}, we endow $\mathfrak{C}(R)$ with the topology generated by the basis $\{\mathcal{O}_R(k):k\in R^0\}$, under which the bijection $\theta_R\colon\K(R)\to\mathfrak{C}(R)$, $V(x)\mapsto\mathcal{C}_R(x)$, is a homeomorphism. Thus, by Corollary~\ref{cor:ufd-kq}(b), the map $\gamma_R\colon(\Fin(\Pe_R),\tau_R)\to\mathfrak{C}(R)$, $\supp(x)\mapsto\mathcal{C}_R(x)$, is a well-defined homeomorphism. The same applies to $S$.

By Theorem~\ref{thm:ufd-general}, $M(R)$ and $M(S)$ are homeomorphic if and only if there is a bijection $\Phi\colon\mathfrak{C}(R)\to\mathfrak{C}(S)$ satisfying conditions~(i) and~(ii) of the theorem. Condition~(i) is equivalent to $\Phi$ being a homeomorphism. Since $\gamma_R$ and $\gamma_S$ are homeomorphisms, $\Phi$ is a homeomorphism if and only if $\gamma_S^{-1}\circ\Phi\circ\gamma_R$ is one, which, by Lemma~\ref{lem:fin-homeo}(b), holds if and only if $\Phi\circ\gamma_R=\gamma_S\circ\varphi_*$ for some bijection $\varphi\colon\Pe_R\to\Pe_S$ preserving comaximally open sets, where $\varphi_*(A):=\varphi(A)\in\Fin(\Pe_S)$ for $A\in\Fin(\Pe_R)$. For such $\Phi$, condition~(ii) is equivalent to $\mu_S(\gamma_S(\varphi(A)))=\mu_R(\gamma_R(A))$ for every $A\in\Fin(\Pe_R)$, since $\gamma_R$ is surjective. Moreover, since $\gamma_R$ is injective and $\gamma_R(\supp(x))=\mathcal{C}_R(x)$ for every $x\in R^0$, we have $\mu_R(\gamma_R(A))=\#\{x\in R^0:\supp(x)=A\}$. By unique factorization, every $x\in R^0$ with $\supp(x)=A$ can be written uniquely as $x=u\prod\limits_{[p]\in A}p^{a_p}$ with $u\in U(R)$ and integers $a_p\geq1$, and every such product has prime support $A$, here $p$ denotes a fixed representative of the class $[p]$. Since $A$ is finite, it follows that $\mu_R(\gamma_R(A))$ equals $\#U(R)$ if $A=\emptyset$ and $\#U(R)\cdot\aleph_0$ otherwise; similarly for $S$.
 Since $\varphi(A)$ is empty precisely when $A$ is, condition~(ii) holds if and only if $\#U(R)=\#U(S)$. Thus the assertion follows.
\end{proof}

Observe that, for UFDs $R$ and $S$, every bijection $\Pe_R\to\Pe_S$ that preserves comaximality also preserves comaximally open sets, as mentioned in Section~\ref{sec:intro}. Indeed, by Proposition~\ref{prop:ufd-comax}(a), a prime element of $R$ is comaximal with $k\in R^0$ if and only if it is comaximal with every prime factor of $k$. Consequently, if a bijection $\varphi\colon\Pe_R\to\Pe_S$ preserves comaximality, then $\varphi(\mathcal{C}_R(k))=\mathcal{C}_S(l)$ for every $k\in R^0$, where $l\in S^0$ is a product of representatives of the classes in $\varphi(\supp(k))$. Hence $\varphi$ maps comaximally open sets to comaximally open sets, and so does $\varphi^{-1}$, which also preserves comaximality. In particular, for UFDs satisfying condition~\ref{cond:H}, Corollary~\ref{cor:comax} also follows from Theorem~\ref{thm:ufd-iff}. Similarly, Corollary~\ref{cor:pid-main} also follows from Theorem~\ref{thm:ufd-iff}. Indeed, PIDs satisfy condition~\ref{cond:H} and, as noted in Section~\ref{sec:intro}, for a PID $R$ the comaximally open subsets of $\Pe_R$ are precisely the cofinite ones; hence, for PIDs $R$ and $S$, every bijection $\Pe_R\to\Pe_S$ preserves comaximally open sets.
\begin{remark}\label{rem:J-ufd-sharp}
Condition~\ref{cond:H} cannot be omitted in Theorem~\ref{thm:ufd-iff}, Corollary~\ref{cor:pid-ufd} and Proposition~\ref{prop:pid-char}. Indeed, the PID $R=\C\pow{t}$ satisfies condition~\ref{cond:H}, whereas the UFD $S=\C\pow{x,y}$ does not (Section~\ref{sec:intro}). Both rings are local, and $\#R=\#\C^{\N}=\#\C^{\N\times\N}=\#S$, so that $M(R)\cong M(S)$ by Corollary~\ref{cor:semilocal}. However, there is no bijection $\Pe_R\to\Pe_S$, since $\#\Pe_R=1$, whereas $x$ and $y$ are non-associated prime elements of $S$. Hence the conclusion of Theorem~\ref{thm:ufd-iff} fails for $R$ and $S$, and so does that of Corollary~\ref{cor:pid-ufd}, as $R$ is a Dedekind domain with trivial class group and $S$ is not a PID. Finally, by the implication from (i) to (iii) in Proposition~\ref{prop:pid-char}, condition~(iii) holds for the PID $R$, and hence also for $S$, since it depends only on the topology of the Mac\'ias space. Thus the implication from (iii) to (i) fails without condition~\ref{cond:H}.

Consequently, for UFDs, neither condition~\ref{cond:H} nor the property of being a PID nor the cardinality of the set of associate classes of prime elements is invariant under homeomorphisms of Mac\'ias spaces, and the converse of Corollary~\ref{cor:comax} does not hold. For this reason, Theorem~\ref{thm:ufd-general} is formulated not in terms of bijections $\Pe_R\to\Pe_S$ but in terms of the sets $\mathcal{C}_R(x)$, which, by Lemma~\ref{lem:types}, correspond to the points of the Kolmogorov quotient of $M(R)$.
\end{remark}
\section{Proof of \texorpdfstring{Theorem~\ref{cor:polynomial}\,(a),\,(c) and (d)}{}}\label{sec:polynomial}
In this section, we prove parts (a), (c) and (d) of Theorem~\ref{cor:polynomial}. The proof of part (b) requires a detailed analysis of the Mac\'ias space $M(A[x])$ and is given separately in Section~\ref{sec:proof-b}.

A subset $\Sigma$ of a ring $R$ is said to be \emph{multiplicatively closed} if $1\in\Sigma$ and $st\in\Sigma$ for all $s,t\in\Sigma$. For such a set $\Sigma$, the \emph{localization} $\Sigma^{-1}R$ of $R$ with respect to $\Sigma$ is the ring of fractions $r/s$ with $r\in R$ and $s\in\Sigma$, where two fractions $r/s$ and $r'/s'$ are equal if and only if $t(rs'-r's)=0$ for some $t\in\Sigma$; the natural map $R\to\Sigma^{-1}R$, $r\mapsto r/1$, is a ring homomorphism. If $R$ is an integral domain and $0\notin\Sigma$, this homomorphism is injective, and we identify $\Sigma^{-1}R$ with the subring $\{r/s:r\in R,\ s\in\Sigma\}$ of the field of fractions of $R$. For an integral domain $R$ and $a\in R^0$, we write $R[1/a]$ for $\Sigma^{-1}R$ with $\Sigma=\{a^k:k\geq0\}$, so that $R[1/a]=\{r/a^k:r\in R,\ k\geq0\}$. For details, see \cite[Chapter~3]{AtiyahMacdonald}. If $A$ is a DVR with maximal ideal $\pi A$, then by \cite[Section~5.2, Theorem~1]{Gopalakrishnan} and \cite[Proposition~9.2]{AtiyahMacdonald}, every non-zero element of $A$ has the form $u\pi^e$ with $u\in U(A)$ and $e\geq0$; hence $A[1/\pi]$ is the field of fractions of $A$. 

The following algebraic lemma ensures that the ring $S$ in Theorem~\ref{cor:polynomial}(a) may be taken to be a polynomial ring in at least two variables over a DVR; we include its proof for lack of a suitable reference.

\begin{lemma}\label{lem:H-several}
For every DVR $A$ and every $n\geq2$, the UFD $A[x_1,\dots,x_n]$ satisfies condition~\ref{cond:H}.
\end{lemma}

\begin{proof}
Let $\pi\in A$ be such that $\Max A=\{\pi A\}$. Set $R=A[x_1,\dots,x_n]$ and $K=A[1/\pi]$, the field of fractions of $A$. By \cite[Section~9.1, Proposition~2]{DummitFoote} and induction on $n$, we have $R/\pi R\cong(A/\pi A)[x_1,\dots,x_n]$. If $a\in J(R/\pi R)$, then $1+ax_1\in U(R/\pi R)$ by~\eqref{eq:jacobson}, so $1+ax_1$ is a non-zero constant \cite[Section~9.1, Proposition~1]{DummitFoote}, which forces $a=0$; thus $R/\pi R$ is semiprimitive.

Let $f$ be a prime element of $R$ not associate to $\pi$. We now show that $D=R/fR$ is semiprimitive. Every prime element of $R$ lying in $A$ is a prime element of the local PID $A$ and hence an associate of $\pi$. Therefore the polynomial $f$ involves some variable $x_j$, and as $n\geq2$ we may choose $i\neq j$. Every non-zero multiple of $f$ involves $x_j$, so $fR\cap A[x_i]=0$. Hence the image $t$ of $x_i$ in $D$ is transcendental over $A$. Fix a non-zero $b\in D$ and an algebraic closure $\overline{K}$ of $K$. As $D$ is a finitely generated domain over its subring $A[t]$, it follows from \cite[Proposition~5.23]{AtiyahMacdonald} that there is a non-zero $u\in A[t]$ such that every homomorphism $A[t]\to\overline{K}$ not vanishing at $u$ extends to a homomorphism $D\to\overline{K}$ not vanishing at $b$. Since $u$ has only finitely many roots in $K$, we may choose $r\geq1$ with $u(\pi^{-r})\neq0$; then the $A$-algebra homomorphism $A[t]\to\overline{K}$, $t\mapsto\pi^{-r}$, extends to a homomorphism $\psi\colon D\to\overline{K}$ with $\psi(b)\neq0$. Then $\psi(D)$ contains $\psi(A)=A$ and $\pi^{-1}=\pi^{r-1}\psi(t)$, so $K=A[1/\pi]\subseteq\psi(D)\subseteq\overline{K}$. Since $\overline{K}$ is algebraic over $K$, the ring $\psi(D)$ is a field by \cite[Proposition~5.7]{AtiyahMacdonald}. Thus, $\ker\psi\in\Max D$. Since $b\notin\ker\psi$ and $b$ was arbitrary, it follows that $J(D)=0$.
\end{proof}
\begin{lemma}\label{lem:field-primes}
Let $F$ be a field. Then $\#\Pe_{F[x]}=\max\{\#F,\aleph_0\}$; in particular, the set $\Pe_{F[x]}$ is infinite.
\end{lemma}
\begin{proof}
The ring $F[x]$ is a PID \cite[Section~9.2, Corollary~4]{DummitFoote} that is not a field, and its units are the non-zero constants \cite[Section~9.1, Proposition~1]{DummitFoote}. If $a\in J(F[x])$, then $1+ax\in U(F[x])$ by~\eqref{eq:jacobson}, so that $1+ax$ is a non-zero constant, which forces $a=0$. Thus $F[x]$ is semiprimitive, and $\Pe_{F[x]}$ is infinite by Corollary~\ref{cor:pid-four}. As $\Pe_{F[x]}$ also contains the pairwise distinct classes of the degree-one polynomials $x-c$ with $c\in F$, it follows that $\max\{\#F,\aleph_0\}\leq\#\Pe_{F[x]}$. On the other hand, since every class in $\Pe_{F[x]}$ contains exactly one monic polynomial, the set $\Pe_{F[x]}$ is in bijection with a subset of $F[x]$, so that $\#\Pe_{F[x]}\leq\#F[x]=\max\{\#F,\aleph_0\}$ by \cite[Theorem~3.5 and (3.14)]{Jech}. By \cite[Theorem~3.2]{Jech}, these two inequalities yield the assertion.
\end{proof}
\begin{lemma}\label{lem:max-pi}
Let $A$ be a principal ideal domain and let $\pi\in A$ be such that $\Max A=\{\pi A\}$; put $\kappa=A/\pi A$ and $R=A[x]$. Identify each class in $\Pe_{\kappa[x]}$ with its monic representative $\varphi$, and fix a monic polynomial $\tilde\varphi\in R$ whose reduction modulo $\pi$ is $\varphi$. Then the maximal ideals of $R$ containing $\pi$ are the pairwise distinct ideals $\pi R+\tilde\varphi R$ with $\varphi\in\Pe_{\kappa[x]}$.
\end{lemma}
\begin{proof}
By \cite[Section~9.1, Proposition~2]{DummitFoote}, reduction modulo $\pi$ is a ring homomorphism from $R$ onto $\kappa[x]$ with kernel $\pi R$. Hence, by \cite[Proposition~1.1]{AtiyahMacdonald}, the maximal ideals of $R$ containing $\pi$ are the preimages of the maximal ideals of $\kappa[x]$. By \cite[Section~9.5, Proposition~15]{DummitFoote}, the latter are the pairwise distinct ideals $\varphi\kappa[x]$ with $\varphi\in\Pe_{\kappa[x]}$. Since reduction modulo $\pi$ maps $\tilde\varphi$ to $\varphi$, the preimage of $\varphi\kappa[x]$ is $\pi R+\tilde\varphi R$, and distinct ideals of $\kappa[x]$ have distinct preimages.
\end{proof}
\begin{proof}[\normalfont\bfseries Proof of Theorem~\ref{cor:polynomial}(a)]
We use the notation of Lemma~\ref{lem:max-pi}. Every $\m\in\Max R$ containing $\tilde\varphi$ contains $\pi$. Indeed, the field $R/\m$ is generated over its subring $A/(\m\cap A)$ by the image of $x$, which is a root of the monic polynomial $\tilde\varphi$. Hence $A/(\m\cap A)$ is a field by~\cite[Proposition~5.7]{AtiyahMacdonald}. Since $A$ is a DVR, we have $\m\cap A=\pi A$ and so $\pi\in\m$. Consequently, every maximal ideal of $R$ containing $\tilde\varphi$ contains the ideal $\pi R+\tilde\varphi R$, which is maximal by Lemma~\ref{lem:max-pi}. Thus $V_R(\tilde\varphi)=\{\pi R+\tilde\varphi R\}\subseteq V_R(\pi)$ for every $\varphi\in\Pe_{\kappa[x]}$. These singletons are minimal non-empty elements of the poset $(\K(R),\subseteq)$, and they are pairwise distinct by Lemma~\ref{lem:max-pi}. As $\Pe_{\kappa[x]}$ is infinite by Lemma~\ref{lem:field-primes}, it follows that $V_R(\pi)$ contains infinitely many minimal non-empty elements of $(\K(R),\subseteq)$.

On the other hand, let $S$ be a UFD satisfying condition~\ref{cond:H}. By Proposition~\ref{prop:J-ufd}, the assignment $V_S(y)\mapsto\supp(y)$ is an isomorphism from the poset $(\K(S),\subseteq)$ onto $(\Fin(\Pe_S),\subseteq)$, whose minimal non-empty elements are the singletons. Hence every element of $\K(S)$ contains only finitely many minimal non-empty elements of $(\K(S),\subseteq)$.

By Proposition~\ref{prop:closure}, every homeomorphism $h\colon M(R)\to M(S)$ satisfies $V_R(y)\subseteq V_R(z)$ if and only if $V_S(h(y))\subseteq V_S(h(z))$, for all $y,z\in R^0$, and therefore induces a poset isomorphism $V_R(y)\mapsto V_S(h(y))$ from $(\K(R),\subseteq)$ onto $(\K(S),\subseteq)$. Since $V_R(\pi)$ contains infinitely many minimal non-empty elements of $(\K(R),\subseteq)$, whereas no element of $\K(S)$ has this property, the two posets are not isomorphic, and hence the spaces $M(R)$ and $M(S)$ are not homeomorphic.
\end{proof}

The proofs of parts~(c) and~(d) rely on the following lemma, whose proof involves the comaximality graph of a UFD and a graph-theoretic property known as the extension property, which we now recall.

For a UFD $T$, the \emph{comaximality graph} of $T$ (cf.~\cite{SharmaBhatwadekar,MaimaniEtAl}) is the graph with vertex set $\Pe_T$ in which two classes $[p]$ and $[q]$ are adjacent if and only if $p$ and $q$ are comaximal. Thus, for UFDs $T$ and $T'$, the isomorphisms between the comaximality graphs of $T$ and $T'$ are precisely the bijections $\Pe_T\to\Pe_{T'}$ that preserve comaximality. 

A graph is said to have the \emph{extension property} (property~$(\ast)$ in \cite[Section~8.3]{Diestel}) if, for all finite disjoint subsets $\mathcal{U}$ and $\mathcal{V}$ of its vertex set, there is a vertex outside $\mathcal{U}\cup\mathcal{V}$ that is adjacent to every vertex in $\mathcal{U}$ and to no vertex in $\mathcal{V}$. By \cite[Theorem~8.3.1]{Diestel}, any two countable graphs with the extension property are isomorphic.
\begin{lemma}\label{lem:bijection}
Let $R$ and $S$ be polynomial rings given by either of the following:
\begin{enumerate}[label=\textup{(\roman*)}]
    \item $R = B[x_1, x_2, \dots]$ and $S = B'[x_1, x_2, \dots]$, where $B$ and $B'$ are at most countable UFDs;
    \item $R = k[x_1, \dots, x_n]$ and $S = k'[x_1, \dots, x_m]$, where $k$ and $k'$ are at most countable fields and $n, m \geq 2$.
\end{enumerate}
Then there exists a bijection $\Pe_R \to \Pe_S$ that preserves comaximally open sets.
\end{lemma}
\begin{proof}
In both cases, $R$ and $S$ are countable rings that are not fields; they are UFDs by \cite[Section~9.3, Theorem~7]{DummitFoote} in case~(ii). In case~(i), every element of $R$ lies in a subring $C=B[x_1,\dots,x_n]$ for some $n$, and $U(C)=U(B)$ \cite[Section~9.1, Proposition~1]{DummitFoote}; hence $U(R)=U(B)$. Moreover, $C$ is a UFD by \cite[Section~9.3, Theorem~7]{DummitFoote}, and every prime element $p$ of $C$ is prime in $R$, as $R/pR\cong(C/pC)[x_{n+1},x_{n+2},\dots]$ is a domain. Consequently, every non-zero non-unit of $R$ is a product of prime elements of $R$. Since a factorization into prime elements is unique up to order and associates in any integral domain, $R$ is a UFD, and so is $S$.

By the observation following the proof of Theorem~\ref{thm:ufd-iff}, every bijection $\Pe_R\to\Pe_S$ that preserves comaximality also preserves comaximally open sets. Since $\Pe_R$ and $\Pe_S$ are countable, it therefore suffices, by \cite[Theorem~8.3.1]{Diestel}, to show that the comaximality graphs of $R$ and $S$ have the extension property. We prove the claim for $R$; the argument for $S$ is analogous.

(i) Choose $n$ such that the classes in $\mathcal{U}\cup\mathcal{V}$ have representatives in $C=B[x_1,\dots,x_n]$, let $h\in C$ be a product of representatives of the classes in $\mathcal{U}$, with $h=1$ if $\mathcal{U}=\emptyset$, and put $z=1+hx_{n+1}$. As a primitive polynomial of degree one in $x_{n+1}$ over the UFD $B[x_i:i\neq n+1]$, the element $z$ is irreducible, hence prime. Since $z-hx_{n+1}=1$, this element is comaximal with every divisor of $h$; in particular, $[z]$ is adjacent to every class in $\mathcal{U}$. If $v\in C$ represents a class in $\mathcal{V}$, then $v$ does not divide $h$, as $\mathcal{U}\cap\mathcal{V}=\emptyset$. Hence the image of $z$ in the polynomial ring $R/vR\cong(C/vC)[x_{n+1},x_{n+2},\dots]$ over the domain $C/vC$ has degree one in $x_{n+1}$ and is therefore not a unit, so that $z$ and $v$ are not comaximal. Consequently, $[z]$ is adjacent to no class in $\mathcal{V}$. Finally, since $z\notin C$ and $U(R)=U(B)\subseteq C$, we have $[z]\notin\mathcal{U}\cup\mathcal{V}$. This proves the extension property in case~(i).

(ii) Let $h\in R$ be a product of representatives of the classes in $\mathcal{U}$, with $h=1$ if $\mathcal{U}=\emptyset$, and fix a representative $v$ of each class in $\mathcal{V}$. As $\mathcal{U}\cap\mathcal{V}=\emptyset$, no $v$ divides $h$. Since $R$ is a Hilbert ring \cite[Theorem~4.19]{Eisenbud}, it follows that for every $v$ there is $\m_v\in\Max R$ with $v\in\m_v$ and $h\notin\m_v$. By the Chinese remainder theorem \cite[Proposition~1.10(ii)]{AtiyahMacdonald}, there is a non-zero $a\in R$ with $a-1\in hR$ and $a\in\m_v$ for every $v$. 

Put $I=\bigcap\limits_v\m_v$, with $I=R$ if $\mathcal{V}=\emptyset$. For every prime element $q$ of $R$, an argument analogous to that in the proof of Lemma~\ref{lem:H-several} shows that $R/qR$ contains an element transcendental over $k$. Hence $R/qR$ is not a field \cite[Proposition~7.9]{AtiyahMacdonald}, so $qR$ is not maximal and, by \cite[Proposition~1.11(ii)]{AtiyahMacdonald}, does not contain $I$. By prime avoidance \cite[Proposition~1.11(i)]{AtiyahMacdonald}, there is a non-zero $c\in I$ with $c\notin qR$ for every prime factor $q$ of $a$. As $a-1\in hR$, the non-zero elements $a$ and $hc$ are then relatively prime.

By \cite[Theorem~1.3]{BodinDebesNajib}, there is $g\in R$ such that $z=a+hcg$ is irreducible, hence prime, and $\deg g>\max\{\deg a,\deg v\}$ for every $v$; then $\deg z=\deg(hc)+\deg g>\deg v$ for every $v$. Since $z-1\in hR$, the element $z$ is comaximal with every divisor of $h$, so that the class $[z]$ is adjacent to every class in $\mathcal{U}$. For every $v$, on the other hand, both $z$ and $v$ lie in $\m_v$, while $z$ has larger degree than $v$. Hence $[z]$ is neither adjacent nor equal to $[v]$. Consequently, the class $[z]$ lies outside $\mathcal{U}\cup\mathcal{V}$. This proves the extension property in case~(ii).
\end{proof}
\begin{proof}[\normalfont\bfseries Proof of Theorem~\ref{cor:polynomial}(c)]
Let $R=B[x_1,x_2,\dots]$ and $S=B'[x_1,x_2,\dots]$. We show that $R$ satisfies condition~\ref{cond:H}; the same argument applies to $S$. Let $p$ be a prime element of $R$, and choose $n$ with $p\in C=B[x_1,\dots,x_n]$, so that $R/pR\cong(C/pC)[x_{n+1},x_{n+2},\dots]$, with $C/pC$ a domain. If $g\in J(R/pR)$ and $m>n$ is such that $x_m$ does not occur in $g$, then, by~\eqref{eq:jacobson}, we have $1+gx_m\in U(R/pR)\cong U\bigl((C/pC)[x_{n+1},x_{n+2},\dots]\bigr)=U(C/pC)$, which forces $g=0$. Since $R$ and $S$ are UFDs that are not fields, with $U(R)=U(B)$ and $U(S)=U(B')$ (see the proof of Lemma~\ref{lem:bijection}), the assertion now follows from Lemma~\ref{lem:bijection}(i) and Theorem~\ref{thm:ufd-iff}.
\end{proof}
\begin{proof}[\normalfont\bfseries Proof of Theorem~\ref{cor:polynomial}(d)]
For every field $F$ and every $d\geq1$, the ring $F[x_1,\dots,x_d]$ is a UFD that is not a field, has unit group $F^\times$ and satisfies condition~\ref{cond:H} (see Section~\ref{sec:intro}), and it is a PID if and only if $d=1$ \cite[Section~8.2, Corollary~8, and Section~9.2, Corollary~4]{DummitFoote}. Moreover, for fields $F$ and $F'$, we have $\#F^\times=\#F'^\times$ if and only if $\#F=\#F'$. Consequently, if $n=m=1$, the assertion follows from Corollary~\ref{cor:pid-main} and Lemma~\ref{lem:field-primes}, and if exactly one of $n$ and $m$ equals $1$, it follows from Corollary~\ref{cor:pid-ufd}. Finally, if $n,m\geq2$, the assertion follows from Lemma~\ref{lem:bijection}(ii) and Theorem~\ref{thm:ufd-iff}. 
\end{proof}
\begin{remark}\label{rem:poly-arbitrary}
In the proof of Theorem~\ref{cor:polynomial}(d), the countability of $k$ and $k'$ is used only for $n,m\geq2$, through Lemma~\ref{lem:bijection}(ii). Hence, for $n=1$ or $m=1$, Theorem~\ref{cor:polynomial}(d) holds for arbitrary fields $k$ and $k'$.
\end{remark}
\section{Proof of \texorpdfstring{Theorem~\ref{cor:polynomial}(b)}{}}\label{sec:proof-b}
Throughout this section, $A$ is a DVR with maximal ideal $\pi A$ and residue field $\kappa=A/\pi A$, and $R=A[x]$, which is a UFD \cite[Section~9.3, Theorem~7]{DummitFoote}. For $y\in R$, we write $\bar y\in\kappa[x]$ for the reduction of $y$ modulo $\pi$. Since $A$ is a PID, Lemma~\ref{lem:max-pi} applies to $A$; as in that lemma, we identify each class in $\Pe_{\kappa[x]}$ with its monic representative $\varphi$ and fix a monic polynomial $\tilde\varphi\in R$ whose reduction modulo $\pi$ is $\varphi$. Let $K=A[1/\pi]$, which, as noted at the beginning of Section~\ref{sec:polynomial}, is the field of fractions of $A$; then $R[1/\pi]=K[x]$. We first determine the maximal ideals of $R$ (Lemma~\ref{lem:max-AX}) and then give explicit descriptions of $\mathfrak{C}(R)$, the sets $\mathcal{O}_R(k)$ and the function $\mu_R$ (Proposition~\ref{prop:param-AX}). These descriptions, together with Theorem~\ref{thm:ufd-general} and an order-theoretic argument given in Lemma~\ref{lem:order-AX}, yield Theorem~\ref{cor:polynomial}(b).
\begin{lemma}\label{lem:arith-AX}
The following statements hold. 
\begin{enumerate}[label=\textup{(\roman*)}]
\item The prime elements of $R$ are, up to associates, $\pi$ and the primitive polynomials that are irreducible in $K[x]$.
\item $\#U(R)=\#R=\#A$.
\item For every $\varphi\in\Pe_{\kappa[x]}$, the polynomial $\tilde\varphi$ is a prime element of $R$.
\end{enumerate}
\end{lemma}

\begin{proof}
(i) This follows from \cite[Section~9.3, Corollary~6]{DummitFoote}, since $R$ is a UFD and $\Pe_A=\{[\pi]\}$.

(ii) The map $a\mapsto1+\pi a$ is an injection of $A$ into $U(A)=U(R)$ and hence $\#A\leq\#U(R)\leq\#R$. On the other hand, we have $R=\bigcup\limits_{n\geq1}R_{<n}$, where $R_{<n}$ denotes the set of polynomials of degree less than $n$, and each $R_{<n}$ is in bijection with $A^n$. Since $A$ is infinite, it follows from \cite[Theorem~3.5]{Jech} that $\#A^n=\#A$ for every $n\geq1$. Therefore by \cite[(3.14)]{Jech}, $\#R\leq\aleph_0\cdot\#A=\#A$. By the Cantor--Bernstein theorem \cite[Theorem~3.2]{Jech}, these inequalities yield~(ii).

(iii) Let $\varphi\in\Pe_{\kappa[x]}$. The polynomial $\tilde\varphi$ is monic, and its reduction modulo $\pi$ is the irreducible polynomial $\varphi$; hence $\tilde\varphi$ is irreducible in $R$ by \cite[Section~9.4, Proposition~12]{DummitFoote}, and thus a prime element of $R$, as $R$ is a UFD.
\end{proof}

Let $\mathcal{H}:=\{[f]\in\Pe_R:\bar f\in U(\kappa[x])\}$, the set of classes of the prime elements of $R$ whose reductions modulo $\pi$ are non-zero constants. 
\begin{lemma}\label{lem:max-AX}
$\Max R=\{\pi R+\tilde\varphi R:\varphi\in\Pe_{\kappa[x]}\}\cup\{fR:[f]\in\mathcal{H}\}$.
\end{lemma}

\begin{proof}
By Lemma~\ref{lem:max-pi}, the maximal ideals of $R$ containing $\pi$ are the ideals $\pi R+\tilde\varphi R$ with $\varphi\in\Pe_{\kappa[x]}$. It remains to show that the maximal ideals not containing $\pi$ are precisely the ideals $fR$ with $[f]\in\mathcal{H}$. Let $[f]\in\mathcal{H}$. Since $\bar f$ is a non-zero constant, we have $f=a+\pi g$ for some $a\in A\setminus\pi A=U(A)$ and $g\in R$; hence $fR+\pi R=R$, that is, $\pi$ is invertible modulo $f$. Consequently, we have $\pi\notin fR$, and the natural map $R/fR\rightarrow R[1/\pi]/fR[1/\pi]=K[x]/fK[x]$ is an isomorphism. Moreover, the prime element $f$ is not an associate of $\pi$, because $\bar f\neq0=\bar\pi$; by Lemma~\ref{lem:arith-AX}(i), it is therefore irreducible in $K[x]$. Therefore the ring $K[x]/fK[x]$ is a field, and thus $fR\in\Max R$.

Conversely, let $\m\in\Max R$ with $\pi\notin\m$. Since $\m$ is a prime ideal not containing $\pi$, it does not meet the multiplicatively closed set $\Sigma=\{\pi^k:k\geq0\}$. Hence, by \cite[Proposition~3.11(iv)]{AtiyahMacdonald}, the ideal $\Sigma^{-1}\m=\m K[x]$ of the PID $K[x]=R[1/\pi]$ is non-zero and prime, say $\m K[x]=fK[x]$ with $f$ irreducible in $K[x]$. Since $K=A[1/\pi]$ and, as noted at the beginning of Section~\ref{sec:polynomial}, every non-zero element of $A$ has the form $u\pi^e$ with $u\in U(A)$ and $e\geq0$, each non-zero coefficient of $f$ has this form with $e\in\Z$. Let $n$ be the least of these exponents, so that $\pi^{-n}f$ has coefficients in $A$, at least one of which is a unit. Replacing $f$ by its associate $\pi^{-n}f$ in $K[x]$, we may therefore assume that $f$ is a primitive polynomial in $R$. By Lemma~\ref{lem:arith-AX}(i), such an $f$ is prime in $R$ and does not divide $\pi$; as every element of $K[x]$ has the form $r/\pi^k$ with $r\in R$, we obtain $fK[x]\cap R=fR$. Hence we have $\m\subseteq\m K[x]\cap R=fR$, and therefore $\m=fR$ by the maximality of $\m$. Since $fR$ is maximal and $\pi\notin fR$, no ideal $\pi R+\tilde\varphi R$ contains $f$; thus no $\varphi\in\Pe_{\kappa[x]}$ divides $\bar f$, which is non-zero as $f$ is primitive, and it follows that $\bar f\in U(\kappa[x])$, that is, $[f]\in\mathcal{H}$.
\end{proof}

For $y\in R^0$, put $\mathcal{H}(y):=\supp(y)\cap\mathcal{H}$, which is finite, and $\mathcal{Q}(y):=\{\varphi\in\Pe_{\kappa[x]}:\varphi\mid\bar y\}$, so that $\mathcal{Q}(y)$ equals $\Pe_{\kappa[x]}$ if $\pi\mid y$ and is finite otherwise. For $[f]\in\mathcal{H}$ and $\varphi\in\Pe_{\kappa[x]}$, we have
\begin{equation}\label{eq:HQ}
\mathcal{H}(f)=\{[f]\},\quad\mathcal{Q}(f)=\emptyset,\quad\mathcal{H}(\tilde\varphi)=\emptyset\quad\text{and}\quad\mathcal{Q}(\tilde\varphi)=\{\varphi\},
\end{equation}
where the third equality holds because $\tilde\varphi$ is a prime element of $R$ by Lemma~\ref{lem:arith-AX}(iii) and $[\tilde\varphi]\notin\mathcal{H}$, as $\varphi\notin U(\kappa[x])$.

\begin{lemma}\label{lem:comax-AX}
Two elements $y,z\in R^0$ are comaximal if and only if $\mathcal{H}(y)\cap\mathcal{H}(z)=\emptyset$ and $\mathcal{Q}(y)\cap\mathcal{Q}(z)=\emptyset$.
\end{lemma}
\begin{proof}
Let $w\in R^0$, $[f]\in\mathcal{H}$ and $\varphi\in\Pe_{\kappa[x]}$. Since $f$ is a prime element of $R$, we have $w\in fR$ if and only if $[f]\in\mathcal{H}(w)$. Also $w\in\pi R+\tilde\varphi R$ if and only if $\varphi\in\mathcal{Q}(w)$, as $\pi R+\tilde\varphi R$ is the preimage of $\varphi\kappa[x]$ under reduction modulo $\pi$. Combining these equivalences for $w=y$ and $w=z$ with Lemma~\ref{lem:max-AX}, we see that the maximal ideals of $R$ containing both $y$ and $z$ are exactly the ideals $fR$ with $[f]\in\mathcal{H}(y)\cap\mathcal{H}(z)$ and the ideals $\pi R+\tilde\varphi R$ with $\varphi\in\mathcal{Q}(y)\cap\mathcal{Q}(z)$. By Lemma~\ref{lem:V}(c), $y$ and $z$ are comaximal if and only if no maximal ideal of $R$ contains both, and the lemma follows. 
\end{proof}
Let $\mathfrak{E}:=\Fin(\Pe_{\kappa[x]})\cup\{\Pe_{\kappa[x]}\}$, so that $(\mathcal{H}(y),\mathcal{Q}(y))\in\Fin(\mathcal{H})\times\mathfrak{E}$ for every $y\in R^0$. For $k\in R^0$, put
\[
\mathcal{B}_k:=\{(\Gamma,\Delta)\in\Fin(\mathcal{H})\times\mathfrak{E}:\Gamma\cap\mathcal{H}(k)=\emptyset\text{ and }\Delta\cap\mathcal{Q}(k)=\emptyset\}.
\]
\begin{proposition}\label{prop:param-AX}
The following statements hold.
\begin{enumerate}[label=\textup{(\roman*)}]
\item The assignment $\mathcal{C}_R(y)\mapsto(\mathcal{H}(y),\mathcal{Q}(y))$ is a well-defined bijection from $\mathfrak{C}(R)$ onto $\Fin(\mathcal{H})\times\mathfrak{E}$, which maps $\mathcal{O}_R(k)$ onto $\mathcal{B}_k$ for every $k\in R^0$.
\item The family $\{\mathcal{B}_k:k\in R^0\}$ is a basis of a topology on $\Fin(\mathcal{H})\times\mathfrak{E}$.
\item $\mu_R(C)=\#A$ for every $C\in\mathfrak{C}(R)$, and $\#\mathcal{H}=\#A$.
\end{enumerate}
\end{proposition}

\begin{proof}
(i) Let $y\in R^0$. By Lemma~\ref{lem:comax-AX}, the set $\mathcal{C}_R(y)$ consists of the classes $[p]\in\Pe_R$ with $\mathcal{H}(p)\cap\mathcal{H}(y)=\emptyset$ and $\mathcal{Q}(p)\cap\mathcal{Q}(y)=\emptyset$. Moreover the same lemma and \eqref{eq:HQ} give $\mathcal{H}(y)=\mathcal{H}\setminus\mathcal{C}_R(y)$ and $\mathcal{Q}(y)=\{\varphi\in\Pe_{\kappa[x]}:[\tilde\varphi]\notin\mathcal{C}_R(y)\}$. Consequently, for $y,y'\in R^0$, we have $\mathcal{C}_R(y)=\mathcal{C}_R(y')$ if and only if $(\mathcal{H}(y),\mathcal{Q}(y))=(\mathcal{H}(y'),\mathcal{Q}(y'))$; thus the assignment $\mathcal{C}_R(y)\mapsto(\mathcal{H}(y),\mathcal{Q}(y))$ is a well-defined injection from $\mathfrak{C}(R)$ into $\Fin(\mathcal{H})\times\mathfrak{E}$. To prove surjectivity, let $(\Gamma,\Delta)\in\Fin(\mathcal{H})\times\mathfrak{E}$, let $g$ be a product of representatives of the classes in $\Gamma$, and put $y:=g\prod\limits_{\varphi\in\Delta}\tilde\varphi$ if $\Delta$ is finite and $y:=g\pi$ otherwise. Since $\bar g\in U(\kappa[x])$, while neither $[\pi]$ nor any $[\tilde\varphi]$ lies in $\mathcal{H}$, it follows that $(\mathcal{H}(y),\mathcal{Q}(y))=(\Gamma,\Delta)$. Finally, by Lemma~\ref{lem:comax-AX}, the resulting bijection maps $\mathcal{O}_R(k)$ onto $\mathcal{B}_k$ for every $k\in R^0$.

(ii) This follows from $\mathcal{B}_1=\Fin(\mathcal{H})\times\mathfrak{E}$ and $\mathcal{B}_k\cap\mathcal{B}_{k'}=\mathcal{B}_{kk'}$ for $k,k'\in R^0$.

(iii) Let $C\in\mathfrak{C}(R)$. The set $\{y\in R^0:\mathcal{C}_R(y)=C\}$ is non-empty, and it is invariant under multiplication by units of $R$, because $\mathcal{C}_R(uy)=\mathcal{C}_R(y)$ for all $y\in R^0$ and $u\in U(R)$. Hence its cardinality $\mu_R(C)$ satisfies $\#U(R)\leq\mu_R(C)\leq\#R$. Since $\#U(R)=\#R=\#A$ by Lemma~\ref{lem:arith-AX}(ii), we obtain $\mu_R(C)=\#A$ by \cite[Theorem~3.2]{Jech}. Moreover, for $a\in A^0$, the polynomial $1+\pi ax$ is primitive of degree one, hence prime in $R$ by Lemma~\ref{lem:arith-AX}(i), and its reduction modulo $\pi$ is $1$; thus $[1+\pi ax]\in\mathcal{H}$. These classes are pairwise distinct, because $U(R)=U(A)$ and the polynomials $1+\pi ax$ all have constant term $1$. Consequently, we have $\#A=\#A^0\leq\#\mathcal{H}\leq\#R=\#A$, so that $\#\mathcal{H}=\#A$ by \cite[Theorem~3.2]{Jech}.
\end{proof}

An element $a$ of a poset is said to lie \emph{below} an element $b$, and $b$ is said to lie \emph{above} $a$, if $a\leq b$. The following lemma shows that the homeomorphism type of the space $\Fin(\mathcal{H})\times\mathfrak{E}$ determines the cardinality of $\Pe_{\kappa[x]}$.
\begin{lemma}\label{lem:order-AX}
Let $A'$ be a DVR with residue field $\kappa'$, let $S=A'[x]$, and let $\mathcal{H}'$, $\mathfrak{E}'$ and $\mathcal{B}'_l$, for $l\in S^0$, be the analogues for $S$ of $\mathcal{H}$, $\mathfrak{E}$ and $\mathcal{B}_k$. Endow $\Fin(\mathcal{H})\times\mathfrak{E}$ and $\Fin(\mathcal{H}')\times\mathfrak{E}'$ with the topologies having $\{\mathcal{B}_k:k\in R^0\}$ and $\{\mathcal{B}'_l:l\in S^0\}$, respectively, as bases. If these two spaces are homeomorphic, then $\#\Pe_{\kappa[x]}=\#\Pe_{\kappa'[x]}$.
\end{lemma}
\begin{proof}
Let $\Psi\colon\Fin(\mathcal{H})\times\mathfrak{E}\to\Fin(\mathcal{H}')\times\mathfrak{E}'$ be a homeomorphism, and take arbitrary $(\Gamma,\Delta),(\Gamma',\Delta')\in\Fin(\mathcal{H})\times\mathfrak{E}$. If $\Gamma'\subseteq\Gamma$ and $\Delta'\subseteq\Delta$, then every set $\mathcal{B}_k$ containing $(\Gamma,\Delta)$ also contains $(\Gamma',\Delta')$. Otherwise, there is $[f]\in\Gamma'\setminus\Gamma$ or $\varphi\in\Delta'\setminus\Delta$, and then, by \eqref{eq:HQ}, the set $\mathcal{B}_f$ or $\mathcal{B}_{\tilde\varphi}$, respectively, contains $(\Gamma,\Delta)$ but not $(\Gamma',\Delta')$. Hence we have $(\Gamma,\Delta)\in\overline{\{(\Gamma',\Delta')\}}$ if and only if $\Gamma'\subseteq\Gamma$ and $\Delta'\subseteq\Delta$. The analogous equivalence holds in $\Fin(\mathcal{H}')\times\mathfrak{E}'$. Since $\Psi$ preserves closures, it follows that $\Psi$ is an order isomorphism with respect to componentwise inclusion.

For this order, the minimal elements of $(\Fin(\mathcal{H})\times\mathfrak{E})\setminus\{(\emptyset,\emptyset)\}$ are the pairs $(\{[f]\},\emptyset)$ with $[f]\in\mathcal{H}$ and $(\emptyset,\{\varphi\})$ with $\varphi\in\Pe_{\kappa[x]}$. By Lemma~\ref{lem:field-primes}, the set $\Pe_{\kappa[x]}$ is the only infinite member of $\mathfrak{E}$. Since the first component of every pair in $\Fin(\mathcal{H})\times\mathfrak{E}$ is finite, a pair lies above infinitely many of these minimal elements if and only if its second component is $\Pe_{\kappa[x]}$. Thus the least such pair is $(\emptyset,\Pe_{\kappa[x]})$, and the minimal elements below it are exactly the pairs $(\emptyset,\{\varphi\})$ with $\varphi\in\Pe_{\kappa[x]}$. As these characterizations are purely order-theoretic and the analogous ones hold in $\Fin(\mathcal{H}')\times\mathfrak{E}'$, the order isomorphism $\Psi$ maps $(\emptyset,\Pe_{\kappa[x]})$ to $(\emptyset,\Pe_{\kappa'[x]})$ and the minimal elements below the former bijectively onto those below the latter. Consequently, we have $\#\Pe_{\kappa[x]}=\#\Pe_{\kappa'[x]}$.
\end{proof}
\begin{proof}[\normalfont\bfseries Proof of Theorem~\ref{cor:polynomial}(b)]
Put $S=A'[x]$, and let $\mathcal{H}'$, $\mathfrak{E}'$ and $\mathcal{B}'_l$, for $l\in S^0$, be as in Lemma~\ref{lem:order-AX}. First assume that $\#A=\#A'$ and $\max\{\#\kappa,\aleph_0\}=\max\{\#\kappa',\aleph_0\}$. Then Proposition~\ref{prop:param-AX}(iii) and Lemma~\ref{lem:field-primes} give $\#\mathcal{H}=\#A=\#A'=\#\mathcal{H}'$ and $\#\Pe_{\kappa[x]}=\max\{\#\kappa,\aleph_0\}=\max\{\#\kappa',\aleph_0\}=\#\Pe_{\kappa'[x]}$ respectively; hence there exist bijections $\alpha\colon\mathcal{H}\to\mathcal{H}'$ and $\beta\colon\Pe_{\kappa[x]}\to\Pe_{\kappa'[x]}$. The definitions of the sets $\mathcal{B}_k$ and $\mathcal{B}'_l$, together with the surjectivity assertion in Proposition~\ref{prop:param-AX}(i), show that the bijection $(\Gamma,\Delta)\mapsto(\alpha(\Gamma),\beta(\Delta))$ from $\Fin(\mathcal{H})\times\mathfrak{E}$ onto $\Fin(\mathcal{H}')\times\mathfrak{E}'$ maps the family $\{\mathcal{B}_k:k\in R^0\}$ onto the family $\{\mathcal{B}'_l:l\in S^0\}$. Precomposing it with the bijection $\mathfrak{C}(R)\to\Fin(\mathcal{H})\times\mathfrak{E}$ of Proposition~\ref{prop:param-AX}(i) and postcomposing it with the inverse of the analogous bijection $\mathfrak{C}(S)\to\Fin(\mathcal{H}')\times\mathfrak{E}'$, we obtain a bijection $\Phi\colon\mathfrak{C}(R)\to\mathfrak{C}(S)$ that maps the family $\{\mathcal{O}_R(k):k\in R^0\}$ onto the family $\{\mathcal{O}_S(l):l\in S^0\}$. Moreover, Proposition~\ref{prop:param-AX}(iii) yields $\mu_R(C)=\#A=\#A'=\mu_S(\Phi(C))$ for every $C\in\mathfrak{C}(R)$. The bijection $\Phi$ therefore satisfies conditions~(i) and~(ii) of Theorem~\ref{thm:ufd-general}, which yields $M(R)\cong M(S)$.

Conversely, suppose that $M(R)\cong M(S)$. By Theorem~\ref{thm:ufd-general}, there is a bijection $\Phi\colon\mathfrak{C}(R)\to\mathfrak{C}(S)$ satisfying conditions~(i) and~(ii) of that theorem. By Proposition~\ref{prop:param-AX}(iii), condition~(ii) yields $\#A=\#A'$. Precomposing $\Phi$ with the inverse of the bijection $\mathfrak{C}(R)\to\Fin(\mathcal{H})\times\mathfrak{E}$ of Proposition~\ref{prop:param-AX}(i) and postcomposing it with the analogous bijection $\mathfrak{C}(S)\to\Fin(\mathcal{H}')\times\mathfrak{E}'$, we obtain a bijection $\Psi\,\colon\Fin(\mathcal{H})\times\mathfrak{E}\to\Fin(\mathcal{H}')\times\mathfrak{E}'$. By Proposition~\ref{prop:param-AX}(i), these two bijections map each set $\mathcal{O}_R(k)$ onto $\mathcal{B}_k$ and each set $\mathcal{O}_S(l)$ onto $\mathcal{B}'_l$; hence condition~(i) of Theorem~\ref{thm:ufd-general} implies that $\Psi$ is a homeomorphism for the topologies considered in Lemma~\ref{lem:order-AX}, which, by that lemma, implies that $\#\Pe_{\kappa[x]}=\#\Pe_{\kappa'[x]}$. Consequently, by Lemma~\ref{lem:field-primes}, we obtain $\max\{\#\kappa,\aleph_0\}=\max\{\#\kappa',\aleph_0\}$.
\end{proof}
\section{Additive homeomorphisms between Mac\'ias spaces}\label{sec:additive}

In this section we compare the spaces $\widetilde{M}(R)$ and $M(R)$ and prove Theorem~\ref{thm:additive} and Corollary~\ref{cor:additive-iso}.

The following lemma, which yields the closedness of $\m^0$ in $M(R)$ for every maximal ideal $\m$ of $R$, is used both in the comparison of $\widetilde{M}(R)$ with $M(R)$ (see Lemma~\ref{lem:tilde-top}) and in the proof of Theorem~\ref{thm:additive} (see Lemma~\ref{lem:additive-max}).
\begin{lemma}\label{lem:max-cover}
Let $\m\in\Max R$. Then $R^0\setminus\m=\bigcup\limits_{k\in\m^0}\sigma_k^0$; in particular, $\m^0$ is closed in $M(R)$.
\end{lemma}

\begin{proof}
For $k\in\m^0$, the set $\sigma_k^0$ is disjoint from $\m$, since $\langle y\rangle+\langle k\rangle\subseteq\m$ for every $y\in\m$. Conversely, let $x\in R^0\setminus\m$. Then $\langle x\rangle+\m=R$, so that $1=ax+k$ for some $a\in R$ and $k\in\m$. If $k\neq0$, then $x\in\sigma_k^0$; otherwise, $x$ is a unit and therefore lies in $\sigma_l^0$ for every $l\in\m^0$.
\end{proof}
The next lemma describes the open subsets of $\widetilde{M}(R)$ in terms of those of $M(R)$.
\begin{lemma}\label{lem:tilde-top}
Let $\mathcal{T}(X)$ denote the topology of a space $X$.
\begin{enumerate}[label=\textup{(\alph*)}]
\item The set $J(R)$ consists of the points of $\widetilde{M}(R)$ whose only open neighborhood is $R$.
\item The set $U(R)$ is the smallest non-empty open subset of $\widetilde{M}(R)$.
\item $\mathcal{T}\bigl(\widetilde{M}(R)\bigr)=\{R,U(R)\}\cup\{O\in\mathcal{T}(M(R)):O\cap J(R)=\emptyset\}$.
\end{enumerate}
\end{lemma}

\begin{proof}
We have $\sigma_0=U(R)$, $\sigma_u=R$ for $u\in U(R)$, and $U(R)\subseteq\sigma_k=\sigma_k^0$ for $k\in R^0\setminus U(R)$. Since $R^0\setminus U(R)=\bigcup\limits_{\m\in\Max R}\m^0$, Lemma~\ref{lem:max-cover} yields

\[
\bigcup_{k\in R^0\setminus U(R)}\sigma_k^0= \bigcup_{\substack{\m\in\Max R\\ k\in\m^0}}\sigma_k^0=\bigcup_{\m\in\Max R}(R^0\setminus\m)=R\setminus J(R).
\]

Hence the open subsets of $\widetilde{M}(R)$ other than $R$ are $U(R)$ and the sets $\bigcup\limits_{k\in K}\sigma_k^0$ with $K\subseteq R^0\setminus U(R)$. In particular, $R\setminus J(R)$ is the largest open subset of $\widetilde{M}(R)$ other than $R$, and every non-empty open subset of $\widetilde{M}(R)$ contains $U(R)$; this proves (a) and~(b). Moreover, the open subsets of $M(R)$ are $R^0$ and the sets $\bigcup\limits_{k\in K}\sigma_k^0$ above, and the latter are disjoint from $J(R)$, as they are contained in $R\setminus J(R)$. On the other hand, the set $R^0$ meets $J(R)$ unless $J(R)=0$, in which case $R^0=R\setminus J(R)$ is the set $\bigcup\limits_{k\in K}\sigma_k^0$ with $K=R^0\setminus U(R)$. Hence the open subsets of $M(R)$ disjoint from $J(R)$ are precisely the sets $\bigcup\limits_{k\in K}\sigma_k^0$ with $K\subseteq R^0\setminus U(R)$, and (c) follows. 
\end{proof}
The relationship between $\widetilde{M}(R)$ and $M(R)$ depends on whether $J(R)$ is trivial.
\begin{proposition}\label{prop:tilde}
Let $R$ and $S$ be integral domains that are not fields.
\begin{enumerate}[label=\textup{(\alph*)}]
\item If $J(R)=J(S)=0$, then the homeomorphisms from $\widetilde{M}(R)$ onto $\widetilde{M}(S)$ are precisely the bijections $R\to S$ that map $0$ to $0$ and restrict to homeomorphisms from $M(R)$ onto $M(S)$.
\item If $J(R)\neq0$, then $\widetilde{M}(R)\cong M(R)$.
\item $\widetilde{M}(R)\cong\widetilde{M}(S)$ if and only if $M(R)\cong M(S)$.
\end{enumerate}
\end{proposition}

\begin{proof}
(a) Since $J(R)=J(S)=0$, Lemma~\ref{lem:tilde-top}(c) yields $\mathcal{T}\bigl(\widetilde{M}(R)\bigr)=\{R,U(R)\}\cup\mathcal{T}(M(R))$, where $U(R)\notin\mathcal{T}(M(R))$ by Theorem~\ref{thm:B}(i); the same holds for $S$. By Lemma~\ref{lem:tilde-top}(a) and~(b), every homeomorphism $\widetilde{M}(R)\to\widetilde{M}(S)$ maps $0$ to $0$ and $U(R)$ onto $U(S)$, and therefore restricts to a homeomorphism $M(R)\to M(S)$. Conversely, every bijection $R\to S$ that maps $0$ to $0$ and restricts to a homeomorphism $M(R)\to M(S)$ maps $U(R)$ onto $U(S)$ by Corollary~\ref{cor:units}, and hence is a homeomorphism $\widetilde{M}(R)\to\widetilde{M}(S)$.

(b) Suppose that $J(R)\neq0$, and let $\mathcal{A}$ be the collection of open subsets of $M(R)$ disjoint from $J(R)$. If $O$ is an open subset of $M(R)$ that meets $J(R)$, then $O$ contains a set $\sigma_k^0$ meeting $J(R)$ for some $k\in R^0$. By Lemma~\ref{lem:V}(a) and~(c), $k$ is a unit. Hence $O=\sigma_k^0=R^0$. Therefore $\mathcal{T}(M(R))=\{R^0\}\cup\mathcal{A}$. Since $U(R)\in\mathcal{A}$ by Theorem~\ref{thm:B}(i), it follows from Lemma~\ref{lem:tilde-top}(c) that $\mathcal{T}\bigl(\widetilde{M}(R)\bigr)=\{R\}\cup\mathcal{A}$. Since $R$ is an integral domain, $J(R)$ contains the pairwise distinct powers of any of its non-zero elements and consequently, the ideal $J(R)$ is infinite. Hence $\#J(R)=\#J(R)^0$, and there is a bijection $f\colon R\to R^0$ fixing $R\setminus J(R)$ pointwise. Then $f(R)=R^0$, and $f(O)=O$ for every $O\in\mathcal{A}$, because $O\subseteq R\setminus J(R)$. Thus $f$ maps $\mathcal{T}\bigl(\widetilde{M}(R)\bigr)$ onto $\mathcal{T}(M(R))$, that is, $f$ is a homeomorphism from $\widetilde{M}(R)$ onto $M(R)$.

(c) If $J(R)$ and $J(S)$ are both trivial or both non-trivial, the assertion follows from (a) or (b), respectively. Otherwise, exactly one of $J(R)$ and $J(S)$ is trivial, say $J(R)$; then neither side of the equivalence holds. Indeed, $\widetilde{M}(R)\not\cong\widetilde{M}(S)$ by Lemma~\ref{lem:tilde-top}(a), since a homeomorphism $\widetilde{M}(R)\to\widetilde{M}(S)$ would map $J(R)=\{0\}$ onto $J(S)$, which has more than one element. Moreover, by Theorem~\ref{thm:B}(i), $U(R)$ is not open in $M(R)$, whereas $U(S)$ is open in $M(S)$, so that $M(R)\not\cong M(S)$ by Corollary~\ref{cor:units}.
\end{proof}

We now turn to Theorem~\ref{thm:additive}, whose proof relies on a topological and an algebraic lemma. We begin with the topological one.
\begin{lemma}\label{lem:additive-max}
Let $h\colon R\to S$ be an additive bijection that restricts to a homeomorphism from $M(R)$ onto $M(S)$. Then $\m\mapsto h(\m)$ is a bijection from $\Max R$ onto $\Max S$.
\end{lemma}
\begin{proof}
Let $\m\in\Max R$. We first show that $h(\m)$ is a proper ideal of $S$. The key observation is that every closed subset $C$ of $M(S)$ contains all non-zero multiples of its elements. Indeed, if $y\in C$ and $s\in S^0$, then $V_S(y)\subseteq V_S(ys)$ by Lemma~\ref{lem:V}(b), so that $ys\in\clo{y}\subseteq C$ by Proposition~\ref{prop:closure}. As $\m^0$ is closed in $M(R)$ by Lemma~\ref{lem:max-cover}, the set $h(\m)^0=h(\m^0)$ is closed in $M(S)$. Hence $h(\m)$ contains all multiples of its elements. Since $h$ is additive, $h(\m)$ is also an additive subgroup of $S$, and thus an ideal of $S$. This ideal is proper, because $h$ maps $U(R)$ onto $U(S)$ by Corollary~\ref{cor:units} and $\m\cap U(R)=\emptyset$. 

To see that $h(\m)$ is maximal, let $\mn$ be a maximal ideal of $S$ containing $h(\m)$. Applying the same argument to $h^{-1}$, we see that $h^{-1}(\mn)$ is a proper ideal of $R$ containing $\m$. By the maximality of $\m$, it follows that $h^{-1}(\mn)=\m$, and hence $h(\m)=\mn$. Thus $h(\m)\in\Max S$ for every $\m\in\Max R$; by symmetry, $h^{-1}(\mn)\in\Max R$ for every $\mn\in\Max S$, and the assertion follows.
\end{proof}
The following algebraic lemma concerns prime fields; we include its proof for lack of a suitable reference.
\begin{lemma}\label{lem:prime-field}
Let $F$ be a prime field, $K$ a field, and $\psi\colon F\to K$ an additive bijection with $\psi(1)=1$. Then $\psi$ is a ring isomorphism.
\end{lemma}
\begin{proof}
It suffices to show that $\psi$ is multiplicative. Every $a\in F$ can be written as $a=(m\cdot1_F)(n\cdot1_F)^{-1}$ with $m,n\in\Z$ and $n\cdot1_F\neq0$, because the elements of this form constitute a subfield of $F$ and $F$ is a prime field. Then, by additivity, $\psi(n\cdot1_F)=n\cdot1_K$, which is non-zero as $\psi$ is injective, and $(n\cdot1_K)\,\psi(a)=\psi(n\cdot a)=\psi(m\cdot1_F)=m\cdot1_K$, so that $\psi(a)=(m\cdot1_K)(n\cdot1_K)^{-1}$. Hence $\psi$ is multiplicative.
\end{proof}

\begin{proof}[\normalfont\bfseries Proof of Theorem~\ref{thm:additive}]
Suppose first that $h(x)=u\,\rho(x)$ for every $x\in R$, where $\rho\colon R\to S$ is a ring isomorphism and $u\in U(S)$. Then $h$ is additive, and hence $h(0)=0$, so that $h$ maps $R^0$ onto $S^0$. Moreover, $\rho(\sigma_k^0)=\sigma_{\rho(k)}^0$ for every $k\in R^0$, and $u\,\sigma_l^0=\sigma_l^0$ for every $l\in S^0$, since $\langle uy\rangle=\langle y\rangle$ for $y\in S$. Thus $h(\sigma_k^0)=\sigma_{\rho(k)}^0$ for every $k\in R^0$. Consequently, $h$ maps the basis $\{\sigma_k^0:k\in R^0\}$ of $M(R)$ onto that of $M(S)$, and therefore restricts to a homeomorphism from $M(R)$ onto $M(S)$.

Conversely, suppose that $h$ is additive and restricts to a homeomorphism from $M(R)$ onto $M(S)$. By Corollary~\ref{cor:units}, the element $u:=h(1)$ is a unit of $S$. The map $g\colon R\to S$ defined by $x\mapsto u^{-1}h(x)$ is then an additive bijection with $g(1)=1$, and it restricts to a homeomorphism from $M(R)$ onto $M(S)$, since $u^{-1}\sigma_l^0=\sigma_l^0$ for every $l\in S^0$. Let $\m\in\Max R$. Then $g(\m)\in\Max S$ by Lemma~\ref{lem:additive-max}, and $g$ induces a well-defined additive bijection
\[
\psi_\m\colon R/\m\to S/g(\m),\qquad x+\m\mapsto g(x)+g(\m),
\]
with $\psi_\m(1)=1$. If $R/\m$ is a prime field, then $\psi_\m$ is a ring isomorphism by Lemma~\ref{lem:prime-field}, and therefore $g(xy)-g(x)g(y)\in g(\m)$ for all $x,y\in R$. As $g$ is injective, condition~\ref{cond:P} implies that the ideals $g(\m)$ with $\m\in\Max R$ and $R/\m$ a prime field have trivial intersection. Hence $g(xy)=g(x)g(y)$ for all $x,y\in R$, so that $\rho:=g$ is a ring isomorphism, and the assertion follows.
\end{proof}
\begin{proof}[\normalfont\bfseries Proof of Corollary~\ref{cor:additive-iso}]
Every ring isomorphism $\rho\colon R\to S$ is an additive bijection that restricts to a homeomorphism from $M(R)$ onto $M(S)$, by Theorem~\ref{thm:additive} applied with $h=\rho$ and $u=1$. Conversely, if such a bijection exists, then $R\cong S$ by Theorem~\ref{thm:additive}.
\end{proof}

\begin{remark}\label{rem:additive-sharp}
(a) Additivity cannot be replaced by multiplicativity in Theorem~\ref{thm:additive} and Corollary~\ref{cor:additive-iso}. Let $\gamma$ be a bijection from the set of prime numbers onto the set of monic irreducible polynomials over $\F_3$, both sets being countably infinite, and let $\nu\colon\Z\to\F_3[x]$ be the multiplicative map with $\nu(0)=0$, $\nu(-1)=-1$ and $\nu(p)=\gamma(p)$ for every prime $p$. As $\Z$ and $\F_3[x]$ are UFDs with unit group $\{\pm1\}$, the map $\nu$ is an isomorphism of multiplicative monoids. Moreover, the bijection $\varphi\colon\Max\Z\to\Max\F_3[x]$, $p\Z\mapsto\langle\gamma(p)\rangle$, satisfies $V_{\F_3[x]}(\nu(y))=\varphi(V_\Z(y))$ for $y\in\Z^0$, so that $\nu$ restricts to a homeomorphism $M(\Z)\to M(\F_3[x])$ by Theorem~\ref{cor:dedekind-homeo}. However, $\Z\not\cong\F_3[x]$, although $\Z$ satisfies condition~\ref{cond:P}.

(b) Condition~\ref{cond:P} cannot be omitted in Theorem~\ref{thm:additive}, even in characteristic~$0$. Consider the localization $R=\Q[x]_{\langle x\rangle}$ of $\Q[x]$ at the prime ideal $\langle x\rangle$. It is a local ring with maximal ideal $\m=xR$ \cite[Chapter~3, Example~1]{AtiyahMacdonald}. Although $R/\m\cong\Q$ is a prime field, $R$ violates condition~\ref{cond:P}. By Lemma~\ref{lem:sigmaJ}, $\mathcal{T}(M(R))=\{\emptyset,U(R),R^0\}$. Let $g\colon R\to R$ be the additive map $f\mapsto f+f'(0)\,x^2$. Then $g$ is bijective, with inverse $f\mapsto f-f'(0)\,x^2$, and $g(f)(0)=f(0)$ for every $f\in R$; hence $g(U(R))=U(R)$, and $g$ restricts to a self-homeomorphism of $M(R)$. However, $g(x)^2=(x+x^2)^2\neq x^2=g(x^2)$, so that $g$ is not multiplicative; since $g(1)=1$, it follows that $g$ is not a unit multiple of a ring automorphism of $R$.
\end{remark}
\section{An open problem}\label{sec:conclusion}
By Theorem~\ref{thm:ufd-iff}, a positive answer to the following problem would imply that the converse of Corollary~\ref{cor:comax} holds for UFDs satisfying condition~\ref{cond:H}. This condition is assumed in the problem because, without it, two UFDs $R$ and $S$ with homeomorphic Mac\'ias spaces need not satisfy $\#\Pe_R=\#\Pe_S$ (see Remark~\ref{rem:J-ufd-sharp}); hence neither Theorem~\ref{thm:ufd-iff} nor the converse of Corollary~\ref{cor:comax} extends to arbitrary UFDs.

\begin{problem}\label{prob:ufd}
Let $R$ and $S$ be UFDs satisfying condition~\ref{cond:H}. If there is a bijection $\Pe_R\to\Pe_S$ that preserves comaximally open sets, is there one that preserves comaximality?
\end{problem}

\begin{acknow}
The first author gratefully acknowledges financial support from the Prime Minister's Research Fellowship, Government of India (PMRF/2502403). The second author was supported by the Centre for Operator Algebras, Geometry, Matter and Spacetime, Ministry of Education, Government of India, through the Indian Institute of Technology Madras (Project no.\ SB22231267MAETWO008573). We sincerely thank the anonymous referee for valuable comments and suggestions, which have significantly improved the quality of the paper.
\end{acknow}
\bibliographystyle{alpha} 
\bibliography{references}

@article{Furstenberg,
AUTHOR = {Furstenberg, H.},
     TITLE = {On the infinitude of primes},
   JOURNAL = {Amer. Math. Monthly},
  FJOURNAL = {American Mathematical Monthly},
    VOLUME = {62},
      YEAR = {1955},
     PAGES = {353},
      ISSN = {0002-9890,1930-0972},
   MRCLASS = {10.0X},
  MRNUMBER = {68566},
       DOI = {10.2307/2307043},
       URL = {https://doi.org/10.2307/2307043},
       NOTE={\url{https://doi.org/10.2307/2307043}},
}

@article{Brown,
  AUTHOR = {Brown, M.},
  TITLE = {A countable connected {H}ausdorff space},
  JOURNAL = {Bull. Amer. Math. Soc.},
  VOLUME = {59},
  YEAR = {1953},
  PAGES = {367},
}

@article{Golomb,
 AUTHOR = {Golomb, S. W.},
     TITLE = {A connected topology for the integers},
   JOURNAL = {Amer. Math. Monthly},
  FJOURNAL = {American Mathematical Monthly},
    VOLUME = {66},
      YEAR = {1959},
     PAGES = {663--665},
      ISSN = {0002-9890,1930-0972},
   MRCLASS = {10.00 (54.00)},
  MRNUMBER = {107622},
MRREVIEWER = {N.\ G.\ de Bruijn},
       DOI = {10.2307/2309340},
       URL = {https://doi.org/10.2307/2309340},
       note={\url{https://doi.org/10.2307/2309340}},
       
}

@article{KnopfmacherPorubsky,
AUTHOR = {Knopfmacher, J. and Porubsky, S.},
     TITLE = {Topologies related to arithmetical properties of integral
              domains},
   JOURNAL = {Exposition. Math.},
  FJOURNAL = {Expositiones Mathematicae. International Journal},
    VOLUME = {15},
      YEAR = {1997},
    NUMBER = {2},
     PAGES = {131--148},
      ISSN = {0723-0869},
   MRCLASS = {13J99 (13F99)},
  MRNUMBER = {1458761},
MRREVIEWER = {J.\ W.\ Brewer},
}

@article{Clark,
 AUTHOR = {Clark, P. L.},
     TITLE = {The {E}uclidean criterion for irreducibles},
   JOURNAL = {Amer. Math. Monthly},
  FJOURNAL = {American Mathematical Monthly},
    VOLUME = {124},
      YEAR = {2017},
    NUMBER = {3},
     PAGES = {198--216},
      ISSN = {0002-9890,1930-0972},
   MRCLASS = {13F15 (13G05)},
  MRNUMBER = {3626242},
MRREVIEWER = {Dario\ Spirito},
       DOI = {10.4169/amer.math.monthly.124.3.198},
       URL = {https://doi.org/10.4169/amer.math.monthly.124.3.198},
       NOTE={\url{https://doi.org/10.4169/amer.math.monthly.124.3.198}},
}

@article{CLLP,
AUTHOR = {Clark, P. L. and Lebowitz-Lockard, N. and Pollack, P.},
     TITLE = {A note on {G}olomb topologies},
   JOURNAL = {Quaest. Math.},
  FJOURNAL = {Quaestiones Mathematicae. Journal of the South African
              Mathematical Society},
    VOLUME = {42},
      YEAR = {2019},
    NUMBER = {1},
     PAGES = {73--86},
      ISSN = {1607-3606,1727-933X},
   MRCLASS = {13G05 (54G99)},
  MRNUMBER = {3905659},
MRREVIEWER = {Dario\ Spirito},
       DOI = {10.2989/16073606.2018.1438533},
       URL = {https://doi.org/10.2989/16073606.2018.1438533},
       NOTE={\url{https://doi.org/10.2989/16073606.2018.1438533}},
}

@article{BST,
AUTHOR = {Banakh, T. and Spirito, D. and Turek, S.},
     TITLE = {The {G}olomb space is topologically rigid},
   JOURNAL = {Comment. Math. Univ. Carolin.},
  FJOURNAL = {Commentationes Mathematicae Universitatis Carolinae},
    VOLUME = {62},
      YEAR = {2021},
    NUMBER = {3},
     PAGES = {347--360},
      ISSN = {0010-2628,1213-7243},
   MRCLASS = {54G15 (11A07)},
  MRNUMBER = {4331287},
MRREVIEWER = {Othman\ Echi},
}

@article{Spirito,
  AUTHOR = {Spirito, D.},
     TITLE = {The {G}olomb topology on a {D}edekind domain and the group of
              units of its quotients},
   JOURNAL = {Topology Appl.},
  FJOURNAL = {Topology and its Applications},
    VOLUME = {273},
      YEAR = {2020},
     PAGES = {107101, 20},
      ISSN = {0166-8641,1879-3207},
   MRCLASS = {54G99 (11A07 13F05 54A10)},
  MRNUMBER = {4062996},
MRREVIEWER = {Mehdi\ Parsinia},
       DOI = {10.1016/j.topol.2020.107101},
       URL = {https://doi.org/10.1016/j.topol.2020.107101},
       NOTE={\url{https://doi.org/10.1016/j.topol.2020.107101}},
}

@article{MaciasIntegers,
  AUTHOR = {Mac\'ias, J.},
     TITLE = {Another topological proof of the infinitude of prime numbers},
   JOURNAL = {Integers},
  FJOURNAL = {Integers. Electronic Journal of Combinatorial Number Theory},
    VOLUME = {24},
      YEAR = {2024},
     PAGES = {Paper No. A47, 4},
      ISSN = {1553-1732},
   MRCLASS = {11A41 (11B05)},
  MRNUMBER = {4747711},
  DOI= {10.5281/zenodo.11221653},
  note={\url{https://doi.org/10.5281/zenodo.11221653}},
}

@article{MaciasDomains,
   AUTHOR = {Mac\'ias, J.},
     TITLE = {The {M}ac\'ias topology on integral domains},
   JOURNAL = {Topology Appl.},
  FJOURNAL = {Topology and its Applications},
    VOLUME = {357},
      YEAR = {2024},
     PAGES = {Paper No. 109070, 9},
      ISSN = {0166-8641,1879-3207},
   MRCLASS = {54A05 (54G05 54H11)},
  MRNUMBER = {4797249},
       DOI = {10.1016/j.topol.2024.109070},
       URL = {https://doi.org/10.1016/j.topol.2024.109070},
        note={\url{https://doi.org/10.1016/j.topol.2024.109070}}
       
}

@misc{MaciasOrtiz,
  AUTHOR = {Mac\'ias, J. and Ortiz, R.},
  TITLE = {A note on the {M}ac\'ias topology},
  YEAR = {2025},
  EPRINT = {2411.06670},
  ARCHIVEPREFIX = {arXiv},
  PRIMARYCLASS = {math.AT},
  NOTE = {Preprint, arXiv:2411.06670v2},
}

@article{MaciasSelf,
   AUTHOR = {Mac\'ias, J.},
     TITLE = {On self-homeomorphisms of the {M}ac\'ias space},
   JOURNAL = {Topology Appl.},
  FJOURNAL = {Topology and its Applications},
    VOLUME = {373},
      YEAR = {2025},
     PAGES = {Paper No. 109520, 9},
      ISSN = {0166-8641,1879-3207},
   MRCLASS = {54D05 (11A41 11B05)},
  MRNUMBER = {4931338},
MRREVIEWER = {Ajoy\ Mukharjee},
       DOI = {10.1016/j.topol.2025.109520},
       URL = {https://doi.org/10.1016/j.topol.2025.109520},
       NOTE={\url{https://doi.org/10.1016/j.topol.2025.109520}},
}

@misc{Mandal,
  AUTHOR = {Mandal, S.},
  TITLE = {A coprimality topology on the {G}aussian integers: {K}olmogorov quotient and {G}aussian prime density},
  YEAR = {2026},
  NOTE = {Preprint, arXiv:2608.06373},
}

@book{AtiyahMacdonald,
   AUTHOR = {Atiyah, M. F. and Macdonald, I. G.},
     TITLE = {Introduction to commutative algebra},
 PUBLISHER = {Addison-Wesley Publishing Co., Reading, Mass.-London-Don
              Mills, Ont.},
      YEAR = {1969},
     PAGES = {ix+128},
   MRCLASS = {13.00},
  MRNUMBER = {242802},
MRREVIEWER = {Johnny\ A.\ Johnson},
}

@book{Munkres,
   AUTHOR = {Munkres, J. R.},
     TITLE = {Topology},
   EDITION = {Second},
 PUBLISHER = {Prentice Hall, Inc., Upper Saddle River, NJ},
      YEAR = {2000},
     PAGES = {xvi+537},
      ISBN = {0-13-181629-2},
   MRCLASS = {54-01},
  MRNUMBER = {3728284},
}

@book{Eisenbud,
  AUTHOR = {Eisenbud, D.},
     TITLE = {Commutative algebra},
    SERIES = {Graduate Texts in Mathematics},
    VOLUME = {150},
 PUBLISHER = {Springer-Verlag, New York},
      YEAR = {1995},
     PAGES = {xvi+785},
      ISBN = {0-387-94268-8; 0-387-94269-6},
   MRCLASS = {13-01 (14A05)},
  MRNUMBER = {1322960},
MRREVIEWER = {Matthew\ Miller},
       DOI = {10.1007/978-1-4612-5350-1},
       URL = {https://doi.org/10.1007/978-1-4612-5350-1},
       NOTE={\url{https://doi.org/10.1007/978-1-4612-5350-1}},
}

@book{Matsumura,
   AUTHOR = {Matsumura, H.},
     TITLE = {Commutative ring theory},
    SERIES = {Cambridge Studies in Advanced Mathematics},
    VOLUME = {8},
      NOTE = {Translated from the Japanese by M. Reid},
 PUBLISHER = {Cambridge University Press, Cambridge},
      YEAR = {1986},
     PAGES = {xiv+320},
      ISBN = {0-521-25916-9},
   MRCLASS = {13-01},
  MRNUMBER = {879273},
MRREVIEWER = {W.\ V.\ Vasconcelos},
}

@book{GoubaultLarrecq,
  
AUTHOR = {Goubault-Larrecq, J.},
     TITLE = {Non-{H}ausdorff topology and domain theory},
    SERIES = {New Mathematical Monographs},
    VOLUME = {22},
 PUBLISHER = {Cambridge University Press, Cambridge},
      YEAR = {2013},
     PAGES = {vi+491},
      ISBN = {978-1-107-03413-6},
   MRCLASS = {54-01 (54D10 54H10 54H99)},
  MRNUMBER = {3086734},
MRREVIEWER = {Alexander\ Yurievich\ Shibakov},
       DOI = {10.1017/CBO9781139524438},
       URL = {https://doi.org/10.1017/CBO9781139524438},
       NOTE={\url{https://doi.org/10.1017/CBO9781139524438}},
}

@book {diri,
    AUTHOR = {Narkiewicz, W.},
     TITLE = {The development of prime number theory},
    SERIES = {Springer Monographs in Mathematics},
 PUBLISHER = {Springer-Verlag, Berlin},
      YEAR = {2000},
     PAGES = {xii+448},
      ISBN = {3-540-66289-8},
   MRCLASS = {11N05 (01A55 11-03)},
  MRNUMBER = {1756780},
MRREVIEWER = {D.\ R.\ Heath-Brown},
       DOI = {10.1007/978-3-662-13157-2},
       URL = {https://doi.org/10.1007/978-3-662-13157-2},
       NOTE={\url{https://doi.org/10.1007/978-3-662-13157-2}},
}

@article{ClarkElliptic,
   AUTHOR = {Clark, P. L.},
     TITLE = {Elliptic {D}edekind domains revisited},
   JOURNAL = {Enseign. Math. (2)},
  FJOURNAL = {L'Enseignement Math\'ematique. Revue Internationale. 2e
              S\'erie},
    VOLUME = {55},
      YEAR = {2009},
    NUMBER = {3-4},
     PAGES = {213--225},
      ISSN = {0013-8584},
   MRCLASS = {13C20 (13F05)},
  MRNUMBER = {2583777},
MRREVIEWER = {David\ F.\ Anderson},
       DOI = {10.4171/LEM/55-3-1},
       URL = {https://doi.org/10.4171/LEM/55-3-1},
       NOTE={\url{https://doi.org/10.4171/LEM/55-3-1}},
}

@book{Silverman,
AUTHOR = {Silverman, J. H.},
     TITLE = {The arithmetic of elliptic curves},
    SERIES = {Graduate Texts in Mathematics},
    VOLUME = {106},
   EDITION = {Second},
 PUBLISHER = {Springer, Dordrecht},
      YEAR = {2009},
     PAGES = {xx+513},
      ISBN = {978-0-387-09493-9},
   MRCLASS = {11-02 (11G05 11G20 14H52 14K15)},
  MRNUMBER = {2514094},
MRREVIEWER = {Vasil\cprime\ \=I.\ Andr\=\i\u ichuk},
       DOI = {10.1007/978-0-387-09494-6},
       URL = {https://doi.org/10.1007/978-0-387-09494-6},
       NOTE={\url{https://doi.org/10.1007/978-0-387-09494-6}},
}

@book{Jech,
  AUTHOR = {Jech, T.},
     TITLE = {Set theory},
    SERIES = {Springer Monographs in Mathematics},
   EDITION = {millennium},
 PUBLISHER = {Springer-Verlag, Berlin},
      YEAR = {2003},
     PAGES = {xiv+769},
      ISBN = {3-540-44085-2},
   MRCLASS = {03Exx (03-01 03-02)},
  MRNUMBER = {1940513},
MRREVIEWER = {Eva\ Coplakova},
}

@book{HardyWright,
    AUTHOR = {Hardy, G. H. and Wright, E. M.},
     TITLE = {An introduction to the theory of numbers},
   EDITION = {Fifth},
 PUBLISHER = {The Clarendon Press, Oxford University Press, New York},
      YEAR = {1979},
     PAGES = {xvi+426},
      ISBN = {0-19-853170-2},
   MRCLASS = {10-01},
  MRNUMBER = {568909},
MRREVIEWER = {T.\ M.\ Apostol},
}

@article{BodinDebesNajib,
  AUTHOR = {Bodin, A. and D\`ebes, P. and Najib, S.},
     TITLE = {The {S}chinzel hypothesis for polynomials},
   JOURNAL = {Trans. Amer. Math. Soc.},
  FJOURNAL = {Transactions of the American Mathematical Society},
    VOLUME = {373},
      YEAR = {2020},
    NUMBER = {12},
     PAGES = {8339--8364},
      ISSN = {0002-9947,1088-6850},
   MRCLASS = {11N32 (12E05 12E25 12E30 13F20)},
  MRNUMBER = {4177261},
MRREVIEWER = {Mario\ Pineda-Ruelas},
       DOI = {10.1090/tran/8198},
       URL = {https://doi.org/10.1090/tran/8198},
       NOTE={\url{https://doi.org/10.1090/tran/8198}}
}

@book{DummitFoote,
  AUTHOR = {Dummit, D. S. and Foote, R. M.},
     TITLE = {Abstract algebra},
   EDITION = {Third},
 PUBLISHER = {John Wiley \& Sons, Inc., Hoboken, NJ},
      YEAR = {2004},
     PAGES = {xii+932},
      ISBN = {0-471-43334-9},
   MRCLASS = {00-01 (16-01 20-01)},
  MRNUMBER = {2286236},
}

@book{Diestel,
  AUTHOR = {Diestel, R.},
     TITLE = {Graph theory},
    SERIES = {Graduate Texts in Mathematics},
    VOLUME = {173},
   EDITION = {Sixth},
 PUBLISHER = {Springer, Berlin},
      YEAR = {2025},
     PAGES = {xx+454},
      ISBN = {978-3-662-70106-5; 978-3-662-70107-2},
   MRCLASS = {05-01 (05Cxx)},
  MRNUMBER = {4874150},
  }

@article{SharmaBhatwadekar,
  AUTHOR = {Sharma, P. K. and Bhatwadekar, S. M.},
     TITLE = {A note on graphical representation of rings},
   JOURNAL = {J. Algebra},
  FJOURNAL = {Journal of Algebra},
    VOLUME = {176},
      YEAR = {1995},
    NUMBER = {1},
     PAGES = {124--127},
      ISSN = {0021-8693,1090-266X},
   MRCLASS = {05C15 (13A99)},
  MRNUMBER = {1345297},
MRREVIEWER = {Xiaodong\ Liu},
       DOI = {10.1006/jabr.1995.1236},
       URL = {https://doi.org/10.1006/jabr.1995.1236},
       NOTE={\url{https://doi.org/10.1006/jabr.1995.1236}},
}

@article{MaimaniEtAl,
  AUTHOR = {Maimani, H. R. and Salimi, M. and Sattari, A. and
              Yassemi, S.},
     TITLE = {Comaximal graph of commutative rings},
   JOURNAL = {J. Algebra},
  FJOURNAL = {Journal of Algebra},
    VOLUME = {319},
      YEAR = {2008},
    NUMBER = {4},
     PAGES = {1801--1808},
      ISSN = {0021-8693,1090-266X},
   MRCLASS = {13A15 (05C99)},
  MRNUMBER = {2383067},
MRREVIEWER = {David\ F.\ Anderson},
       DOI = {10.1016/j.jalgebra.2007.02.003},
       URL = {https://doi.org/10.1016/j.jalgebra.2007.02.003},
       NOTE={\url{https://doi.org/10.1016/j.jalgebra.2007.02.003}},

}

@book{Gopalakrishnan,
  AUTHOR = {Gopalakrishnan, N. S.},
  TITLE = {Commutative Algebra},
  PUBLISHER = {Oxonian Press},
  ADDRESS = {New Delhi},
  YEAR = {1984},
}
\end{document}